\documentclass[11pt]{amsart}
\usepackage{pb-diagram,amssymb,epic,eepic,verbatim,graphicx,graphics,epsfig,psfrag,paralist}
\usepackage{color}
\newtheorem{theorem}{Theorem}[section]

\newtheorem{assumption}[theorem]{Assumption}
\newtheorem{corollary}[theorem]{Corollary}

\newtheorem{construction}[theorem]{Construction}

\newtheorem{proposition}[theorem]{Proposition}
\newtheorem{lemma}[theorem]{Lemma}

\theoremstyle{definition}
\newtheorem{definition}[theorem]{Definition}
\theoremstyle{remark}
\newtheorem{remark}[theorem]{Remark}
\newtheorem{notation}[theorem]{Notation}
\newtheorem*{claim}{Claim}

\newcounter{notes}
\newenvironment{Notes}
{\begin{list}{%{\textsc{Note}}
\arabic{notes}. }{\usecounter{notes}%
\setlength{\labelsep}{0pt}%
\setlength{\itemsep}{5pt}%
\setlength{\leftmargin}{0pt}%
\setlength{\labelwidth}{0pt}%
\setlength{\listparindent}{0pt}}}%
{\end{list}}

\makeatletter
\@addtoreset{theorem}{section}
\makeatother
\newcommand\A{\mathcal{A}}

\newcommand\mN{\mathcal{N}}
\newcommand\mC{\mathcal{C}}
\newcommand\mP{\mathcal{P}}

\newcommand{\K}{\mathcal{K}}

\renewcommand{\L}{\mathcal{L}}

\newcommand{\F}{\mathcal{F}}
\newcommand{\N}{\mathbb{N}}
\newcommand{\R}{\mathbb{R}}

\newcommand{\C}{\mathbb{C}}

\newcommand{\Z}{\mathbb{Z}}

\newcommand\G{\mathcal{G}}
\renewcommand\H{\mathcal{H}}

\newcommand\ulg{{\on g}}

\newcommand{\ddt}{\frac{d}{dt}}

\renewcommand{\P}{\mathbb{P}}

\newcommand{\Def}{ \on{Def} } 
\newcommand{\Spec}{ \on{Spec} } 
\newcommand{\oSig}{\overline \Sigma}
\newcommand{\tSig}{\tilde \Sigma}

\newcommand{\Edge}{\on{Edge}}
\newcommand{\vertex}{\on{Vert}}
\newcommand{\Cyl}{\on{Cyl}}
\newcommand{\inc}{\on{\iota}}
\newcommand{\ev}{\on{ev}}
\newcommand{\ft}{\on{ft}}
\newcommand{\st}{\on{st}}
\newcommand{\img}{i}
\newcommand\lie[1]{\mathfrak{#1}}
\renewcommand{\k}{\lie{k}}

\newcommand{\g}{\lie{g}}

\newcommand{\on}{\operatorname}
\newcommand{\Aut}{ \on{Aut} }

\newcommand{\Ad}{ \on{Ad} } 
 
\newcommand{\End}{\on{End}}

\newcommand\Hol{\on{Hol}}

\newcommand\Id{\on{Id}}
\newcommand\D{D}
\newcommand\bd{\on{bd}}

\newcommand{\hra}{\hookrightarrow}

\newcommand{\tensor}{\otimes}

\renewcommand{\d}{{\on{d}}}

\newcommand{\ol}{\overline}
\newcommand\ul{\underline}
\newcommand{\delbar}{\ol{\partial}}
\newcommand\bs{\backslash}

\newcommand\Sig{\Sigma}
\newcommand\sig{\sigma}
\newcommand\eps{\epsilon}
\newcommand\Om{\Omega}
\newcommand\om{\omega}

\newcommand{\lan}{\langle}
\newcommand{\ran}{\rangle}
\newcommand{\hh}{{\frac{1}{2}}}

\newcommand{\qq}{{\frac{1}{4}}}

\renewcommand{\ss}{{\operatorname{ss}}}
\newcommand{\s}{{\operatorname{s}}}
\newcommand{\dual}{{\raisebox{0.5ex}{\scalebox{0.5}[0.25]{$\vee$}}}}
\newcommand\Mod[1]{\lVert #1 \rVert}
\newcommand\qu{/\kern-.7ex/} % Categorical quotients

\begin{document}
\author{Sushmita Venugopalan} \address{ Chennai Mathematical
  Institute, H1 Sipcot IT Park, Siruseri, Chennai, India.  }
 
\title{Vortices on surfaces with cylindrical ends}
\begin{abstract}
  We consider Riemann surfaces obtained from nodal curves with
  infinite cylinders in the place of nodal and marked points, and
  study the space of finite energy vortices defined on these
  surfaces. To compactify the space of vortices, we need to consider
  stable vortices -- these incorporate breaking of cylinders and
  sphere bubbling in the fibers. In this paper, we prove that the
  space of gauge equivalence classes of stable vortices representing a
  fixed equivariant homology class is compact and Hausdorff under the
  Gromov topology. We also show that this space is homeomorphic to the
  moduli space of quasimaps defined by Ciocan-Fontanine, Kim and
  Maulik in \cite{CKM:quasimap}.
\end{abstract}

\maketitle
\tableofcontents

\parskip .05in The moduli space of stable quasimaps described by
Ciocan-Fontanine, Kim and Maulik in \cite{CKM:quasimap} is a
compactification of the space of maps from non-singular curves to
targets that are Geometric Invariant Theory (GIT) quotients. The
domain curves have genus $g$ and $n$ marked points and the complex
structure is allowed to vary. The points in the boundary of the
quasimap moduli space exhibit two phenomena: first, maps may acquire a
finite number of base points, and second, the domain curve may
degenerate to a nodal curve. Base points are required to be away from
marked points and nodal points. This paper provides a symplectic
version of stable quasimaps.

The quasimap compactification is different from Kontsevich's stable
map compactification, which works for general targets. The boundary
points of the Kontsevich moduli space are maps whose domains are nodal
curves. In contrast, quasimaps with base points are no longer honest
maps from a nodal curve to the target. We explain this point. Let $G$
be a connected complex reductive group and $X$ be a polarized affine
$G$-variety. The GIT quotient $X \qu G$ is an open substack of the quotient
stack $X/G$. A quasimap $u$ is a map from a curve $C$ to the stack
$X/G$. In the complement of the set of base points in $C$, the map $u$
has image in $X \qu G$. Maps from a curve $C$ to the quotient stack
correspond to $G$-maps from a principal $G$-bundle on $C$ to $X$. The
GIT quotient is the stack-theoretic quotient of the semistable locus $X^\ss$. Base points are thus the points on $C$ that map to $X \bs
X^\ss$.

Suppose $K$ is a maximal compact subgroup of $G$. In symplectic
geometry, the maps $C \to X/G$ are analogous to {\em $K$-vortices}
from $C$ to $X$. When $X$ is as above, the action of $K$ on $X$ is
Hamiltonian and has a moment map $\Phi:X \to \k^\dual$. By choosing an
$\Ad_K$-invariant metric on $\k$, we assume there is an identification
$\k^\dual \simeq \k$. A {\em vortex} $(A,u)$ consists of a connection
$A$ on a principal $K$-bundle $P \to C$ and a holomorphic section $u:C
\to P \times_K X$ with respect to $\delbar_A$ that satisfies the
equation
$$F_A + \Phi(u)vol_C=0.$$
This equation requires a choice of area form on $C$. To make sure that
base points are away from special points on $C$, we `blow up' the area
form at the special points. Punctured neighborhoods of these points
will be isometric to semi-infinite cylinders.  The blowing up of the
area form near the special points ensures that the limit of $u$, as we
approach the special points, lies in $\Phi^{-1}(0)$ and hence is in
the semistable locus $X^\ss$. We define a smooth family of metrics, called
the {\em neck-stretching metrics} on stable nodal genus $g$,
$n$-pointed curves such that the metric blows up at special points in
the above-mentioned way, so any such curve now corresponds to a
Riemann surface with cylindrical ends. The space of vortices
representing a given equivariant homology class in $H_2^K(X)$ defined
on stable nodal curves with neck-stretching metric is not compact. To
compactify it, we allow breaking of cylinders (as in Floer theory) and
sphere bubbles in $X$-fibers (as in the work of Ott \cite{Ott}).  The
resulting objects are called stable vortices. The space of stable
vortices in a fixed class of $H_2^K(X)$ modulo $K$-gauge
transformations is compact and Hausdorff under the Gromov topology.

Suppose $\ol M_{g,n}$ is the coarse moduli space of stable nodal
curves of genus $g$ with $n$ marked points. We assume $n \geq 1$, and
for stability $n+2g-3 \geq 0$. If a vortex $(A,u)$ has finite energy
and bounded image in $X$, a removal of singularity result applies at
the cylindrical ends, which means that $u$ is well-defined over a
closed complex curve. Then, $u$ represents a class in $H_2^K(X)$. Let
$MV^K_{g,n}(X,\beta)$ be the space of vortices $(A,u)$ on stable genus
$g$, $n$-pointed curves equipped with the neck-stretching metric, such
that $[u]=\beta \in H_2^K(X)$ modulo the group of (unitary) gauge
transformations. Removal of singularity at the cylindrical ends
ensures that the evaluation maps
$$\ev_j:MV^K_{g,n}(X) \to X\qu G \quad (A,u)\mapsto \lim_{z \to z_j}Ku(z_j)$$
are well-defined for marked points $z_1,\dots,z_n$. By the definition
of $MV^K_{g,n}(X)$, there is a forgetful map
$$\ft:MV^K_{g,n}(X) \to \ol M_{g,n}.$$
In the compactification of $MV^K_{g,n}(X)$, the domain may not be
stable, $\ft$ is defined as the stabilization of the domain, achieved
by contracting unstable components. Our first result is:
\begin{theorem}\label{thm:maincpt} Suppose $(X,\om,K,\Phi)$ is a
  $K$-Hamiltonian symplectic manifold that is equivariantly convex at
  $\infty$ (see Definition \ref{def:conv}), has a proper moment map
  $\Phi$ and has free action of $K$ on $\Phi^{-1}(0)$.
 The compactification of $MV^K_{g,n}(X,\beta)$, called $\ol
  {MV}^K_{g,n}(X,\beta)$, is a compact Hausdorff space under the
  Gromov topology. The forgetful map $\ft$ and the evaluation maps
  $\ev_j$ are well-defined and continuous on
  $\ol{MV}^K_{g,n}(X,\beta)$.
\end{theorem}
Compactifications of the space of symplectic vortices have been
constructed by \cite{MundetTian:cpt}, \cite{Zilt:QK} and
\cite{Ott}. Ziltener (\cite{Zilt:QK}) constructs a compactification of
the space of vortices on the complex plane $\C$. In this case, besides
sphere bubbling in the fibers, there is bubbling at infinity that
produces sphere bubbles in the quotient $X \qu G$ and vortices on $\C$
attached to these bubbles. Our situation for vortices on surfaces with
cylindrical ends is simpler in comparison.

Mundet-Tian (\cite{MundetTian:cpt}) have constructed a
compactification for vortices with varying domain curve, equipped with
a finite volume metric. In this case, when the domain curve
degenerates to a nodal curve, the map $u$ can degenerate to a chain of
gradient flow lines of the moment map $\Phi$ ($K=S^1$, so $i\Phi$ maps
to $\R$). In our approach, by allowing infinite volume at nodal points
we avoid these structures. The infinite volume also ensures that nodal
points map to the semistable locus, where the group action is free,
which helps us avoid meromorphic connections present in
\cite{MundetTian:cpt}. Further, the behavior of vortices away from
base points is similar to that of holomorphic curves on GIT quotient
$X \qu G$. This phenomenon ties in with the philosophy of gauged
Gromov-Witten theory where the moduli spaces of vortices with target
$X$ is related to the moduli space of holomorphic curves on the
quotient $X \qu G$ by wall-crossing as in Toda \cite{Toda} and
Woodward (\cite{W:qkirwan1}, \cite{W:qkirwan2}, \cite{W:qkirwan3}).

Let $Qmap_{g,n}(X \qu G, \beta)$ be the space of stable quasimaps
whose domains are nodal $n$-pointed curves of genus $g$ and which
represent the homology class $\beta \in H_2^G(X)$. Assuming that the
$G$-action on the semistable locus $X^\ss$ is free and that $X$ is an
affine variety, the paper \cite{CKM:quasimap} proves that the moduli
space of quasimaps $Qmap_{g,n}(X \qu G, \beta)$ is a Deligne-Mumford
stack that is proper over the affine quotient. The next result of this
paper is that $Qmap_{g,n}(X \qu G, \beta)$ is homeomorphic to the
space of stable vortices that are in the equivariant homology class
$\beta$.  We state the following theorem for the case that the affine
quotient is a single point, or in other words, the GIT quotient is
projective. In that case, $X$ can be realized as a K\"ahler
Hamiltonian $K$-manifold with a proper moment map that is
equivariantly convex at infinity.
\begin{theorem} \label{thm:main} Let $g$, $n$, $K$, $G$, $X$ be as
  above and $\beta \in H_2^K(X)$. Suppose the GIT quotient $X \qu G$
  is projective. There is a homeomorphism
$$\Psi:Qmap_{g,n}(X \qu G, \beta) \to \ol {MV}^K_{g,n}(X,\beta)$$ 
that commutes with the evaluation maps $\ev_j$ to the quotient $X \qu
G$ and the forgetful map $\ft$ to $\ol M_{g,n}$.
\end{theorem}
The map $\Psi$ in the above Theorem is
still a homeomorphism in the case of a general affine quotient, i.e. when $X\qu G$ is not projective -- this
generalization is discussed in Section \ref{sec:noncpt}. The proof of the Theorem \ref{thm:main} does not require that $X$ is affine. It only
requires that $Qmap_{g,n}(X \qu G, \beta)$ is compact. Therefore, we
expect $Qmap_{g,n}(X \qu G, \beta)$ to be compact in more general
situations, provided $X$ is aspherical. In order to remove the
asphericity assumption, the definition of quasimaps will have to be
broadened to include sphere bubbles in the fiber. 

In the proof of Theorem \ref{thm:main}, the bijection $\Psi$ is a
Hitchin-Kobayashi correspondence established in \cite{VW:affine}. The
notion of stability required by this correspondence -- of the point at
infinity (marked points and nodal points in the case of $Qmap$)
mapping to the semistable locus $X^\ss$ -- is part of the definition of
quasimaps. The proof of continuity of $\Psi$ relies on the convexity
of the moment map -- but there are analytic difficulties arising
because of the non-compact domains. To overcome these, we crucially
rely on a stronger version of removal of singularity at infinity for
vortices. The original such result proved by Ziltener in his thesis
(\cite{Zilt:thesis}) for the affine case gives only $L^p$ control on
the decay of the connection. But in the cylindrical case, we are able
to show a similar result (Proposition \ref{prop:connW1p}) giving
$W^{1,p}$ control.

The paper is organized as follows. Section 1 constructs a smooth
family of neck-stretching metrics on stable nodal curves parametrized
by $\ol M_{g,n}$. Sections 2-5 describe vortices and prove Theorem
0.1, the proof appears in Section 5. Section 6 introduces quasimaps
and Section 7 establishes the homeomorphism $\Psi$, the proof of
Theorem 0.2 appears in Section 7.4.

{\bf Acknowledgements:} I want to thank Chris Woodward for suggesting
the idea for this paper and many discussions that helped me along the
way. I was a post-doctoral fellow in Tata Institute of Fundamental
Research at the time the article was written. I was also hosted by
Department of Mathematics, Rutgers University for a month, whose
hospitality I am grateful for.  Finally, I thank the referee for
carefully reading the paper and suggesting improvements.
 
\section{Description of neck stretching
  metrics}\label{sec:neckstrectch}
In this section, we construct a family of metrics for stable nodal
marked curves with special points deleted so that punctured
neighborhoods in the curves are isometric to semi-infinite
cylinders. The presentation is similar to Section 2 in Gonzalez-Woodward
\cite{GW:deform}.
\subsection{Stable curves, gluing}
A compact complex {\em nodal} curve $C$ is obtained from a collection
of smooth compact curves $(C_1,\dots,C_k)$ by identifying a collection
of distinct {\em nodal points}
$$w_j^+\sim w_j^-, \quad w_j^\pm \in C_{\inc(w_j^\pm)}, \quad j=1,\dots,m.$$  
Points on the curves $C_i$ that are not nodal points are called {\em
  smooth points}. A nodal curve with marked points comes with a
collection of $n$ distinct smooth points $\{z_1, \dots, z_n\}$. A
marked nodal curve is {\em stable} if it has finite automorphism
group, i.e. every genus 0 component has at least 3 special points
(marked or nodal point) and a genus 1 component has at least 1 special
point. The {\em genus} of a nodal curve is the genus of a
``smoothing'' of the curve. For example, the genus of the curve in
Figure \ref{fig:nodal_curve} is one.

A {\em family of nodal curves} over a scheme $S$ is a proper flat
morphism $\pi:\mC \to S$ such that each fiber $\mC_s$, $s \in S$ is a
nodal curve. One can ask if there is a space $M$ and a family $U \to
M$ such that for any family $\mC \to S$, there is a unique map $\phi:
S \to M$ such that $\mC$ is isomorphic to the pullback $\phi^*U$. For
marked nodal curves with genus $\geq 1$, such a family does not exist
because there are curves with a non-trivial automorphism group - see
p. 267 in the book \cite{Arb}.
For stable curves, the automorphism group is finite. In that case,
there is a {\em coarse moduli space} $\ol M_{g,n}$ -- this means for
any family of curves $\mC \to S$, there is a unique map $\phi_S: S \to
\ol M_{g,n}$. There is a universal curve $\ol U_{g,n}:=\ol M_{g,n+1}$
that maps to $\ol M_{g,n}$ by forgetting the last point and
contracting unstable components. For any stable curve $C$, the fiber
over the point $[C] \in \ol M_{g,n}$ in $\ol U_{g,n}$ is isomorphic to
$C/\Aut(C)$. Given a family of curves over a pointed scheme $\mC \to
(S,0)$, the finite automorphism group of the central curve
$G:=\Aut(\mC_0)$ acts on an open neighborhood of $0$ in $S$ (which we
continue to call $S$). Further, the map $\phi_S:S \to \ol M_{g,n}$
factors through the quotient:
$$S \overset{\pi}{\to} S/G \overset{\phi_{S/G}}{\to} \ol M_{g,n},\quad \phi_S=\phi_{S/G} \circ \pi.$$
The action of $G$ lifts to $\mC$ and the quotient $\mC/G$ is
isomorphic to the pullback $\phi_{S/G}^*\ol U_{g,n}$. The moduli space
$\ol M_{g,n}$ has the structure of a {\em complex orbifold}. We use
the definition of an orbifold in Satake \cite{Satake}, which says that
locally a complex orbifold is homeomorphic to a neighborhood of the
quotient of $\C^N$ under the holomorphic action of a finite group, and
the transition functions are bi-holomorphisms. This definition has
problems and more sophisticated definitions have been provided, for
example Lerman \cite{Lerm:stack} says that orbifolds should be thought of as Deligne-Mumford stacks. But Satake's definition is
enough for the purposes of this paper.

The combinatorial type of a marked nodal curve $(C,z)$ is a {\em
  modular graph} $\Gamma=(\vertex(\Gamma), \Edge(\Gamma),
\Edge_\infty(\Gamma))$ and a genus function $g:\vertex(\Gamma) \to
\Z_{\geq 0}$. The vertices are the components of $C$, the edges in
$\Edge(\Gamma)$ are nodes in $C$ and the infinite edges
$\Edge_\infty(\Gamma)$ correspond to markings. An edge $w \in
\Edge(\Gamma)$ has two end points $w^\pm$, so it is incident on two
vertices $\inc(w^\pm)$. It is possible that these two vertices are the
same -- see for example Figure \ref{fig:nodal_curve}. For the
markings, there is an ordering of the set $\Edge_\infty(\Gamma)$ by a
bijection $\{1,\dots,n\} \to \Edge(\Gamma)$ and each edge $z \in
\Edge_\infty(\Gamma)$ is incident on only one vertex $\inc(z)$.

A modular graph is {\em stable} if any curve corresponding to it is
stable. The {\em stabilization} of a nodal curve $C$ is obtained by contracting curve components that have genus $0$ and less than 3 marked points.
For a graph $\Gamma$, the stabilization 
$\st(\Gamma)$ is the combinatorial type of the stabilization of any curve $C_\Gamma$ of type $\Gamma$. A {\em morphism} $f:\Gamma_1 \to \Gamma_2$ of
modular graphs corresponds to a sequence of moves, each of one of the two forms:
\begin{description}
\item [Contracting a non-loop edge] Contract an edge that is incident
  on vertices $v_1$, $v_2$, the genus of the combined vertex is the sum of the genus of $v_1$ and $v_2$.
\item [Contracting a loop edge] Delete an edge both whose end-points are incident on the same vertex $v$ and increment the genus of $v$ by $1$.
\end{description}
In terms of the
related nodal curves, each move smooths out a nodal singularity. Based
on the combinatorial type $\Gamma(C,z)$, we can define a
stratification of $\ol M_{g,n}$. If there is a morphism $\Gamma(C) \to
\Gamma(C')$, then $[C'] \preceq [C]$. The lowest stratum consists of
points representing smooth curves, and it is an open set in $\ol
M_{g,n}$. Denote by $M_{g,n}^\Gamma \subset \ol M_{g,n}$ the subspace
parametrizing curves of combinatorial type $\Gamma$. Then, the
boundary of $M_{g,n}^\Gamma$ consists of $\bigcup_{\Gamma \prec
  \Gamma'}M_{g,n}^{\Gamma'}.$

\begin{figure}[t]
  \centering \scalebox{.5}{
  \input{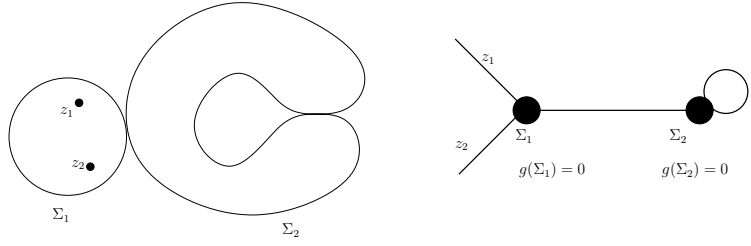} }
  \caption{A nodal curve and its combinatorial type. Both vertices
    have genus $0$, but the curve has genus $1$.}
  \label{fig:nodal_curve}
\end{figure}

Given a combinatorial type $\Gamma$, we next discuss how a
neighborhood in $M_{g,n}^\Gamma$ fits into $\ol M_{g,n}$ -- this is
done through deformation theory. Suppose $C$ is a compact curve. Then
a {\em deformation of $C$ by a pointed scheme $(S,0)$}, $0 \in S$, is
a proper flat morphism $\phi:\mC \to S$ plus an isomorphism between
$C$ and the central fiber $\phi^{-1}(0) \simeq \mC_0$. A deformation
$\mC$ of $C$ is {\em universal} if given any other deformation $\mC'
\to (S',0)$, for any sufficiently small neighborhood $U$ of $0$ in
$S'$, there is a {\em unique} morphism $\iota:U \to S$ such that
$\mC'|_U$ is isomorphic to the fibered product $\mC \times_S
U$. Stable curves possess universal deformations. Suppose $\pi:\mC \to
(S,0)$ is a universal deformation of the curve $C$. If $C$ has a
non-trivial stabilizer $G$, then for a sufficiently small neighborhood
$V$ of $0$, the action of $G$ extends to compatible actions on $V$ and
$\mC|_V$ (theorem 6.5, chapter 11, \cite{Arb}).  If $V$ is small
enough and $G$-invariant, then there is an injection $V/G \hra \ol
M_{g,n}$, and thus deformations of curves provide holomorphic orbifold
charts for $\ol M_{g,n}$.  Suppose $\mC \to (S,0)$ is a universal
deformation of the curve $C$. Then, the {\em space of infinitesimal
  universal deformations} of $C$, denoted by $\Def(C)$ is the tangent
space $T_0S$. If $C$ is a nodal curve of type $\Gamma$, {\em a
  deformation of $C$ of type $\Gamma$} is a family of curves $\mC \to
(S_\Gamma,0)$ where every fiber is a curve of type $\Gamma$, along
with an isomorphism between $C$ and the central curve $\mC_0$. If $\mC
\to S_\Gamma$ is a universal deformation of $C$ of type $\Gamma$, then
the {\em space of infinitesimal deformations of type $\Gamma$},
denoted by $\Def_\Gamma(C)$, is defined as the tangent space
$T_0S_\Gamma$. For a nodal curve $C$ of type $\Gamma$, given a
universal deformation $S_\Gamma$ of type $\Gamma$, we can construct a
universal deformation $S$ using a gluing procedure. Suppose $\tilde C$
is the normalization of $C$. By Proposition 3.32 in
\cite{HarrisMorrison},
$$\Def(C)/\Def_\Gamma(C)\simeq \oplus_{w \in \Edge(C)}T_{w^+}\tilde
C_{\inc(w^+)}\tensor T_{w^-}\tilde C_{\inc(w^-)}.$$
\vskip .05in
\begin{construction} \label{const:glue} Let $C$ be a nodal curve and
  suppose for all nodes $w \in \Edge(C)$, we are given holomorphic
  coordinates $\kappa_{w^\pm}$ in small neighborhoods of the lifts of
  the nodes $w^\pm$ in the normalization $\tilde C$. To every small
  $\delta \in \oplus_{w \in \Edge(C)}T_{w^+}\tilde
  C_{\inc(w^+)}\tensor T_{w^-}\tilde C_{\inc(w^-)}$, we can associate
  a curve $C^\delta$.
\end{construction}
\vskip -.1in
\noindent
$C^\delta$ is constructed as follows. The tensor
$\delta_w^\kappa:=(\d\kappa^+_w \tensor \d\kappa^-_w)(\delta_w)$ is a
complex number and we can define an equivalence
\begin{equation}\label{eq:gluedid}
  z_1 \sim z_2 \Leftrightarrow
  \kappa_w^+(z_1)\kappa_w^-(z_2)=\delta_w^\kappa.
\end{equation}
Define the glued curve as a quotient $C^\delta:=C - \{w^\pm:w \in
\Edge(C)\}/\sim$.  If $\delta_w=0$, we let the node $w$ stay in place
in $C^\delta$. Since the coordinates $\kappa_w^\pm$ are holomorphic,
the quotient relation respects complex structure, and so $C^\delta$
has a complex structure.\qed

The gluing process in Construction \ref{const:glue} can be done in
families also. Suppose, $\mC \to S_\Gamma$ is a family of curves of type $\Gamma$ and
we have a family of holomorphic co-ordinates
on neighborhoods of the lifts of nodes in $\mC_s$, where $s \in S_\Gamma$. Let
$I_\Gamma \to S_\Gamma$ be a vector bundle whose fiber over $s \in
S_\Gamma$ is
$$I_{\Gamma,s}:=\oplus_{w \in \Edge(\Gamma)}T_{w^+}\tilde \mC_{s,\inc(w^+)}\tensor T_{w^-}\tilde \mC_{s,\inc(w^-)}.$$
In a neighborhood of the zero section of
$I_\Gamma$, we can associate a glued curve to every point. 
If $S_\Gamma$ were a universal deformation of type $\Gamma$ of a curve $C$, by Theorem
 3.17, chapter 11 in \cite{Arb}, the gluing procedure above produces a universal deformation of the curve $C$.

\subsection{Riemann surfaces with cylindrical ends}

\begin{definition} Let $C$ be a nodal curve with marked points
  $z_1,\dots,z_n$ and nodal points $w_1,\dots, w_k$. A {\em Riemann
    surface with cylindrical ends} \footnote{Strictly speaking, the
    terminology is {\em Riemann surface with metric that has
      cylindrical ends at punctures}. To shorten notation, we assume
    {\em cylindrical ends} implies an underlying metric.}  associated
  to the curve $C$ is the punctured curve $\Sig:=C \bs
  \{z_1,\dots,z_n, w_1,\dots,w_k\}$ equipped with a metric satisfying
  the following property: for any $z=z_j,w^\pm_j$ there is a
  neighborhood of $z$ in the normalization of $C$, $N(z) \subset
  \tilde C$ and an isometry
$$\rho_z:N(z)\bs \{z\} \to \{r+\img\theta:r > 0, \theta \in \R/2\pi \Z\},$$
such that $\lim_{r \to \infty} \rho_z(r,\theta)=z$. The metric on the
right hand side is $\d r^2+\d\theta^2$. We further require that the
volume of $\Sig \bs (\cup_{z=z_j,w^\pm_j} N(z))$ is finite and any $f
\in \Aut(C)$ preserves the metric on $\Sig$.
\end{definition}

\begin{notation} In the context of cylindrical co-ordinates, $S^1$ will denote $\R/2\pi\Z$.
\end{notation}

Given integers $g,n \geq 0$ such that $n+2g-3 \geq 0$, the goal of
this section is to show that to any stable $n$-pointed curve of genus
$g$, we can associate a Riemann surface with cylindrical ends, and
that the metric varies smoothly as we vary $C$. In other words, if
$\mC \to S$ is a family of such curves, we need to put a metric on the
fibers of $\mC\bs (z_j(S) \cup w_j(S)) \to S$ that varies smoothly
with $s \in S$. We call such a family of metrics {\em neck-stretching
  metrics}, because they stretch the `neck region' in the glued
surfaces described in Construction \ref{const:glue}. A family of neck-stretching metrics is constructed using a family of holomorphic co-ordinates $\kappa$
defined on $\ol M_{g,n}$ on neighborhoods of marked points and lifts
of nodal points. Note that in the construction \ref{const:glue}, for a
curve $C^\delta$, if a node $w$ stays in place, i.e. $\delta_w=0$,
then the coordinates $\kappa_{w^\pm}$ on $C$ also induce coordinates in the neighborhoods of $w^\pm$
on $\tilde C^\delta$. A family of holomorphic coordinates defined on a family
of nodal curves is said to be {\em compatible} if the co-ordinate
$\kappa_{w^\pm}$ on $C$ and $C^\delta$ agree, when $\delta$ is in a small neighborhood of $0$.

\begin{proposition}\label{prop:neckstretch} There exists a family of
  neck-stretching metrics on stable curves parametrized by $\ol
  M_{g,n}$. The space of neck-stretching metrics on $\ol M_{g,n}$ is
  contractible.
\end{proposition}
\begin{proof}
  Assume $\kappa$ is a compatible family of holomorphic coordinates
  defined on $\ol M_{g,n}$ in neighborhoods of marked points and lifts
  of nodal points.  We first describe a neck-stretching metric on a
  component of a curve, which is called $C$ and has special points
  $z_1,\dots,z_n$. Holomorphic co-ordinates $\kappa_{C,j}$ are defined
  on neighborhoods $N(z_j)$ of $z_j$. Define map $f:B_1 \to \R \times
  \img S^1$ as $z \mapsto -\ln z$. In the neighborhood
  $N(z_j)\bs \{z_j\}$, define co-ordinates $\rho_{z_j}:N(z_j) \bs
  \{z_j\} \to \R \times \img S^1$ as $f \circ
  \kappa_{C,i}$. If $z_j=w^-$ for some edge $w$, we use the
  co-ordinates $\rho_{w^-}:=-f \circ \kappa_{C,w^-}$.  A cylindrical
  metric is given on $N(z_j) \bs \{z_j\}$ by pulling back the
  Euclidean metric by $\rho_{z_j}$. This metric can be extended to all
  of $\Sig:=C\bs \{z_1,\dots,z_n\}$ so that the volume of $C \bs \cup_j N(z_j)$ is finite.

  For a combinatorial type $\Gamma$, the family of metrics defined on
  the stratum $M_{g,n}^\Gamma$ can be extended to a small neighborhood
  of $M_{g,n}^\Gamma$ by the gluing construction for families. Given a
  curve $C$ of type $\Gamma$, and $\delta \in \oplus_{w \in
    \Edge(\Gamma)}T_{w^+}\tilde C \tensor T_{w^-}\tilde C$, we now
  describe the metric on the curve $\Sig^\delta:=C^\delta \bs
  \{\text{special points}\}$. In the previous paragraph, by pulling
  back the Hermitian metric on $B_1$ via $\kappa_{w^\pm}$ to a
  neighborhood $N(w^\pm) \subset \tilde C$, the element $\delta_w$ can
  be identified to a complex number. Let $l_w=L_w+it_w:=-\ln \delta_w
  \in \R_{\geq 0} \times S^1$. The identification \eqref{eq:gluedid}
  used in the construction of $C^\delta$ can be re-written as
  \begin{equation}\label{eq:lid}
    (\rho_{w^+})^{-1}(r+\img\theta) \sim (\rho_{w^-})^{-1}(l_w+r+\img\theta).
  \end{equation}
  To construct the punctured curve with metric $\Sig^\delta$, we start
  with $\Sig$. For any node $w \in \Edge(\Sig)$, parts of the
  semi-infinite cylinders $\rho_{w^\pm}^{-1}\{r> \pm L_w\}$ are
  discarded and the remainder of the cylinders are identified via
  \eqref{eq:lid} to produce a finite cylinder in $\Sig^\delta$. The
  two regions being identified are isometric, so the metric on $\Sig$
  induces a metric on $\Sig^\delta$. If a modular graph $\Gamma'$ is
  obtained from $\Gamma$ by contracting a single edge, the family of
  metrics can be extended to $M_{g,n}^{\Gamma'}$. If we assume that
  the metrics on curves of type $\Gamma$ are invariant under the
  action of the automorphism group of the curve, the condition would
  also be satisfied for the metrics on curves of type $\Gamma'$.  A
  family of metrics on $\ol M_{g,n}$ can thus be constructed
  inductively.

  Finally we prove that the space of such metrics is contractible. For
  a fixed family $\kappa$ of holomorphic co-ordinates in the
  neighbourhood of nodes in curves parametrized by $\ol M_{g,n}$, the space of
  metrics on stable curves is convex, hence contractible: Given two
  families of neck-stretching metrics $\ulg_0$, $\ulg_1$, for any $t
  \in [0,1]$, $(1-t)\ulg_0+t\ulg_1$ is also a neck-stretching
  metric. This is because the metrics $\ulg_0$ and $\ulg_1$ agree on
  the neighborhoods of special points. Now, consider different choices
  of $\kappa$, say $\kappa_0$ and $\kappa_1$. We remark that
  multiplying $\kappa_{w^+}$ or $\kappa_{w^-}$ by a constant non-zero
  complex number does not affect the metric constructed above - it
  only modifies the identification $\delta \mapsto C^\delta$,
  i.e. $C^\delta$ would now be relabelled by a multiple of
  $\delta$. So, the family of co-ordinates $\kappa_1$ can be
  multiplied by a family of scalars so that for any curve $C$ and
  special point $z$, $\d\kappa_0(z)=\d\kappa_1(z)$. Such a
  modification is possible because any strata of $\ol M_{g,n}$ is
  simply connected. Now, we see that for $t \in [0,1]$, the
  interpolation $\kappa_t:=(1-t)\kappa_0+t\kappa_1$ is also a family
  of holomorphic coordinates in (possibly smaller) neighborhoods of
  special points (see Remark 2.1.1 in \cite{GW:deform}).
\end{proof}

\begin{figure}[t]
  \centering \scalebox{.5}{
    \begin{picture}(0,0)%
      \includegraphics{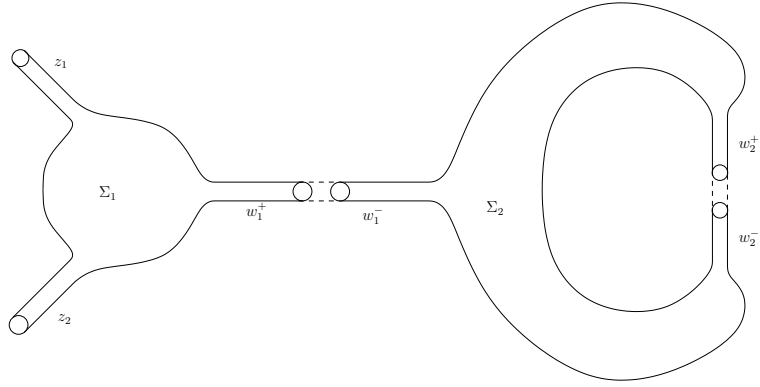}%
    \end{picture}%
    \setlength{\unitlength}{4144sp}%
    \begingroup\makeatletter\ifx\SetFigFont\undefined%
    \gdef\SetFigFont#1#2#3#4#5{%
      \reset@font\fontsize{#1}{#2pt}%
      \fontfamily{#3}\fontseries{#4}\fontshape{#5}%
      \selectfont}%
    \fi\endgroup%
    \begin{picture}(8823,4543)(-321,-5955)
      \put(5401,-3931){\makebox(0,0)[lb]{\smash{{\SetFigFont{12}{14.4}{\rmdefault}{\mddefault}{\updefault}{\color[rgb]{0,0,0}$\Sig_2$}%
            }}}}
      \put(271,-5236){\makebox(0,0)[lb]{\smash{{\SetFigFont{12}{14.4}{\rmdefault}{\mddefault}{\updefault}{\color[rgb]{0,0,0}$z_2$}%
            }}}}
      \put(226,-2176){\makebox(0,0)[lb]{\smash{{\SetFigFont{12}{14.4}{\rmdefault}{\mddefault}{\updefault}{\color[rgb]{0,0,0}$z_1$}%
            }}}}
      \put(2521,-3976){\makebox(0,0)[lb]{\smash{{\SetFigFont{12}{14.4}{\rmdefault}{\mddefault}{\updefault}{\color[rgb]{0,0,0}$w_1^+$}%
            }}}}
      \put(3916,-4021){\makebox(0,0)[lb]{\smash{{\SetFigFont{12}{14.4}{\rmdefault}{\mddefault}{\updefault}{\color[rgb]{0,0,0}$w_1^-$}%
            }}}}
      \put(8416,-3166){\makebox(0,0)[lb]{\smash{{\SetFigFont{12}{14.4}{\rmdefault}{\mddefault}{\updefault}{\color[rgb]{0,0,0}$w_2^+$}%
            }}}}
      \put(8416,-4291){\makebox(0,0)[lb]{\smash{{\SetFigFont{12}{14.4}{\rmdefault}{\mddefault}{\updefault}{\color[rgb]{0,0,0}$w_2^-$}%
            }}}}
      \put(766,-3751){\makebox(0,0)[lb]{\smash{{\SetFigFont{12}{14.4}{\rmdefault}{\mddefault}{\updefault}{\color[rgb]{0,0,0}$\Sig_1$}%
            }}}}
    \end{picture}%
  }
  \caption{Nodal curve in Figure \ref{fig:nodal_curve} with
    neck-stretching metric}
  \label{fig:neckstretch}
\end{figure}

We make a choice of a family of neck-stretching metrics and fix it for
the rest of the paper. Now, we can talk about a `Riemann surface with
cylindrical ends corresponding to a stable nodal curve'.

\section{Preliminaries: vortices on surfaces with cylindrical
  ends}\label{sec:prelim}
\subsection{Definitions}
% Let $G$ be a connected complex reductive Lie group, so it is the
% complexification of a maximal compact subgroup $K$.
Let $K$ be a compact connected Lie group, $(\Sig,j)$ be a
Riemann surface with metric and $P \to \Sig$ be a principal $K$-bundle.  Let
$(X,\om, J)$ be a K\"{a}hler manifold with a $K$-action that preserves $\om$ and $J$. We assume the $K$-action on $X$ is {\em Hamiltonian}, whose meaning we now define.

\begin{definition}
  \begin{Notes} \item {\rm (Hamiltonian actions)} A {\em moment map}
    is a $K$-equivariant map $\Phi$ such that $\iota(\xi_X) \omega =
    \d \lan \Phi, \xi \ran , \ \forall \xi \in \k$, where the vector field $\xi_X \in \on{Vect}(X)$ is the infinitesimal action of $\xi$ on $X$, i.e. $\xi(x):= \ddt \exp(t\xi)x \in T_xX$.  The
    action of $K$ is {\em Hamiltonian} if there exists a moment map
    $\Phi : X \to \k^\dual$. Since $K$ is compact, $\k$ has an
    $\Ad$-invariant metric. We fix such a metric so
    the moment map becomes a map $\Phi : X \to \k$.
  \item {\rm (Connections, curvature and gauge transformations)} A {\em connection} is a $K$-equivariant
    one-form $A \in \Om^1(P,\k)$ that satisfies $A(\xi_P)=\xi$ for all
    $\xi \in \k$. The space of connections $\A(P)$ is an affine space
    modeled on $\Om^1(\Sig,P(\k))$ where $P(\k):=P\times_K \k$ is the
    adjoint bundle. The form $\d A+\hh[A \wedge A] \in \Om^2(P,\k)$ is
    basic, so it descends to a two-form $F_A \in \Om^2(\Sig,P(\k))$,
    which is the {\em curvature} of the connection $A$. A {\it gauge
      transformation} is an automorphism of $P$ -- it is an
    equivariant bundle map $P \to P$. The group of gauge
    transformations on $P$ is denoted by $\K(P)$.
  \end{Notes}
\end{definition}

In case $P$ is a trivial bundle $\Sig \times K$, there is a trivial
connection $\d$. The adjoint bundle has a trivialization $P(\k) \simeq
\Sig \times \k$. A connection $A$ is then of the form $A=\d+a$ where
$a \in \Om^1(\Sig,\k)$. The formula of curvature is $F_A=\d a + \hh[a
\wedge a].$ A gauge transformation $k:\Sig \to K$ acts on the
connection $A=\d + a$ as $$k(A)= \d+(\d k k^{-1} + \Ad_ka).$$

On an associated bundle $P(X):=P \times_K X$, a connection $A$ defines
a covariant derivative $d_A$ on sections of $P(X)$:
$$ \d_A: \Gamma(\Sig,P(X)) \to \cup_{u \in  \Gamma(\Sig,P(X))}\Om^1(\Sig, u^*T^{vert}P(X)).$$
On a local trivialization of the bundle $P$, where $A=\d+a$, the
covariant derivative of a section $u:\Sig \to X$ is $\d u + a_u \in
\Om^1(\Sig,u^*TX)$. At a point $x \in \Sig$, $a_u(x)$ is the
infinitesimal action of $a(x)$ at $u(x)$. 
\begin{remark} \label{rem:conntohol} Suppose $(M,J)$ is a complex
  $K$-manifold whose complex structure $J$ is preserved by $K$.  A
  connection $A$ on $P$ determines a holomorphic structure on the
  associated bundle $P(M):=P \times_K M$ via the operator
  $\delbar_A:=\d_A^{0,1}$. This construction can be reversed when $M$
  is the complexified group $G=K^\C$. Then, the associated bundle is a
  principal $G$-bundle containing $P$ as a sub-bundle. Given a
  holomorphic structure on $P_\C$, the co-dimension one distribution
  $TP \cap J(TP) \subset TP$ corresponds to a connection on $P$ (see
  \cite{Singer:conn}).
\end{remark}
\begin{definition} \label{def:gaugedmap} 
  \begin{Notes}
  \item {\rm (Gauged holomorphic maps)} A {\em gauged holomorphic map}
    $(A,u)$ from $P$ to $X$ consists of a connection $A$ and a section
    $u$ of $P(X)$ that is holomorphic with respect to $\delbar_A$. The
    space of gauged holomorphic maps from $P$ to $X$ is called
    $\H(P,X)$.
  \item {\rm (Symplectic vortices)} A {\em symplectic vortex} is a
    gauged holomorphic map that satisfies 
$$F_A +\Phi(u)\om_\Sig=0,$$
    where $\om_\Sig \in \Omega^2(\Sig)$ is an area form on $\Sig$.
  \item {\rm (Energy)} The {\em energy} of a gauged holomorphic map
    $(A,u)$ is
    $$E(A,u):= \int_\Sig (|F_A|^2+|\d_Au|^2+|\Phi \circ u|^2)
    \om_\Sig.$$
  \end{Notes}
\end{definition}

For a compact base manifold $\Sig$, one can define Sobolev completions
of the spaces $\A(P)$, $\K(P)$ in a standard way (see Appendix B of
the book \cite{Weh:Uh}).  But for a manifold with cylindrical ends
$\Sig$, we restrict our attention to a trivial principal bundle $P
=\Sig \times K$. Let $k$, $p$ such that $(k+1)p>\dim(\Sig)$. The space
of $W^{k,p}$ connections is $\A^{k,p}(P):=\d+ W^{k,p}(\Sig,\k)$ where
$\d$ is the trivial connection on $\Sig \times K$. The space of
$W^{k+1,p}$ gauge transformations is $\K^{k+1,p}(P):=\{ke^s|k \in K, s
\in W^{k+1,p}(\Sig,\k)\}$. Unlike in the case of a compact base
manifold, these Sobolev spaces are dependent on the choice of
trivialization and hence it is not possible to define these concepts
for a general principal bundle. The space $\K^{k+1,p}(P)$ acts
smoothly on $\A^{k,p}(P)$. For integers $k$, we also use the Sobolev spaces $H^k$, which are the same as $W^{k,2}$.

We assume the following in the rest of this article.

\begin{assumption}\label{ass:freeaction} $K$ acts freely on
  $\Phi^{-1}(0)$. %or equivalently $G$ acts freely on $X^\ss$ in cases where $X$ is a polarized $G$-variety which means that the semistable and stable loci of $X$ coincide.
\end{assumption}

\begin{definition}[Equivariant convexity at infinity]\label{def:conv} The $K$-Hamiltonian manifold $X$ is equivariantly convex if there is a proper $K$-invariant function $f:X \to \R_{\geq 0}$, and
  a constant $c>0$ such that $x \in f^{-1}[0,c)$ and $\xi \in T_xX$,
  \begin{equation*}%\label{eq:conv1}
\lan \nabla_\xi \nabla f(x), \xi \ran \geq 0, \quad \d f(x)(J\Phi(x)_X) \geq 0.
  \end{equation*}
\end{definition}
The above notion is defined by Cieliebak et al. in \cite{CGMS} and is an equivariant version of the idea of convexity in
symplectic geometry introduced by Eliashberg and Gromov
\cite{Eliashberg:convexity}. If the moment map $\Phi$ is proper and
$X$ is equivariantly convex, the image of a finite energy vortex with bounded image is
contained in the compact set $\{f \leq c\}$ (see Lemma 2.7 in \cite{CGMS}). An important example of equivariantly
convex spaces are symplectic vector spaces with a linear group action and a
proper moment map.
 
\subsection{Asymptotic behavior on cylindrical ends}

On a non-compact base space, vortices with bounded image and finite energy have good
asymptotic properties (see Ziltener \cite{Zilt:invaction}). In this section we prove a stronger result (Proposition \ref{prop:connW1p}) in the special case that the base space has cylindrical ends.

\begin{proposition}[Decay for vortices on the half
  cylinder]\label{prop:decaycyl} Let $\Sig$ be a half cylinder
$$\Sig:=\{r+\img\theta:r \geq 0, \theta \in S^1\}$$ 
with the standard metric $\om_\Sig=\d r \wedge \d\theta$. Let $X$ be a
Hamiltonian $K$-manifold with a proper moment map and such that $K$
acts freely on $\Phi^{-1}(0)$. There exists a constant $0<\gamma<1$
that satisfies the following. Suppose $(A,u)$ is a finite energy
vortex from the trivial bundle $\Sig \times K$ to $X$ whose image is
contained in a compact subset of $X$. Then, there exists a constant
$C>0$ so that
\begin{align*}
  |F_A(z)|^2+|\d_Au(z)|^2+|\Phi(u(z))|^2 \leq Ce^{-\gamma r}.
\end{align*}
where $z=r+\img\theta$.
\end{proposition}

\begin{remark} This result is similar to Theorem 1.3 in Ziltener
  \cite{Zilt:invaction} and the proof carries over. It is an
  equivariant version of the isoperimetric inequality. But this result
  is much weaker, because in \cite{Zilt:invaction}, $\gamma$ is
  arbitrarily close to $1$. The hypothesis in \cite{Zilt:invaction}
  place some condition on $X$ in order to make the metric on $\Sig$
  `admissible', which can be dropped if we do not require an optimal
  value of $\gamma$.
\end{remark}

The following result is an easy consequence of Proposition
\ref{prop:decaycyl}. It says that for a finite energy vortex on a semi-infinite cylinder, the connection asymptotically converges to a limit connection defined on $S^1$. Further the map $u$ has a limit at infinity after it is twisted by a gauge transformation determined by the limit connection.
 It is enough to consider trivial principal bundles for this result because on a Riemann surface which is
homotopic to a one-cell, any principal bundle with connected structure group is trivializable.
\begin{proposition} [Removal of singularity for vortices at
  infinity]\label{prop:infvortsing} Suppose $(A,u)$ is a finite
  energy vortex on a trivial $K$-bundle on the half cylinder
  $\{r+\img\theta:r \geq 0, \theta \in S^1\}$ whose image is contained
  in a compact subset of $X$. Then, there exist a point $x_0 \in \Phi^{-1}(0)$
  and a gauge transformation $k_0 \in C^1(S^1,K)$ such that
  $$\lim_{r \to \infty}\max_\theta d(k_0(\theta)x_0,u(r+\img\theta))=0.$$
  Suppose, the restriction of $A$ in radial gauge to the circle
  $\{r=r_0\}\simeq S^1$ is $\d+a_\theta\d \theta$, there exist
  constants $c, \gamma>0$ such that for all $r_0 \geq 0$,
  \begin{equation}\label{eq:conndecay}
    |k_0^{-1}\partial_\theta k_0+a_\theta(r_0,\cdot)|<ce^{-r_0\gamma}.
  \end{equation}
\end{proposition}
The above result ensures that the $K$-orbit $Ku(\infty)$ is
well-defined.

\begin{proposition} \label{prop:connW1p} Let $p>1$. Suppose $(A,u)$ is
  a finite energy vortex on a trivial $K$-bundle on the half cylinder
  $\Sig=\{r+\img\theta:r \geq 0, \theta \in S^1\}$ whose image is
  contained in a compact subset of $X$. There is a gauge
  transformation $k \in W^{2,p}_{loc}(\Sig,K)$, such that if
  $kA=\d+a$, then,
$$\Mod{a}_{W^{1,p}(\{n \leq r \leq n+1\})} \leq ce^{-n\gamma}$$
for some positive constant $c>0$ and $0<\gamma<1$. Hence
$\Mod{a}_{W^{1,p}(\Sig)}<\infty$.
\end{proposition}
\begin{proof} We first describe a cover of $\Sig$, in which all but a finite number of open sets have identical geometry. On these sets, the connection can be uniformly bounded using Uhlenbeck compactness. Fix any
  $\eps<\hh$. Let $U:=\{r+\img\theta: 1-\eps \leq r \leq 2+\eps,
  -\eps<\theta<\pi+\eps\}$. For any integer $n \geq 0$,
  \begin{align*}
    U_{0,n}&:=U+n,\quad U_{1,n}:=U+n+\img\pi,\\
    \tilde U_n&:=\{r+\img\theta:  n +1-\eps\leq r\leq n  +2+\eps, -\eps\leq \theta \leq 2\pi+\eps\},\\
    U_n&:=\tilde U_n/\{\theta \sim \theta+2\pi:-\eps \leq \theta \leq
    \eps\} \subset \Sig.
  \end{align*}
  There is an integer $n_0$ such that
  $\Mod{F_A}_{L^p(U_{j,n})}<\kappa$ for all $j=0,1$, $n \geq n_0$,
  where $\kappa$ is the constant in Uhlenbeck's local Theorem
  (Theorem \ref{thm:Uhcpt}). This bound ensures that the connection
  can be put in Coulomb gauge in each of these sets. By Uhlenbeck
  compactness, there is a gauge transformation $g_{j,n}$ on $U_{j,n}$
  for all $n \geq n_0$, $j=0,1$ such that $g_{j,n}A$ is in Coulomb
  gauge, i.e. denoting $a_{j,n}=g_{j,n}A-d$,
  \begin{equation}\label{eq:uhgauge}
    \Mod{a_{j,n}}_{W^{1,p}(U_{j,n})} \leq c\Mod{F_A}_{L^p(U_{j,n})}, \quad \d^*a_{j,n}=0.
  \end{equation}
  The constants $c$, $\kappa$ are independent of $(j,n)$ because the
  domains $U_{j,n}$ are identical to each other. We remark that since
  the set $U_{j,n}$ has corners, Uhlenbeck's Theorem \ref{thm:Uhcpt}
  will have to be applied to closed sets with smooth boundary that are
  slightly bigger than $U_{j,n}$. From Proposition \ref{prop:decaycyl}, $\Mod{F_A}_{L^p(U_{j,n})} \leq
  ce^{-\gamma n}$ for some $\gamma>0$.

  Next, we patch all these gauge
  transformations to get one on $\cup_{n \leq n_0} \tilde U_n$, which is a cover of $\cup_{n \leq n_0} U_n \subset \Sig$.  First look at one
  of the components of $U_{0,n} \cap U_{1,n}$ given by
  $\pi-\eps<\theta<\pi+\eps$, we call these $V_n$. On $V_n$, let
  $g_{1,n}=e^{\xi_n} g_{0,n}$, where $\xi_n:V_n \to \k$. Let $\psi$ be
  a cut-off function on $[\pi-\eps,\pi+\eps]$ that is $1$ in the
  neighborhood of $\pi-\eps$ and supported away from $\pi+\eps$. Let
  $h_n:=e^{\psi\xi_n}g_{0,n}$. Define a gauge transformation $g_n$ on
  $\tilde U_n$ as being equal to $h_n$ on $V_n$ and equal to $g_{0,n}$
  or $g_{1,n}$ outside $V_n$. Define $a_n:=g_nA-\d$. We have to show
  that
  \begin{equation*}
    \Mod{a_n}_{W^{1,p}(\tilde U_n)} \leq c\Mod{F_A}_{L^p(\tilde U_n)}
  \end{equation*}
  For this it is enough to show that $\Mod{h_n}_{W^{2,p}(V_n)} \leq
  c\Mod{F_A}_{L^p(\tilde U_n)}$, which in turn can be shown by a
  similar bound on $\Mod{\xi_n}_{W^{2,p}(V_n)}$. The inequality
  \eqref{eq:uhgauge} is unchanged if $g_{1,n}$ is multiplied by a
  constant factor. So we can assume there is a point $p \in V_n$ on
  which $g_{0,n}$ and $g_{1,n}$ agree and so $\xi_n(p)=0$. We know $a_{1,n}=\d\xi_n +
  \Ad_{e^{\xi_n}}a_{0,n}$. By a standard argument (see for example
  proof of Lemma 2.4 in \cite{Uh:compactness}), we get
$$\Mod{\xi_n}_{W^{2,p}(V_n)} \leq c(\Mod{a_{0,n}}_{W^{1,p}(U_{0,n})} + \Mod{a_{1,n}}_{W^{1,p}(U_{1,n})}) \leq c\Mod{F_A}_{L^p(U_n)}.$$
By a similar process, we can patch $g_n$ and $g_{n+1}$ for each $n$
and obtain a gauge transformation on all of $\tilde g$ on $\tilde
U=\cup_n \tilde U_n$. Denote $\tilde a:=\tilde g A-d$. We have bounds
\begin{equation*}\label{eq:tildeabd}
  \Mod{\tilde a}_{W^{1,p}(\tilde U_n)} \leq c\Mod{F_A}_{L^p(U_n)} \leq ce^{-\gamma n}.
\end{equation*}

Finally, we produce a gauge transformation on $\Sig$. The gauge transformation $g: \cup_{n \geq n_0}U_n \to K$ is defined by patching $\tilde g$ as above, but we cannot get bounds $g$ using the above technique. This is because for the other component of
$U_{0,n} \cap U_{1,n}$, given by $[-\eps,\eps]$ and
$[2\pi-\eps,2\pi+\eps]$, the gauge
transformations may not agree at any point as we no longer have the
flexibility of modifying anything by a constant. But, we can get a
similar bound if we can show that
\begin{equation}\label{eq:expconv}
  d(\tilde g(r+\img\theta)^{-1}\tilde g(r+\img(\theta+2\pi)),\Id) \leq ce^{-\gamma r}.
\end{equation}
This inequality is proved as follows. Denote by $\on{Hol}_r$ the
holonomy of $A$ around the loop $[0,2\pi]\ni \theta \mapsto r+\img\theta
\in \Sig$. By Proposition \ref{prop:infvortsing},
\begin{equation}\label{eq:hololdbd}
  d(\on{Hol}_r,\Id) \leq ce^{-\gamma r}.
\end{equation}
Now, denote by $\Hol^{new}_r$ the holonomy of $\tilde g A$ along the
path $[0,2\pi] \ni \theta \mapsto r+\img\theta \in \tilde U$ -- note that
the end points are not identified in $\tilde U$, but since we are on a
trivial bundle there is a canonical identification between the fibers
at the end points. So, we have
\begin{equation}\label{eq:holrel}
  \Hol^{new}_r=\tilde g(r+2\pi \img)\Hol_r \tilde g(r)^{-1}.
\end{equation}
Recall $\Hol^{new}_r$ is defined as follows: denote $\tilde a=\tilde
a_r\d r +\tilde a_\theta \d\theta$, let $h_r:[0,2\pi] \to K$ be given
by the ODE
$$h_r^{-1}\frac {\d h_r}{\d\theta}=\tilde a_\theta, \quad h_r(0)=\Id,$$
then $\Hol^{new}_r:=h_r(2\pi)$. Since $\Mod{\tilde
  a_\theta}_{W^{1,p}(\tilde U_n)} \leq ce^{-\gamma n}$, we have
$\sup_{\theta}|a_\theta(r+\img\theta)| \leq ce^{-\gamma r}$, and this
implies $d(\Hol_r^{new},\Id) \leq ce^{-\gamma r}$. Together with
\eqref{eq:hololdbd} and \eqref{eq:holrel}, \eqref{eq:expconv} follows. The resulting gauge transformation $g:\cup_{n \geq n_0}U_n \to K$ can be extended to all of $\Sig$. The choice of extension is not important because we only require an asymptotic bound.
\end{proof}

\begin{corollary} \label{cor:connW1psig} Let $p>1$. Suppose $(A,u) \in
  W^{1,p}_{loc} \times W^{2,p}_{loc}$ is a finite energy vortex on a
  Riemann surface $\Sig$ with cylindrical ends whose image is
  contained in a compact subset of $X$. We assume $\Sig \simeq \ol\Sig\bs \{z_1,\dots,z_n\}$, where $\oSig$ is compact. Further, for any $1 \leq j \leq n$, let $\rho_{z_j}$  be cylindrical coordinates  on the punctured neighborhood $N(z_j)\bs \{z_j\}$. Then, there exists a
  gauge transformation $k \in W^{2,p}_{loc}(\Sig,K)$ and
  $\lambda_1,\dots,\lambda_n \in \k$ so that if
  $kA=\d+\lambda_j\d\theta + a_j$ on $N(z_j)\bs \{z_j\}$, then,
  \begin{equation} \label{eq:twistconnbd} \Mod{a_j}_{W^{1,p}(\rho_{z_j}^{-1}(n
      \leq r \leq n+1))} \leq ce^{-\gamma n}
  \end{equation}
  and hence, $\Mod{a_j}_{W^{1,p}(N(z_j)\bs \{z_j\})}<\infty$.
\end{corollary}
\begin{proof} It is enough to assume that there is only one
  cylindrical end corresponding to $z \in \ol \Sig \bs \Sig$, so we can drop $j$ from the notation. We fix a trivialization of the principal $K$-bundle $P\to \Sig$. Apply Proposition
  \ref{prop:connW1p} to the restriction of $(A,u)$ to the half
  cylinder $N(z) \bs \{z\}$, and call the resultant gauge
  transformation $k_1$. The homotopy equivalence class of the map
  $k_1|_{\{r=0\}}:S^1 \to K$ will contain a geodesic loop
  $\theta \mapsto e^{-\lambda \theta}$, where $\lambda \in \frac 1 {2\pi} \exp^{-1}(\Id) \subset \k$. Now, $e^{\lambda \theta}k_1A=\d+\lambda
  \d\theta + a$ and $a$ satisfies \eqref{eq:twistconnbd}. The gauge
  transformation $k:=e^{\lambda \theta}k_1:[0,\infty) \times S^1
  \to K$ is homotopic to the constant identity map, so it is possible to choose an extension of $k$ to all of $\Sig$.
\end{proof}

\begin{corollary}\label{cor:extendoverinfty} Let $(A,u)$ be a finite
  energy vortex as in Corollary \ref{cor:connW1psig}. Assume $p$
  satisfies $0<1-\frac 2 p<\gamma$. Then, there is a principal
  $K$-bundle over $\oSig$ so that $(A,u)$ extends to a $L^p_{loc}
  \times W^{1,p}_{loc}$ gauged holomorphic map on $P$.
\end{corollary}
\begin{proof} The bundle $P$ can be defined as follows. The trivial
  bundle $N(z_j) \times K$ is glued to $\Sig \times K$ using
  transition function
$$k:N(z_j) \bs \{z_j\} \to K, \quad \theta \mapsto e^{-\lambda_j\theta}.$$
On the bundle $\Sig \times K$, $A|_{N(z_j) \bs \{z_j\}}=\d+\lambda_j
\d\theta +a_j$ and $a_j$ satisfies an exponential bound as in
\eqref{eq:twistconnbd}. Then, on the bundle $N(z_j) \times K$,
$A|_{N(z_j) \bs \{z_j\}}=\Ad_{e^{-\lambda_j \theta}}a_j$. By the estimate 
\eqref{eq:twistconnbd}, the one-form $a_j$, 
and hence $\Ad_{e^{-\lambda_j\theta}}a_j$, are both in $L^p(N(z_j))$. This is because the coordinates
in $N(z_j)$ are given by $e^{-\rho_j}$. This operation stretches unit
vectors by a factor of $e^{-r}$, so the 1-form $a_j$ is in
$L^p(N(z_j))$ if $-\gamma + 1 - \frac 2 p<0$. Using Proposition
\ref{prop:infvortsing}, we can see that $ku$ extends continuously over
$z_j$, subsequently $u \in W^{1,p}(N(z_j))$ by elliptic regularity.
\end{proof}

\begin{remark}[Equivariant homology class] Suppose $(A,u)$ is a finite
  energy vortex on a surface $\Sig$ with cylindrical ends whose image
  is pre-compact in $X$, then by the removal of singularity
  Proposition \ref{prop:infvortsing} and Corollary
  \ref{cor:connW1psig}, there is a $K$-bundle $P$ on $\oSig$ such that
  $(A,u)$ extends to a gauged holomorphic map over $\oSig$. So, $[u]$
  represents an equivariant homology class in $H_2^K(X)$. Theorem 3.1 in
  \cite{CGS} is also applicable, which says that vortices are energy
  minimizers in the equivariant homology class and
  \begin{equation}\label{eq:energyhom}
    E(A,u)=\lan [\om -\Phi],[u]\ran, 
  \end{equation}
  where $\om -\Phi \in H^2_K(X)$ and $\lan \cdot, \cdot \ran$ is the
  pairing between equivariant homology and cohomology.
\end{remark}

\section{Gromov convergence for vortices}\label{sec:gromovdef}
Suppose $C$ is a marked nodal curve and $\Sig$ is a Riemann surface with cylindrical ends corresponding to $C$. 
A {\em vortex} on $\Sig$ is a tuple $(A_\alpha,u_\alpha)_{\alpha \in \vertex(C)}$ of finite energy vortices defined on the components of $\Sig$ that satisfy
the {\em connectedness} condition :
\begin{equation}\label{eq:connect}
  Ku_{\inc(w^+)}(w^+)=Ku_{\inc(w^-)}(w^-), \quad \forall w \in \Edge(C).
\end{equation}
To compactify the space of vortices on Riemann surfaces with cylindrical ends, we need to allow breaking of cylinders at marked points and nodes. We illustrate this phenomena in a simplified setting. Let $(A_\nu,u_\nu)$ be a sequence of vortices on $\Sig$ whose energy is bounded. The limiting object $(A_\infty,u_\infty)$ will be a {\em stable vortex}, consisting of a vortex defined on each component $\Sig_\alpha$ in $\Sig$. In addition, 
\begin{enumerate}
\item\label{part:markbubble} at any marked point $z \in \Edge_\infty(C)$, there may be a path of cylindrical
  vortices (vortices on $S^1 \times \R$).
\item\label{part:nodebubble} Any edge $w \in \Edge(C)$ may be
  replaced by a path of cylindrical vortices joining $w^+$ to $w^-$.
\end{enumerate}
\begin{assumption} In Sections \ref{sec:gromovdef}, \ref{sec:convpf}
  and \ref{sec:hausdorff}, $(X,\om,J)$ is an aspherical
  $K$-Hamiltonian symplectic manifold with a $K$-invariant almost
  complex structure $J$. Refer to Remark \ref{rem:asphere} for the
  case that $X$ has spheres.
\end{assumption}
%
%  For any marked point $z$ or nodal point $w$ in $C$,
% by the finite energy condition,
% the limits $Ku(z)$, $Ku(w^\pm)$ exist -- see Proposition
% \ref{prop:infvortsing}. Recall that $MV^K_{g,n}(X)$ denotes the space
% of $K$- vortices from this family of neck-stretched Riemann surfaces
% to $X$ modulo the action of the gauge group.
%
We now rigorously define stable vortices. A {\it pre-stable curve} $C$ is a
nodal curve with $n$-marked points such that any unstable component of
$C$ is a sphere with two special points. Suppose $C$ is modeled on the
graph $\Gamma$. Denote by $\Gamma^{cyl}$ the graph induced by the
unstable components $\vertex(\Gamma)\bs
\vertex(\st(\Gamma))$. $\Gamma^{cyl}$ is a collection of paths, each
path is of one of the following form (corresponding to
\eqref{part:markbubble} and \eqref{part:nodebubble} above).
\begin{enumerate}
\item A path may be attached to the graph $\Gamma\bs \Gamma^{cyl}$ at
  one end and the other end of the path has a marked point $z$. We denote
  this path by $\Gamma^{cyl}_z$.
\item A path may be attached to $\Gamma \bs \Gamma^{cyl}$ at both
  ends. Contracting the path produces a node $w$ in $\st(\Gamma)$. We
  denote this path by $\Gamma^{cyl}_w$.
\end{enumerate}
For our convenience, we label the vertices in these paths as
$\alpha_{x,1},\dots,\alpha_{x,|C(x)|}$ for any $x \in
\Edge(\st(\Gamma)) \cup \Edge_\infty(\st(\Gamma))$ satisfying the
following convention. If $x$ is a marked point in $\st(\Gamma)$, then $x$ lies on
the vertex $\alpha_{x,|C(x)|}$ and the vertex $\alpha_{x,1}$ has an edge
to $\Gamma\bs \Gamma^{cyl}$. If $x=x^+x^-$ is an edge in $\st(\Gamma)$,
then $\alpha_{x,1}$ (resp. $\alpha_{x,|C(x)|}$) is attached to
$\Gamma\bs \Gamma^{cyl}$ at $x^+$ (resp. $x^-$).

Given a pre-stable curve $C$, a Riemann surface $\Sig$ with
cylindrical ends can be associated to it. As earlier, as a complex
curve $\Sig$ is just $C$ with all special points deleted. On the
stable components of $C$, we use the neck-stretching metric
corresponding to $\st(C)$. For any unstable component $C_1$ with marked
points $z_0$, $z_1$, we use the pull-back metric $\rho^*(\d r^2 +
\d\theta^2)$, where $\rho$ is a bi-holomorphism
$$\rho:C_1 \bs \{z_0,z_1\} \to \{(r,\theta):r \in \R, \theta \in S^1\}$$
such that $\lim_{r \to -\infty}\rho^{-1}(r,\theta)=z_0$ and $\lim_{r
  \to \infty}\rho^{-1}(r,\theta)=z_1$. The map $\rho$ is unique up to
translation by a constant $(r_0,\theta_0)$, hence the pull-back metric
is uniquely determined.

\begin{definition}[Stable vortex] 
  Let $C$ be a pre-stable $n$-pointed curve of genus $g$. A {\em
    stable vortex} on $C$ has domain $\Sig$, which is a Riemann
  surface with cylindrical ends corresponding to $C$. A stable vortex
  on $C$ is a tuple of finite energy vortices
  $(A_\alpha,u_\alpha)_{\alpha \in \vertex(\Gamma(C))}$ defined on the
  principal $K$-bundles $\Sig_\alpha\times K$ that satisfy the
  connectedness condition :
  \begin{equation*}
    Ku_{\inc(w^+)}(w^+)=Ku_{\inc(w^-)}(w^-) \quad \forall w \in \Edge(\Gamma(C)).
  \end{equation*}
  In addition, for an unstable component $\alpha \in \Gamma^{cyl}$,
  the vortex $(A_\alpha,u_\alpha)$ has non-zero energy.
\end{definition}
To define convergence on a sequence of Riemann surfaces, these
surfaces must be identified to each other. For a Riemann surface
$\Sig$ with cylindrical ends, we denote by $\Cyl(\Sig)$ the parts that
are isometric to parts of a cylinder.

\begin{lemma}[Identification of curves in a neighborhood of a stable
  curve $C$] \label{lem:convprep} Suppose $C$ is an $n$-pointed stable
  nodal genus $g$ curve and $\Def(C)$ be a universal deformation
  parametrized by $(S,0)$, where $S$ is a neighborhood of the origin
  in $C^N$.  Let $\Sig_s$ be the Riemann surface with cylindrical ends
  corresponding to the curve $C_s$, $s \in S$. For a small enough
  neighborhood $S$, for any $\alpha \in \vertex(C)$, there is a family
  of connected open subsets $\tSig_{s,\alpha} \subset \Sig_\alpha$ and
  diffeomorphisms
  \begin{equation}\label{eq:curveid}
    h_s:\bigl(\bigsqcup_\alpha \tSig_{s,\alpha}\bigr)/\sim \to \Sig_s
  \end{equation}
  where $\sim$ is an isometric equivalence relation defined on
  $\sqcup_\alpha\Cyl(\tSig_{s,\alpha})$. The map
  \eqref{eq:curveid} is an isometry when restricted to
  $\Cyl(\tSig_{s,\alpha})$ on the l.h.s. and $\Cyl(\Sig_s)$ on the
  r.h.s. For $s=0$, $\tSig_{0,\alpha}=\Sig_{0,\alpha}$ and $h_0=\Id$.
\end{lemma}
\begin{proof} Suppose the nodal curve $C$ is of type $\Gamma$ and
  $\Def_\Gamma(C)$ be a universal deformation of type $\Gamma$
  parametrized by $(S_\Gamma,0)$. Then, there is a smooth family of
  diffeomorphisms
  \begin{equation}\label{eq:defdiffeo}
    h_s:C_s \to C_0, \quad s \in S_\Gamma,\quad \quad h_0=\Id
  \end{equation}
  satisfying the condition that $h_s$ preserves special
  points. Further, $h_s$ can be chosen such that in a punctured
  neighborhood $N(z)\bs\{z\}$ of any special point $z=z_j,w_j^\pm$,
  $h_s$ is an isometry under the neck-stretching metric. Set
  $\tSig_{s,\alpha}:=\Sig_\alpha$ and $\sim$ to be trivial. This
  proves the lemma for $s \in S_\Gamma$.

  Next, consider $s \in S$ such that $C_s$ can be obtained from
  $C_0\simeq C$ by the gluing procedure.  Suppose $\delta(s)$ is the
  gluing parameter used to produce $C_s$ from $C$. We write
  $\delta(s)=(\delta_w(s))_{w \in \Edge(C)}$, where $\delta_w \in
  T_{w^+}\tilde C_{\inc(w^+)}\tensor T_{w^-}\tilde C_{\inc(w^-)}$ and
  $\tilde C$ denotes the normalization of $C$. Let
  \begin{equation}\label{eq:breakglue}
    \tSig_{s,\alpha}=\Sig_\alpha \bs\bigl( \bigcup_{w^\pm:\inc(w^\pm)=\alpha} (\rho_{w^\pm})^{-1}\{r+\img\theta:\pm r > \on{Re}(l_w(s)), \theta \in S^1\}\bigr),
  \end{equation}
  where $l_w(s)=-\ln \delta_w(s)$. Then, there is an isometry
  $h_s:\Sig_s \simeq (\bigsqcup_{\alpha \in
    \vertex(C)}\tSig_{s,\alpha})/\sim$, where $\sim$ is an equivalence
  relation like in \eqref{eq:lid} corresponding to every edge in
  $\Gamma(C)$. Note that $l_w(s) \to \infty$ as $s \to 0$.

  Finally, we consider a general $s \in S$. There is an $s' \in
  S_\Gamma$ such that $C_s$ is obtained from $C_{s'}$ by the gluing
  procedure. Apply the discussion is the previous paragraph, with
  $C_0$ replaced by $C_{s'}$. The resulting isometry can be composed
  with $h_{s'}$ in \eqref{eq:defdiffeo} to obtain the result of the
  lemma.
\end{proof}
The benefit of the construction in Lemma \ref{lem:convprep} is that,
now $\tSig_{s,\alpha}$ is isometrically embedded in  $(\Sig,\ulg_s)$.
% Denote by $j_\nu$ the complex structure associated to $g_\nu$.
A gauged holomorphic map $w$ defined on $\Sig_s$ can be pulled back to
one on $\tSig_{s,1}\cup \dots \cup \tSig_{s,k}$ and restricted to
$\tSig_{s,\alpha}$ -- this object is also denoted by $w$. So, it is
possible to talk about convergence of a sequence of vortices defined
on $\Sig_{s(\nu)}$, where the sequence $s(\nu) \to 0$ as $\nu \to
\infty$. The limit will be a stable vortex on $\Sig$. For simplicity
of exposition, we assume the curves $C_{s(\nu)}$ are smooth in Definition \ref{def:gromconv}.
\begin{definition}[Gromov Convergence]\label{def:gromconv} Suppose $\{C_\nu\}_{\nu \in
    \N}$ be a sequence of smooth $n$-pointed curves of genus $g$ and
  $C$ is a pre-stable nodal curve such that the sequence $[C_\nu]$ converges to  $[\st(C)]$ in
  $\ol M_{g,n}$. Let $\Sig_\nu$, $\Sig$ be Riemann surfaces with
  cylindrical ends corresponding to $C_\nu$, $C$. Let
  $v_\nu:=(A_\nu,u_\nu)$ be a sequence of vortices on the trivial
  bundle $\Sig_\nu \times G$. We say the sequence {\em Gromov
    converges} to a stable vortex $v:=((v_\alpha)_{\alpha \in
    \vertex(C)},(z_j)_{1 \leq j \leq n})$ on $\Sig$ if on a deformation of the
  curve $\st(C)$ parametrized by $(S,0)$ one can find a sequence in
  $S$, $s(\nu) \to 0$, such that there are isomorphisms $C_\nu \simeq
  C_{s(\nu)}$ and the following are satisfied
  % if there exist s(\nu) such that we can make .. identifications
  % after which following are satisfied.
  \begin{enumerate}
  \item {\rm{(Convergence on stable domain
        components)}} \label{part:1gc} For each $\alpha \in
    \vertex(\st(C))$, there exist a sequence of gauge transformations
    $k_{\nu,\alpha}$ on $\tSig_{\nu,\alpha}$ so that $k_{\nu,\alpha}
    (A_{\nu,\alpha},u_{\nu,\alpha})$ converges in $C^\infty$ to
    $(A_\alpha,u_\alpha)$ on compact subsets of $\Sig_\alpha$. Here,
    $\tSig_{\nu,\alpha}:=\tSig_{s(\nu),\alpha}$ defined in Lemma
    \ref{lem:convprep}.
  \item \label{part:2gc} {\rm{(Convergence on cylindrical bubbles at
        marked points)}} For every $\alpha \in \Gamma_z^{cyl}$ for
    some $z \in \Edge_\infty(\st(C))$, there exist
    \begin{itemize}
    \item a sequence of numbers $s_{\nu,\alpha} \to \infty$ which
      define a sequence of translations
      \begin{align*}
        \phi_{\nu,\alpha} : ((-s_{\nu,\alpha},\infty) \times S^1)_\alpha &\to (S^1 \times \R_{\geq 0})_z,\\
        z&\mapsto z+s_{\nu,\alpha}
      \end{align*}
      mapping a part of $\Sig_\alpha$ to the cylindrical end
      corresponding to $z$.
    \item a sequence of gauge transformations $k_{\nu,\alpha}$ so that
      $k_{\nu,\alpha}(\phi_{\nu,\alpha}^*(A_\nu,u_\nu))$ converges to
      $(A_\alpha,u_\alpha)$ on compact subsets of $(S^1\times
      \R)_\alpha$.
    \end{itemize}
  \item \label{part:3gc} {\rm{(Convergence on cylindrical bubbles at
        nodes)}} For every $\alpha \in \Gamma_w^{cyl}$ for some $w \in
    \Edge(\st(C))$, there exist
    \begin{itemize}
    \item a sequence of numbers $0<s_{\nu,\alpha}<l_w(\nu)$ such that
      $s_{\nu,\alpha}$, $l_w(\nu)-s_{\nu,\alpha}\to \infty$ as $\nu
      \to \infty$, which define a sequence of translations
      \begin{align*}
        \phi_{\nu,\alpha} : ((-s_{\nu,\alpha},l_w(\nu)-s_{\nu,\alpha}) \times S^1)_\alpha&\to  ((0,l_w(\nu)) \times S^1)_{w^+}\\
        z&\mapsto z+s_{\nu,\alpha}
      \end{align*}
      mapping a part of $\Sig_\alpha$ to the cylindrical end
      corresponding to $w^+$. The notation $l_w(\nu)$ is same as that in the proof of Lemma \ref{lem:convprep}.
    \item a sequence of gauge transformations $k_{\nu,\alpha}$ so that
      $k_{\nu,\alpha}(\phi_{\nu,\alpha}^*(A_\nu,u_\nu))$ converges to
      $(A_\alpha,u_\alpha)$ on compact subsets of $(S^1\times
      \R)_\alpha$.
    \end{itemize}
  \item{\rm{(Rescaling)}} \label{part:4gc} Let $\alpha$, $\beta \in
    \Gamma_\gamma^{cyl}$ be distinct unstable component connected to $\gamma$, which is either a node or marked point in $\st(C)$. Suppose $\alpha <
    \beta$, then $s_{\nu,\beta}-s_{\nu,\alpha} \to \infty$ as $\nu \to
    \infty$.
  \item{\rm{(Energy)}} \label{part:5gc} $\lim_{\nu \to
      \infty}E(A_\nu,u_\nu)=\sum_{\alpha \in
      \on{Vert}(\Gamma)}E((A_\alpha,u_\alpha))$.
  \end{enumerate}
\end{definition}

\begin{remark} The definition of Gromov convergence is independent of
  the choice of $h$ in \eqref{eq:defdiffeo}.
\end{remark}

The main Theorem is
\begin{theorem}[Gromov compactness]\label{thm:gromovcpt} Suppose
  $\{C_\nu\}_{\nu \in \N}$ are smooth $n$-pointed genus $g$ curves and $C'$ is
  a stable nodal curve such that $[C_\nu] \to [C']$ in $\ol
  M_{g,n}$. Suppose $v_\nu=(A_\nu,u_\nu)$ is a sequence of vortices defined
  on $\Sig_\nu$, the Riemann surface with cylindrical ends
  corresponding to $C_\nu$, that satisfy
  $$\sup_\nu E(v_\nu)<\infty,$$
  and there is a compact subset of $X$ containing the images of
  $u_\nu$.  Then, after passing to a subsequence, the sequence
  Gromov-converges to a stable vortex $v=(A,u)$ on $\Sig$, which is a Riemann surface with cylindrical ends corresponding to a pre-stable curve $C$, such that $\st(C)=C'$.
\end{theorem}
Theorem \ref{thm:gromovcpt} follows from Proposition
\ref{prop:convmodbreak}, Proposition \ref{prop:markedbreak} and
Proposition \ref{prop:nodebreak}, they respectively prove conditions
\eqref{part:1gc}, \eqref{part:2gc} and \eqref{part:3gc} in Gromov
Convergence. The other two conditions are ensured along the way.  The
next Proposition proves that evaluation maps are continuous under the
Gromov topology, and the equivariant homology class $[u_\nu]$ is
preserved in the limit.
\begin{proposition}[Continuity of $\ev_j$]\label{prop:evcont} Assume
  the setting of Theorem \ref{thm:gromovcpt}.
  \begin{enumerate}
  \item \label{part:evlt} For each of the marked points
    $z_1,\dots,z_n$, $\lim_{\nu \to \infty}\ev_j(A_\nu,u_\nu)=\ev_j(A,u).$
\item \label{part:homology} If $[u_\nu]=\beta \in H_2^K(X)$ for all
  $\nu$, then $[u]=\beta$.
\end{enumerate}
\end{proposition}
The proof of this Proposition appears in Section \ref{sec:convpf}.

\begin{remark} We have defined Gromov convergence and stated the
  Gromov compactness theorem in the case when the sequence of vortices
  $(A_\nu, u_\nu)$ is defined on smooth curves and the limit stable
  vortex is defined on a possibly nodal curve. The definition and
  result can both be generalized to the case when
  $(A_\nu,u_\nu)$ are stable vortices on possibly nodal curves by applying the same ideas component-wise. After
  proving Theorems \ref{thm:gromovcpt} and \ref{prop:evcont}, we will
  have a compact space $\ol {MV}_{g,n}^K(X,\beta)$ of stable vortices
  modulo gauge transformations.
\end{remark}

\section{Proof of Gromov compactness}\label{sec:convpf}
\subsection{Convergence modulo breaking of cylinders}
\begin{proposition}[Convergence of vortices modulo breaking of
  cylinders]\label{prop:convmodbreak} Let $\Sig_\nu$, $\Sig'$ be Riemann surfaces with cylindrical ends corresponding to curves $C_\nu$, $C'$ respectively in Theorem \ref{thm:gromovcpt}. Let
  $v_\nu:=(A_\nu,u_\nu)$ be a sequence of vortices on the trivial
  bundle $\Sig_\nu\times G$ that satisfy
  $$\sup_\nu E(v_\nu)<\infty.$$ 
  The images of $u_\nu$ are contained in a compact subset of $X$.
  
  After passing to a subsequence, there exists a vortex $v:=(A,u)$ on
  $\Sig'$ (possibly not satisfying the connectedness condition
  \eqref{eq:connect}) and a sequence of gauge transformations $k_\nu
  \in \K(P)$ such that $k_\nu(A_\nu,u_\nu)$ converges to $(A,u)$ in
  $C^\infty$ on compact subsets of $\Sig'$. Further,
$$\lim_{\nu \to \infty}E(v_\nu)=E(v) + \sum_{z \in \Edge_\infty(\Sig')}E_z + \sum_{w \in \Edge(\Sig')}^m E_w,$$
where $E_z$ and $E_w$ are defined as
\begin{align*}
  E_z&:=\lim_{R \to \infty}\lim_{\nu \to \infty} E((A_\nu,u_\nu),\rho_z^{-1}\{s>R\}), \\
  E_w&:=\lim_{R \to \infty}\lim_{\nu \to \infty}
  E((A_\nu,u_\nu),\rho_w^{-1}\{R<s<l_w(\nu)\}).
\end{align*}
\end{proposition}
The proof of this Theorem is analogous to proof of Proposition 37 in
\cite{Zilt:QK}. It uses a combination of Uhlenbeck compactness on
non-compact domains (Theorem A' in \cite{Weh:Uh}) and elliptic
regularity on $u_\nu$. The only difference here is that the metric on
$\Sig$ is not fixed -- but this does not create any changes because
Uhlenbeck compactness (\cite{Weh:Uh}) and elliptic regularity for
pseudoholomorphic curves (Proposition B.4.2 in \cite{MS}) are valid
when we have a sequence of converging metrics.

\subsection{Breaking of cylinders}

In Proposition \ref{prop:convmodbreak}, $E_z$ and $E_w$ represent
energy that escapes to infinity on cylindrical ends. This leads to
breaking of cylinders. To analyze that, we need quantization of energy
for vortices on a cylinder and that the ends of the cylinders connect
in the image.
\begin{proposition}\label{prop:energyquantvort} For any compact set $S
  \subset X$, there exists constants $C$, $\delta$, $E_C>0$ that
  satisfy the following:
  \begin{enumerate}
  \item \label{part:ann} {\rm (Annulus Lemma)} For a finite energy
    vortex $(A,u)$ on the cylinder $\R \times S^1$ whose image is
    contained in $S$ and any $s_0, s_1 \in \R$ such that
    $E:=E((A,u),[s_0,s_1] \times S^1)<E_C$, then for any $T>0$
    \begin{align*}
      E((A,u), [s_0+T,s_1-T]\times S^1) &\leq Ce^{-\delta T}E\\
      \sup_{z,z' \in [s_0+T,s_1-T]\times S^1 }d(Ku(z),Ku(z')) &\leq
      Ce^{-\delta T}\sqrt E
    \end{align*}
  \item \label{part:quant} {\rm (Quantization of energy for vortices)}
    If $E((A,u), \R \times S^1)<E_C$. Then, $E(A,u)=0$.
  \end{enumerate}
\end{proposition}
\begin{proof}
  The proof of \eqref{part:ann} is identical to the proof of the
  corresponding result for vortices in $\C$ in \cite{Zilt:QK}.  To
  prove \eqref{part:quant} consider a vortex $(A,u)$ on $\R \times
  S^1$ satisfying $E(A,u) \leq E_C$. For any $s_0<s_1$ and $T>0$, by
  part \eqref{part:ann},
  $$E((A,u), [s_0,s_1]\times S^1) \leq Ce^{-\delta T} E((A,u), [s_0-T,s_1+T]\times S^1) \leq Ce^{-\delta T} E_C.$$
  It follows that $E(A,u,[s_0,s_1])=0$. Since the choice of $s_0$,
  $s_1$ is arbitrary, $E(A,u)=0$.
\end{proof}

\begin{lemma} [Breaking of cylinders]\label{lem:break} Suppose
  $v_\nu:=(A_\nu,u_\nu)$ is a sequence of vortices on
  $\Sig_0=[0,\infty)\times S^1$ whose images are contained in a
  compact subset of $X$. Assume further that the sequence $v_\nu$
  $C^\infty$-converges to $v:=(A,u)$ in compact subsets of
  $[0,\infty)\times S^1$. Suppose
  \begin{equation}\label{eq:m0def}
    m_0:=\lim_{R \to \infty}\on{limsup}_{\nu \to \infty}E((A_\nu,u_\nu),[R,\infty) \times S^1)>0.
  \end{equation}
  Then, after passing to a sub-sequence, there exists a finite energy
  non-constant vortex $(A_1,u_1)$ on a cylinder $\Sig_1=(\R \times
  S^1)$, a sequence of translations $\phi_\nu:[0,\infty)\times S^1 \to
  \R \times S^1$ given by $r+\img\theta \mapsto r+s_\nu+\img \theta$
  with $\lim_\nu s_\nu = \infty$ such that the following are
  satisfied:
  \begin{enumerate}
  \item \label{part:conv} Let
    $(A^1_\nu,u^1_\nu):=\phi^*_\nu(A_\nu,u_\nu)$. There is a sequence
    of gauge transformations $k_\nu$ so that $k_\nu (A^1_\nu,u_\nu^1)$
    converges to a non-constant vortex $(A^1,u^1)$ in $C^\infty$ on
    compact subsets of $\R \times S^1$.
  \item \label{part:energy} Let $m_1:=\lim_{R \to
      \infty}\on{limsup}_{\nu \to
      \infty}E((A^1_\nu,u^1_\nu),[R,\infty) \times S^1)$. Then,
   $$m_0=E(A^1,u^1)+m_1.$$
 \item \label{part:connect} $Ku(\infty)=Ku^1(-\infty)$.
 \end{enumerate}
Suppose $s_\nu' \to \infty$ be another sequence such that \eqref{part:conv}-\eqref{part:connect} are satisfied when $s_\nu$ is replaced by $s_\nu'$. Then, the sequence $|s_\nu-s'_\nu|$ is bounded.
\end{lemma}

\begin{proof} The proof runs parallel to the proof of Proposition
  4.7.1 in \cite{MS}. Pick $0< \delta < \min\{E_C, m_0\}$ (from
  Proposition \ref{prop:energyquantvort}). Define the sequence $s_\nu$
  as
  \begin{equation}\label{eq:snudef}
    E(v_\nu, [s_\nu,\infty) \times S^1)=m_0-\frac \delta 2.
  \end{equation}
  {\sc Step 1}: {\em $\lim_{R \to \infty}\lim_{\nu \to \infty}E(v_\nu,[s_\nu-R,\infty)\times S^1)=m_0.$ }\\
  We observe that \eqref{eq:m0def} implies that for any sequence
  $R_\nu \to \infty$,
  \begin{equation*}
    \lim_{\nu \to \infty}E(v_\nu,[R_\nu,\infty)\times S^1) \leq m_0
  \end{equation*}
  and there exists a sequence $\sig_\nu \to \infty$ for which
  equality is attained. i.e.
  \begin{equation}\label{eq:signudef}
    \lim_{\nu \to \infty}E(v_\nu,[\sig_\nu,\infty)\times S^1)= m_0.
  \end{equation}
  This implies for every $T \geq 0$,
  \begin{equation}\label{eq:signusphere}
    \lim_{\nu \to \infty} E(v_\nu,[\sig_\nu-T,\infty)\times S^1)= m_0.
  \end{equation}
  The result of step 1 would be true if $s_\nu-\sig_\nu$ is bounded
  above. So, we assume $\lim_\nu(s_\nu-\sig_\nu)=\infty$. From
  \eqref{eq:snudef} and
  \eqref{eq:signudef},\begin{equation}\label{eq:signusnu} \lim_{\nu
      \to \infty}E(v_\nu,[\sig_\nu,s_\nu)\times S^1)= \frac \delta 2.
  \end{equation}
  We split up the cylinders $[\sig_\nu,s_\nu]\times S^1$ into 3 parts
  and show that in the limit all the energy is focused in the part at
  the $s_\nu$ end. First, we handle the middle part. By the annulus
  lemma (Proposition \ref{prop:energyquantvort} \eqref{part:ann}),
  there exist constants $c$, $\delta >0$ so that for any $T > 0$,
  \begin{equation}\label{eq:neckenergybd}
    \lim_{\nu \to \infty}E(v_\nu,[\sig_\nu+T,s_\nu-T]\times S^1) \leq \frac \delta 2 ce^{-\delta T}
  \end{equation} 
  For the part at the $\sig_\nu$-end, we prove the following.
  \begin{claim} $$\lim_{\nu \to
      \infty}E(v_\nu,[\sig_\nu,\sig_\nu+T]\times S^1)=0$$ for all $T
    \geq 0$.
  \end{claim}
  \begin{proof} Consider the sequence of re-scaled vortices $\tilde
    v_\nu(z):=v_\nu(z+\sig_\nu)$. Then, for all $T \geq 0$,
    \begin{align*} \lim_{\nu \to \infty}E(\tilde v_\nu,[-T,T]\times S^1)&=\lim_{\nu \to \infty} E(v_\nu,[\sig_\nu-T,\sig_\nu+T]\times S^1)\\
      &\leq \lim_{\nu \to \infty} E(v_\nu,[\sig_\nu-T,s_\nu]) \\
      &\leq \frac \delta 2,
    \end{align*}
    where the last inequality follows from \eqref{eq:snudef} and
    \eqref{eq:signusphere}. By quantization of energy (Proposition
    \ref{prop:energyquantvort}), the above limit is $0$. So, the claim
    follows.
  \end{proof}
  Together with \eqref{eq:signusnu} and \eqref{eq:neckenergybd}, this
  implies, for all $T \geq 0$
  \begin{equation}
    \lim_{\nu \to \infty}E(v_\nu,[s_\nu-T,s_\nu]) \geq \frac \delta 2(1-ce^{-\delta T}).
  \end{equation}
  Step 1 is proved because, for all $T \geq 0$
  \begin{equation*}
    m_0 \geq \lim_{\nu \to \infty}E(v_\nu,[s_\nu-T,\infty)\times S^1) \geq m_0 -\frac {\delta ce^{-\delta T}} 2.
  \end{equation*}
 
  {\sc Step 2}: {\em Proof of \eqref{part:conv}-\eqref{part:connect}.}\\
  Let $v^1_\nu(z):=v_\nu(z+s_\nu)$, then Step 1 implies that
 $$\lim_{R \to \infty}\lim_{\nu \to \infty}E(v^1_\nu, (-R,\infty)\times S^1)=m_0.$$
 Repeating the steps in the proof of Proposition
 \ref{prop:convmodbreak}, parts \eqref{part:conv} and
 \eqref{part:energy} of the Proposition follow. $v^1$ is non-constant,
 because $E(v^1,(-\infty,0])=\lim_\nu E(v^1_\nu,(-\infty,0]\times
 S^1)=\frac \delta 2>0$. Now we come to \eqref{part:connect}. Define
   $$E(R)=\lim_{\nu \to \infty}E(u_\nu,[R,s_\nu-R]\times S^1).$$
   Then,
   $$E(R)=E(u,(R,\infty)\times S^1) + E(v(-\infty,-R)).$$
   So, $\lim_{R\to \infty}E(R)=0$.  For large enough $R$, apply the
   annulus lemma on $u_\nu$ on the cylinder $[R-1,s_\nu-R+1]$ with
   $T=1$. Then,
   \begin{align*}
     \sup_{z,z' \in [R,s_\nu-R]\times S^1}d(Ku_\nu(z),Ku_\nu(z')) \leq
     Ce^{-\delta}\sqrt {E_\nu(R-1)}
   \end{align*}
   Taking limit $\nu \to \infty$,
   $$\sup_{\on{Re}(z)\geq R,\on{Re}(z') \leq -R}d(Ku(z),Ku^1(z')) \leq Ce^{-\delta}\sqrt {E(R-1)}.$$
   Letting $R \to \infty$, we get $Ku(\infty)=Ku^1(-\infty)$.

   {\sc Step 3:} Proof of boundedness of $|s_\nu-s_\nu'|$.\\
   It is not possible to have a subsequence for which $\lim_\nu
   (s_\nu'-s_\nu)=\infty$, because then the energy of the bubble
   $(A_1,u_1)$ (which is non-zero) would be lost. So, lets assume
   there is a subsequence such that $\lim_\nu (s_\nu-s_\nu')=\infty$. In
   Step 1, we showed $\lim_\nu(s_\nu-\sig_\nu)$ is finite, so
   $s'_\nu<\sig_\nu$. Therefore,
 $$\lim_{\nu \to \infty}E(v_\nu,[s'_\nu,\infty)\times S^1)= m_0.$$ 
 By the same arguments as were used for $\sig_\nu$, we can say
 $\lim_\nu (s_\nu-s'_\nu)$ is finite. 
 \end{proof}

 Lemma \ref{lem:break} captures the first cylindrical bubble
 $(A_1,u_1)$. If $m_1$ is positive, then the process can be repeated
 to capture the next bubble. Each cylindrical bubble has energy at
 least $E_C$, by Proposition \ref{prop:energyquantvort}. So, after a
 finite number of steps the process terminates. Thus we have proved
 the following.
 \begin{proposition}[Breaking of cylinders at marked
   points]\label{prop:markedbreak} Suppose $v_\nu=(A_\nu,u_\nu)$ is a
   sequence of vortices on the trivial $K$ bundle on the cylinder
   $\Sig:=[0,\infty)\times S^1$ whose images are contained in a
   compact subset of $X$ and whose energies are bounded. 
   % We assume the sequence converges to a finite energy vortex
   % $v=(A,u)$ smoothly on compact subsets of $[0,\infty)\times S^1$
   % and the following limit exists.
   % $$m_0:=\lim_{R \to \infty}\on{limsup}_{\nu \to \infty}E((A_\nu,u_\nu),[R,\infty) \times S^1).$$
Then, there is a stable vortex $(A,u)$ modeled on the graph $\{\Sig\}
\cup \Gamma^{cyl}$ so that a subsequence of $(A_\nu,u_\nu)$
Gromov-converges to $(A,u)$, except there is no stability requirement
on $\Sig$. Here $\Gamma^{cyl}$ consists of a path of $|\Gamma^{cyl}|$
cylindrical vertices $\{\alpha_1,\dots,\alpha_{\Gamma^{cyl}}\}$ and
$\alpha_1$ is connected to $\Sig$.
\end{proposition}

For the set-up of the next Proposition, suppose $l_\nu=L_\nu+it_\nu
\in \R_{\geq 0} \times S^1$ be a sequence such that $\lim_{\nu \to
  \infty}L_\nu=\infty$. Let $\Sig_\nu^+=[0,L_\nu]\times S^1$ and
$\Sig_\nu^-=[-L_\nu,0]\times S^1$ identified to each other by the
isomorphism
$$r_\nu:\Sig_\nu^- \to \Sig_\nu^+,\quad z\mapsto z+l_\nu.$$
Let $\Sig^+$ and $\Sig^-$ denote $\R_{\geq 0} \times S^1$ and $\R_{\leq 0}
\times S^1$ respectively.

\begin{proposition} [Breaking of cylinders at nodal
  points] \label{prop:nodebreak} Suppose $v_\nu:=(A_\nu,u_\nu)$ be a
  sequence of vortices on the trivial bundle $\Sig_\nu^+\times K$
  whose images are contained in a compact subset of $X$. 
  Then, there is a stable vortex $(A,u)$ modeled on the tree
  $\{\Sig^+,\Sig^-\} \cup \Gamma^{cyl}$ so that a subsequence of
  $(A_\nu,u_\nu)$ Gromov-converges to $(A,u)$, except there is no
  stability requirement on $\Sig^+$ and $\Sig^-$. Here, $\Gamma^{cyl}$
  is a path connecting $\Sig^+$ and $\Sig^-$ with nodes
  $\{\alpha_1,\dots,\alpha_{|\Gamma^{cyl}|}\}$.
\end{proposition}

\begin{proof} 
The first step is to show convergence modulo bubbling on the surfaces $\Sig^\pm$.
By Proposition \ref{prop:convmodbreak}, there are sequences of gauge transformations $k_\nu^\pm$ such that $k_\nu^+ v_\nu$ (resp. $k_\nu^-(r_\nu^*v_\nu$)) converges to a finite energy vortex $v^+$ (resp. $v^-$) smoothly on compact subsets
of $\Sig^+$ (resp. $\Sig^-$). The following limit 
  $$m_0:=\lim_{R \to \infty}\on{limsup}_{\nu \to \infty}E((A_\nu,u_\nu),[R,L_\nu) \times S^1).$$
exists and satisfies
  \begin{equation}\label{eq:energylt}
    \lim_{\nu \to \infty}E(v_\nu,[0,L_\nu])=E(v_{\Sig^+}) + m_0.
  \end{equation}
  The rest of the proof of the Proposition is by induction on the integer $\lfloor \frac {m_0} {E_C}
  \rfloor$, where $E_C$ is from the Proposition
  \ref{prop:energyquantvort}. The number $\lfloor \frac {m_0} {E_C}
  \rfloor$ is an upper bound on the number of vertices in $\Gamma^{cyl}$. If $m_0>0$, as in the proof of Lemma
  \ref{lem:break}, pick $0< \delta < \min\{E_C, m_0\}$ and define a
  sequence $0<s_\nu<L_\nu$ that satisfies
$$E((A_\nu,u_\nu), [s_\nu,\infty) \times S^1)=m_0-\frac \delta 2.$$ 

We divide our analysis into three cases. In the first case there are cylindrical bubbles, i.e. $\Gamma^{cyl} \neq \emptyset$. In the second case, there are no cylindrical bubbles, and the energy $m_0$ is contained in the limit vortex $v^-$ defined on $\Sig^-$. In the third case, where $m_0=0$, there are no cylindrical bubbles and the limit vortex on $\Sig^-$ is a constant vortex.

{\em Case 1: $m_0>0$ and $\lim_{\nu \to \infty}(L_\nu-s_\nu)=\infty$.}\\
In this case, the proof of Lemma \ref{lem:break} is applicable. So,
the conclusions carry over: Modulo gauge, the sequence
$v^1_\nu:=v_\nu(\cdot + s_\nu)$ converges to a vortex $v^1:=(A^1,u^1)$
defined on $S^1 \times \R$.
\begin{equation}\label{eq:breakupm0}
  m_0=E(v^1)+m_1, \quad \text{where }m_1:=\lim_{R \to \infty}\on{limsup}_{\nu \to \infty}E(v^1_\nu,[R,L_\nu-s_\nu) \times S^1).
\end{equation}
The cylindrical vortex $v^1$ is non-constant and satisfies $Ku^+(\infty)=u^1(-\infty)$. It is the first cylindrical vortex in the path
connecting $\Sig^+$ to $\Sig^-$ and is represented by the vertex $\alpha_1 \in \Gamma^{cyl}$.  We set
$s_{\nu,\alpha_1}:=s_\nu$. Pick $R_0'$ so that
$E(v^1,\{s<R_0'\})>E_C$. By the induction hypothesis, the Proposition
is true for the sequence of vortices $v^1_\nu$ defined on
$[R_0',L_\nu-s_\nu]$. This would give us a path of cylindrical
vertices connecting $\Sig_{\alpha_1}$ to $\Sig^-$ -- we label the
vertices $\{\alpha_2,\dots,\alpha_{|\Gamma^{cyl}|}\}$. We have a
sequence of translation maps $\Sig_{\alpha_i} \to \Sig_{\alpha_1}$ for
$i>1$. These translation maps can be composed with the sequence
$$\phi_{\nu,\alpha_1}:((-s_{\nu,\alpha_1},L_\nu-s_{\nu,\alpha_1}) \times S^1)_{\alpha_1} \to ((0,L_\nu)\times S^1)_{\Sig^+}, \quad z \mapsto z+s_{\nu,\alpha_1}$$
to yield $\phi_{\nu,\alpha_i}$ for $i=2,\dots, |\Gamma^{cyl}|$. From the induction hypothesis, we obtain
$$m_1=\sum_{2 \leq i \leq |\Gamma^{cyl}|}E(v_i)+ E(v_{\Sig^-}).$$
Combined with \eqref{eq:breakupm0} and \eqref{eq:energylt}, this
establishes the energy equality for Gromov convergence.

{\em Case 2: $m_0>0$ and $\lim_{\nu \to \infty}(L_\nu-s_\nu)=L<\infty$.}\\
In this case $\Gamma^{cyl}=\emptyset$. The sequences
$v_\nu(\cdot+s_\nu)$ and $v_\nu(\cdot+l_\nu)$ (recall
$l_\nu=L_\nu+it_\nu$) will have the same limit up to
reparametrization. But $v_\nu(\cdot+l_\nu)=r_\nu^*v_\nu$. Imitating
the proof of Lemma \ref{lem:break}, we get $m_0=E(v^-)$ and
$Ku^+(\infty)=Ku^-(-\infty)$.

{\em Case 3: $m_0=0$}\\
If $m_0=0$, we have
$$\lim_{\nu \to \infty}E(r_\nu^*v_\nu,[0,1]\times S^1)=\lim_{\nu \to \infty}E(v_\nu,[L_\nu-1,L_\nu]\times S^1)=0.$$
So, $E(v^-)=0$, and $u^-$ maps to $Ku^+(\infty)$.
\end{proof}

We next prove Proposition \ref{prop:evcont} : this is about the
continuity of the evaluation map and preservation of homology class
under Gromov convergence.
\begin{proof}[Proof of Proposition \ref{prop:evcont}
  \eqref{part:evlt}]
  We focus on a marked point $z_0$ and a sequence of vortices
  $v_\nu=(A_\nu,u_\nu)$ defined on trivial $K$-bundles on cylinders
  $[0,\infty)\times S^1$. The Proposition is equivalent to showing --
  `Given $v_\nu$ converges to $v=(A,u)$ in $C^\infty$ on compact
  subsets of $[0,\infty)\times S^1$ and
  \begin{equation}\label{eq:noenergyescape}
    \lim_{R \to \infty}\lim_{\nu \to \infty}E(v_\nu, [R,\infty) \times S^1)=0,
  \end{equation}
  then $\lim_{\nu \to \infty}Gu_\nu(\infty)=Gu(\infty)$.'  First, we
  observe that the limits $Ku_\nu(\infty)$, $Ku(\infty)$ exist using
  the removal of singularity theorem (Proposition
  \ref{prop:infvortsing}). Next, using \eqref{eq:noenergyescape} and
  the annulus lemma (Proposition \ref{prop:energyquantvort}
  \eqref{part:ann}), the convergence $Ku_\nu(re^{\img\theta})\to
  Gu_\nu(\infty)$ as $r \to \infty$ is uniform for all $\nu$. This
  proves the result.
\end{proof}

\begin{proof}[Proof of Proposition \ref{prop:evcont}
  \eqref{part:homology}] The conservation of equivariant homology
  class can be proved in a way similar to chapter 5 in Ziltener's
  thesis \cite{Zilt:thesis}. But, the proof is quite transparent if
  the sequence of vortices $(A_\nu,u_\nu)$ is part of a continuous
  family as in the quasimap case. We use that to give an indirect
  proof in the cases when $X$ is affine. In the proof of Theorem \ref{thm:main} in section
  \ref{sec:homeo}, we show that
$$\Psi:\bigsqcup_{\beta \in H_2^G(X)}Qmap(X\qu G, \beta) \to \bigsqcup_{\beta \in H_2^K(X)}MV^K(X, \beta)$$
is continuous, bijective and preserves $\beta$, without using the compactness of  $MV^K(X, \beta)$. The compactness of $MV^K(X, \beta)$ is now implied by the compactness of $Qmap(X\qu G, \beta)$, which is proved in \cite{CKM:quasimap} when $X$ is affine.
\end{proof}

We finally show that the Gromov limit is unique.
\begin{lemma}[Uniqueness of Gromov limit] \label{lem:cylunique} Assume
  the setting in Theorem \ref{thm:gromovcpt}. Suppose the sequence of
  vortices $(A_\nu,u_\nu)$ converges to a stable vortex $(A,u)$. Then
  this limit is unique up to 1) gauge transformations, 2) translation
  on cylindrical components i.e. $\alpha \in \Gamma^{cyl}$, 3)
  action of the finite automorphism group on $C$.
\end{lemma}
\begin{proof} In Proposition \ref{prop:convmodbreak}, suppose $(A_\nu,
  u_\nu)$ is a converging sub-sequence. The limit $(A,u)$ is unique up
  to gauge transformation because: $A$ is unique up to gauge since the
  space $\A/\K$ is Hausdorff (see Lemma 4.2.4 in \cite{DoKr}). Once we
  know gauge transformations $k_\nu$ such that $k_\nu A_\nu$ converges
  to $A$, then the limit $k_\nu u_\nu$, if it exists, is obviously
  unique. Suppose $k_\nu'$ is another sequence of gauge
  transformations for which $k_\nu'A_\nu$ converges to $A_\infty$. By
  arguments in the proof of Lemma 4.2.4 in \cite{DoKr}, after passing to
  a subsequence, $k_\nu'k_\nu^{-1}$ converge to a limit
  $k_\infty$. Then, the sequence $k_\nu'(A_\nu,u_\nu)$ converges to
  $k_\infty(A,u)$, i.e. the limit is still in the same gauge equivalence class. We remark that in such a situation $k_\infty$ is a
  stabilizer of the connection $A_\infty$.

  Next, consider the formation of cylindrical bubbles. By Lemma
  \ref{lem:break}, the choice of the sequence $\{s_\nu\}_\nu$ is
  unique in the following sense. If there is another sequence $s_\nu'$
  for which part \eqref{part:conv}-\eqref{part:connect} of Lemma
  \ref{lem:break} are satisfied, then after passing to a subsequence
  $\lim_\nu s_\nu' - s_\nu=L<\infty$. In that case, the limit $v_1'$
  will just be a re-reparametrization of $v$ :
  $v_1(\cdot)=v(L+\cdot)$. The rest of cylinder-breaking proof is an
  application of this Lemma, and hence the limit stable vortex is
  unique up to translation on cylindrical components.

  Finally, we consider the action of the finite group
  $\Aut(\st(C))$. Let $C_S \to (S,0)$ be a deformation of the curve
  $\st(C)$, with an isomorphism $\st(C) \simeq C_0$. Recall that the
  action of $\Aut(\st(C))$ extends to an equivariant action on $C_S
  \to S$ (after possibly shrinking $S$). The sequence $s(\nu) \in S$
  can be replaced by $\gamma\cdot s(\nu)$, where $\gamma \in
  \Aut(C)$. Then the Gromov limit is pulled back by the action of
  $\gamma$ on the stable components.
\end{proof}

\section{$\ol {MV}^K_{g,n}(X)$ is Hausdorff}\label{sec:hausdorff}
On the space ${MV}^K_{g,n}(X)$, {\em Gromov topology} is defined as: a
set $S \subset \ol {MV}^K_{g,n}(X)$ is closed if for any Gromov
convergent sequence in $S$, the limit also lies in $S$. In this
section we show
\begin{proposition}\label{prop:Haus}
  On $\ol {MV}^K_{g,n}(X)$, convergence in the Gromov topology
  coincides with Gromov convergence. $\ol {MV}^K_{g,n}(X)$ is
  Hausdorff under this topology.
\end{proposition}
At a first glance, the first statement may appear obvious. But it is
true only if we know beforehand that $\ol {MV}_{g,n}^K(X)$ is
Hausdorff and first countable -- see discussion in Section 5.6 of
\cite{MS}. We follow the approach in \cite{MS}, borrowing some ideas
from \cite{IP:VFC}. For any stable vortex $v:=(A,u)$ on a Riemann
surface with cylindrical ends $\Sig$ (corresponding to pre-stable
nodal curve $C$) modeled on a graph $\Gamma$, we define a function
$$d_v: \ol {MV}^K_{g,n}(X) \to [0,\infty]$$
that approximates `distance from $v$'. The function $d_v$ satisfies
the following:
\begin{description}
\item [D1] $d_v(v)=0$.
\item [D2] A sequence $v_\nu:=(A_\nu,u_\nu)$ in $\ol {MV}^K_{g,n}(X)$
  Gromov converges to $v$ if and only if $\lim_{\nu \to
    \infty}d_v(v_\nu)=0.$
\item [D3] If a sequence $v_\nu$ in $\ol {MV}^K_{g,n}(X)$ Gromov
  converges to $v':=(A',u')$, then \\$\limsup_{\nu \to \infty}
  d_v(v_\nu) \leq d_v(v')$.
\end{description}
In addition, if we know that Gromov limits are unique, then
Proposition \ref{prop:Haus} follows using Proposition 5.6.5 in
\cite{MS}.

Suppose $v_0:=(A_0,u_0)$ is a stable vortex defined on $\Sig_0$, a
Riemann surface with cylindrical ends corresponding to a pre-stable
curve $C_0$ and modular graph $\Gamma_0$. The forgetful map $\ft:\ol
{MV}^K_{g,n}(X) \to \ol M_{g,n}$ is continuous. Since the properties
D1-D3 are regarding a small neighborhood of $v$, $d_v$ can be defined
so that $d_v(v_0)$ is finite only if $\ft(v_0)$ is in an open neighborhood
$S_{\st(C)} \subset \ol M_{g,n}$ of $[\st(C)]$. We assume $S_{\st(C)}$ is
small enough that there is a morphism of
modular graphs $\st(\Gamma) \overset{f}{\to} \st(\Gamma_0)$ and 
Lemma \ref{lem:convprep} applies on $S_{\st(C)}$.
We can write
$\st(\Sig_0)=\bigsqcup_{\alpha \in
  \vertex(\st(\Gamma))}\tSig_{0,\alpha}/\sim$,
so that $\tSig_{0,\alpha} \subset (\st(\Sig),\ulg_{\Sig_0})$. Notice that $\tSig_0:=\cup_{\alpha \in \vertex(\st(\Gamma))}\tSig_{0,\alpha} \cup \cup_{\alpha \in \Gamma_0^{cyl}}(\R
\times S^1)_\alpha$ is a cover of $\Sig_0$.
% The metrics $g_{\Sig_0}$ and $g_{\Sig}$ agree on cylindrical ends.
% While defining $d_v(v_0)$, we do not keep track of metrics and just
% use $g_\Sig$ everywhere.  The cylinders $\Sig_\alpha=(S^1 \times
% \R)_\alpha$, $\alpha \in \Gamma^{cyl}$ are equipped with the
% standard metric.

To compute $d_v(v_0)$, we stabilize the domains of $v$ and $v_0$ by
adding some marked points.  On $C$, the domain curve of $v$, we add
$k$ marked points $\ol x=(x_\alpha)_{\alpha \in \Gamma^{cyl}}$ -- one
on each unstable cylindrical component $\Sig_\alpha$.
% Without loss of generality, we can assume $x_\alpha$ is the origin
% in the coordinate system $\R \times \R/\Z$.
Now, $[(C,\ol x)] \in \ol M_{g,n+k}$. Next, we choose $k$ additional marked
points $\ol x_0$ on $\Sig_0$, calculate $d_v(v_0;\ol x, \ol x_0)$ and
then take infimum over all choices. The marked points $\ol x_0 \in
C_0$ are chosen in a way that
\begin{itemize}
\item the $n+k$-marked nodal curve $(C_0,\ol x_0)$ is stable.
\item The morphism $f$ can be extended to $\Gamma' \overset{f}{\to}
  \Gamma'_0$, where $\Gamma'$ (resp. $\Gamma_0'$) is the graph $\Gamma$
  (resp. $\Gamma_0$) with the additional marked points.
\item All the marked points $\ol x_0$ lie in $\Cyl(\Sig_0)$, the
  cylindrical part of $\Sig_0$.
\end{itemize}
Given a choice of marked points $\ol x_0$ on the domain of
$v_0$, we construct an exhausting sequence of increasing pre-compact subsets of $\tSig_0$, denoted by $\Sig_{L,\ol x, \ol x_0}$, where $L\in \Z_{\geq 0}$. The
Riemann surface $\Sig_{L,\ol x, \ol x_0}$ is the disjoint union of the
following:
\begin{equation*}
  \Sig_{L,\ol x, \ol x_0,\alpha}=\begin{cases} \tSig_{0,\alpha} \bs (\cup_{z
      \in \Edge_\infty(\st(\Gamma))} \rho_z^{-1}\{r > L\} \cup\\
    \quad \cup_{w \in \Edge(\st(\Gamma))}\rho_{w^\pm}^{-1}\{\pm r > L\}) &\alpha \in
    \on{vert}(\st(\Gamma)) \\ 
    \{x_{0,\alpha}+z:|\on{Re}(z)| \leq L\} \cap \on{Cyl}(\Sig_0) &\alpha \in \Gamma^{cyl}.
  \end{cases}
\end{equation*}
The Riemann surface $\Sig_{L,\ol x, \ol x_0}$ can be embedded in $\Sig$.
The embedding has been discussed above for the components $\alpha\in \vertex(\st(\Gamma))$. For $\alpha \in \Gamma^{cyl}$, the embedding is isometrically induced
by identifying the marked point $x_{0,\alpha}$ with $x_\alpha$. Any
principal $K$-bundle over $\Sig_{L,\ol x, \ol x_0}$ is trivial.  One
can now consider the restriction of the vortices $v$ and $v_0$ to the
bundle $\Sig_{L,\ol x, \ol x_0}\times K$, this involves making a
choice of gauge both for $v$ and $v_0$. Define
\begin{align*}
  d_v(v_0;\ol x, \ol x_0, L)&:= \Mod{A_0-A}_{C^1(\Sig_{L,\ol x, \ol
      x_0})} + \Mod{d_X(u_0,u)}_{C^0(\Sig_{L,\ol x, \ol x_0})},\\
  d_v(v_0;\ol x, \ol x_0)&:=\inf_{k \in \K(\Sig)}\sum_{L \in \Z_{\geq
      0}}2^{-L} \frac {d_{kv}(v_0;\ol x, \ol x_0, L)}
  {1+d_{kv}(v_0;\ol x, \ol x_0, L)}.
\end{align*}
Finally, we define $d_v:\ol {MV}^K_{g,n}(X) \to [0,\infty]$ as
\begin{equation}\label{eq:dw}
  d_v(v_0):=\inf_{\ol x, \ol x_0}(d_v(v_0;\ol x, \ol x_0) + d_{\ol
    M_{g,n+k}}([(C,\ol x)],[(C_0,\ol x_0)])).
\end{equation}

\begin{proof}[Proof of Proposition \ref{prop:Haus}]: By the definition
  of $d_v$, it is easy to see that D1-D3 are satisfied. By Lemma
  \ref{lem:cylunique}, Gromov limits are unique up to gauge
  transformation. The proof follows using Proposition 5.6.5 in
  \cite{MS}.
\end{proof}

\begin{proof}[Proof of Theorem \ref{thm:maincpt}] Theorem \marginpar{*****}
  \ref{thm:maincpt} follows from Theorem \ref{thm:gromovcpt},
  Proposition \ref{prop:evcont} and Proposition \ref{prop:Haus}.  We
  examine the hypothesis of the compactness result (Theorem
  \ref{thm:gromovcpt}). The energy bound follows from the fact that
  for vortices, energy is constant on an equivariant homology class
  (see \eqref{eq:energyhom}). The images of the sequence of vortices
  is contained in a compact subset of $X$, since $X$ is equivariantly
  convex and the moment map is proper (see Lemma 2.7 in \cite{CGMS}).
\end{proof}

\begin{remark}\label{rem:asphere} 
  In case $X$ has spheres, i.e. $\pi_2(X)$ is non-trivial, then the
  limit of a sequence of vortices may have sphere bubbling in the
  $X$-fibers of the bundle $P(X)$. In order to incorporate this into
  our set-up, we first have to expand the definition of a pre-stable
  modular graph. In addition to allowing unstable cylindrical
  vertices, we also require unstable spherical vertices, we call the
  subgraph of spherical vertices $\Gamma^{sphere}$. The modular graph
  obtained by deleting the spherical vertices $\Gamma \bs
  \Gamma^{sphere}$ is connected. Each connected component of
  $\Gamma^{sphere}_x$ is a tree that is attached to a unique point $x$
  on the curve corresponding to $\Gamma \bs \Gamma^{sphere}$.  The
  graph $\Gamma^{sphere}$ does not contain any marked points. A stable
  vortex now has a sphere bubble tree in the fiber $P(X)_x$ modeled on
  $\Gamma^{sphere}_x$. The details of the convergence are proved by
  Ott \cite{Ott} using the following technique: a vortex
  $(A_\nu,u_\nu)$ can be viewed as a $J_{A_\nu}$-holomorphic curve
  from $C$ to $P(X)$. Here $J_{A_\nu}$ is the complex structure on
  $P(X)$ corresponding to the Dolbeault operator $\delbar_{A_\nu}$ and
  the sequence $J_{A_\nu}$ converges to $J_A$. Now, one can apply the
  result of Gromov convergence of $J$-holomorphic curves. In this
  expanded setting with sphere bubbles, one can still prove that the
  Gromov limit is unique and the space of stable vortices is
  Hausdorff. The proof is similar to the corresponding proof for
  $J$-holomorphic curves carried out in \cite{MS}.
\end{remark}

\section{Quasimaps}\label{sec:quasi}
In this section, we define quasimaps and recall why the moduli space
of stable quasimaps is proper. For details, refer to the paper by
Ciocan-Fontanine, Kim and Maulik \cite{CKM:quasimap}. Suppose $X$ is
an affine variety and $G$ a connected reductive complex algebraic
group acting on $X$. Let $\theta:G \to \C^\times$ be a character. It gives
a one-dimensional representation $\C_\theta$ of $G$, and hence a
$G$-equivariant polarization line bundle $L_\theta=X \times
\C_\theta$. The GIT quotient is defined as in King \cite{King}. Let
$S(L_\theta):=\oplus_{n\geq 0}\Gamma(X,L_\theta^n)^G$, then
$$X\qu G:=X\qu_\theta G:= \on{Proj}(S(L_\theta))$$
is a quasi-projective variety. Suppose $A$ is a ring such that
$X=\Spec(A)$, and $A^G \subset A$ be the ring of $G$-invariant
elements. Then the affine quotient is
$X/_{\on{aff}}G:=\Spec(A^G)$. Since $A^G$ is the zeroth graded piece
of $S(L_\theta)$, we have a projective morphism $X \qu G \to
X/_{\on{aff}} G$. The quotient $X \qu G$ is projective exactly when
$X/_{\on{aff}}G$ is a point, i.e. $A^G$ only consists of constants.

The semistable locus $X^\ss(\theta)$ is defined as the set of points
$x \in X$ for which there exists $f \in \Gamma(X,L_\theta^n)^G$, $n
\geq 1$ satisfying $f(x) \neq 0$. The stable locus $X^s(\theta)$
consists of points $x \in X^\ss(\theta)$ for which the orbit $Gx
\subset X^\ss(\theta)$ is closed and the stabilizer group of $x$ is
finite.
  
\begin{assumption}\label{ass:67} In Sections \ref{sec:quasi} and \ref{sec:homeo},
  $X$ is an affine variety with an action of a reductive group $G$ and
  a character $\theta:G \to \C^\times$. Further,
  \begin{enumerate}
  \item\label{part:assfree} $X^\ss(\theta)=X^\s(\theta)$, $G$ acts
    freely on $X^\ss$.
  \item \label{part:assproj} $X \qu G$ is projective.
  \end{enumerate}
\end{assumption}
In Assumption \ref{ass:67}, part \eqref{part:assfree} is same as the
assumptions in \cite{CKM:quasimap}. Part \eqref{part:assproj} is put
in place to simplify presentation, as this case captures all the
technical aspects. Section \ref{sec:noncpt} discusses how this
assumption can be removed. Part \eqref{part:assfree} of the above
assumption implies that $X \qu G$ is the orbit space of the $G$ action
on $X^\ss$ and it is a smooth manifold.
\begin{definition} Given integers $g$, $n \geq 0$ and a class $\beta
  \in H_2^G(X)$, an $n$-pointed genus $g$ {\em quasimap} of class
  $\beta$ to $X \qu G$ consists of the data $(C,p_1,\dots,p_n, P, u)$,
  where
  \begin{itemize}
  \item $(C,p_1,\dots,p_n)$ is a nodal curve of genus $g$ with $n$
    marked points,
  \item $P$ is a principal $G$-bundle on $C$,
  \item $u \in \Gamma(C,P \times_G X)$ such that $(P,u)$ is of class
    $\beta$,
  \end{itemize}
  satisfying: there is a finite set $B$ so that $u(C \bs B) \subset
  P \times_G X^\ss$. The points in $B$ are called the {\em base points} of the
  quasimap.  For a component $C' \subset C$, $u$ is {\em constant} if it does not have base points and $\pi_G \circ u:C' \to X\qu G$ is
  a constant. Here $\pi_G:X^\ss \to X \qu G$ is the projection
  map. The quasimap $(C,p_1,\dots,p_n, P, u)$ is {\em stable} if the
  base points are disjoint from the nodes and markings on $C$ and
  \begin{itemize}
  \item every genus $0$ component $C'$ of $C$ has at least 2 marked or
    nodal points. If it has exactly 2 special points, then $u$ is
    non-constant on $C'$,
  \item if for a genus $1$ component $C'$ of $C$, $u$ is constant on
    $C'$, then $C$ has at least one marked or nodal point.
  \end{itemize}
  An {\em isomorphism} between two quasimaps $(C,\ul p, P, u)$ and
  $(C',\ul p', P', u')$ consist of isomorphisms $f:(C,\ul p) \to
  (C',\ul p')$ and $\sig: P \to f^*P'$ such that under the induced
  isomorphism $\sig_X:P \times_G X \to f^*(P' \times_G X)$, $u$ maps
  to $u'$.
\end{definition}

\begin{definition} A {\em family of quasimaps} over a base scheme $S$
  consists of the data $(\pi:\mC \to S,\{p_j:S \to
  \mC\}_{j=1,\dots,k},\mP,u)$ where $\mC \to S$ is a proper flat
  morphism such that each geometric fiber $\mC_s$, $s \in S$ is a
  connected curve, $p_j$ are sections of $\pi$, $\mP \to \mC$ is a
  principal $G$-bundle and $u:\mC \to \mP \times_G X$ is a
  section. For each $s \in S$, $(\mC_s, \ul p(s), \mP_s, u|_{\mC_s})$
  is a quasimap.
\end{definition}

We describe the compactification of $Qmap_{g,n}(X \qu G, \beta)$ from
\cite{CKM:quasimap}. The proof is by the valuative criterion. Let
$(S,0)$ be a pointed curve and let $S^0=S \bs \{0\}$. Let
$((\mC,p_j),\mP,u)$ be a $S^0$-family of stable quasimaps. After possibly
shrinking $S$ and making an \'etale base change, the base points can be
regarded as additional sections $y_i:S^0\to \mC$. The family gives a
rational map $[u]:\mC \to X \qu G$ which is smooth on $\mC \bs B$. Since
$X \qu G$ is projective, after shrinking $S$ (removing closed sets
where the extension requires blowing up), $[u]$ extends to a map on
all of $\mC$, denoted by $[u_{reg}]$. Stability of the quasimaps implies
that the family $(\mC,(\ul p, \ul y),[u_{reg}])$ is a family of Kontsevich stable
maps to $X \qu G$. By compactness of the
moduli space of stable maps, after an \'etale base change possibly
ramified at $0$, there is a family of stable maps
$$(\hat \mC, (\ul p, \ul y), [\hat u]) \to S, \quad [\hat u]:\hat \mC \to X \qu G$$  
extending $(\mC,[u_{reg}])$. The surface $\hat \mC$ has at most nodal
singularities in the central fiber $\hat \mC_0$. Pull back the principal
$G$-bundle $X^s \to X\qu G$ via $[\hat u]$ to obtain a principal
$G$-bundle $\hat \mP$ and an induced section $\hat u: \hat \mC \to \hat \mP
\times_G X$.

Next, consider all maximal sub-trees of sphere bubbles (rational
tails) $\Gamma_1,\dots,\Gamma_N$ in $\hat \mC_0$ that have none of the
marked points $p_j$ and meets the rest of the curve $\ol{(\hat \mC_0
  \bs \Gamma_l)}$ in a single point $z_l$. Contract these subtrees in
$\hat \mC$ and call the resulting surface $\ol \mC$ and let $(\ol
\mP,\ol u):=(\hat \mP, \hat u)|_{\ol \mC}$. The gauged map $(\ol \mP,
\ol u)$ is well-defined on $\mC\bs \{z_1,\dots,z_N\}$. Base points
$y_i$ may come together at the points $z_1,\dots z_N$, but these
points are away from the markings $p_j$. The families of quasimaps
$(\ol \mP, \ol u)$ and $(\mP, u)$ are defined over $\mC \bs
\{y_j,z_l\}$ and $\mC \bs \mC_0$ respectively and they agree on the
intersection. So, they can be patched to yield a family $(\mP, u)$
that is defined on $\mC \bs \{y_j(0),z_l\}$.  By Lemma 4.3.2 in
\cite{CKM:quasimap}, $\mP$ extends to a principal $G$-bundle on
$\mC$. By Hartog's Theorem, since $X$ is affine, $u$ extends to
all of $\mC$.  The central fiber $u_0$ will have base points at which
rational tails were contracted. The following two steps affect the
equivariant homology of $u$: (1) replacing $u$ by $u_{reg}$, (2)
contracting rational tails in the central fiber and letting the
central fiber acquire base points. But it can be shown that the
equivariant homology class $\beta$ is preserved in the limit. This is
done in Section 7 in \cite{CKM:quasimap} by analyzing the contribution
of base points to equivariant homology. The properness of
$Qmap_{g,n}(X \qu G,\beta)$ now follows by the valuative criterion.
\section{A homeomorphism between quasimaps and
  vortices}\label{sec:homeo}
The proof of Theorem \ref{thm:main} is carried out in this section.

\subsection{From a $G$-variety to a Hamiltonian $K$-manifold}\label{sec:KempfNess}
Given an affine $G$-variety $X$ as in Section \ref{sec:quasi}, one can
give it the structure of a Hamiltonian $K$-manifold in a way that the
GIT and symplectic quotients are homeomorphic.  By Proposition 2.5.2
in \cite{CKM:quasimap}, there is a complex vector space $V$ with a linear
$G$-action and a $G$-equivariant closed embedding $X \hra V$. Let $(\cdot,\cdot)$ be a Hermitian product on $V$ and $K \subset G$ be a maximal compact subgroup whose action preserves the Hermitian product. The imaginary part of the Hermitian product gives a symplectic form on $V$:
$$\om_v(v_1,v_2):=\on{Im}(v_1,v_2), \quad v, v_1, v_2 \in V.$$
The $K$-action on $V$ has moment map
$$\Phi_0:V \to \k^\dual, \quad \lan\Phi_0(v),\xi \ran:=\hh\on{Im}(v,\xi v), \quad \xi \in \k.$$
The moment map on the polarized affine variety $X$ is 
$$\Phi:=\Phi_0 + i\d\theta,$$
where $\d\theta:\k \to i\R$ is the restriction of the derivative of
the character $\theta:G \to \C^\times$ at the identity:
$\d\theta:\g^\dual \to \C$.  King (\cite{King}, Theorem 6.1) proves a
generalization of the Kempf-Ness theorem which says that any
$\theta$-semistable $G$ orbit that is closed in $X^\ss(\theta)$ meets
$\Phi^{-1}(0)$ in a unique $K$-orbit. So, there is a homeomorphism
$\Phi^{-1}(0)/K \to X \qu G$. By part \eqref{part:assfree} of
Assumption \ref{ass:67}, we get $X^\ss =G\Phi^{-1}(0)$, and part
\eqref{part:assproj} ensures that $\Mod{\Phi}^2$ is proper on $X$. The Hamiltonian manifold $X$ is equivariantly proper. In particular, if $V=\C^r$ with the standard Hermitian product, then equivariant
convexity is satisfied by taking $f=\Sigma_{i=1}^r|z_i|^2$.

\subsection{Correspondence between vortices and stable quasimaps}
In this section, we produce the bijection $\Psi$ in Theorem
\ref{thm:main} -- see Corollary \ref{cor:biject}. This is done by
applying the Hitchin-Kobayashi correspondence (Theorem \ref{thm:HK})
on each component of $C$.

For a principal $G$-bundle $P_\C \to C$, a reduction of structure group can be
achieved by choosing a section $\sig:C \to P_\C/K$, this produces a
principal $K$-bundle $P \subset P_\C$ satisfying $P_\C=P \times_K
G$. By Remark \ref{rem:conntohol}, holomorphic structures on $P_\C$
correspond to connections on $P$. Therefore after a choice of
reduction, a $G$-gauged map $(P_\C,u)$ can be viewed as a $K$-gauged
map $(A,u)$ on $P$.

\begin{definition}[Complex gauge transformations] A {\em complex gauge
    transformation} $g$ is an automorphism of $P_\C$, it is a
  $G$-equivariant bundle map $P_\C \to P_\C$. Under a local
  trivialization $g$ is given by a map from the Riemann surface to the
  group $G$. Recall that the Cartan decomposition (see Helgason
  \cite{Helg2} VI.1.1)
  \begin{equation}\label{eq:complexgroupiso}
    K \times \k \to G, \quad (k,s) \mapsto ke^{is}
  \end{equation}
  is a diffeomorphism. So, a complex gauge transformation $g$ can be
  written as $g=ke^{i\xi}$, where $k \in \K(P)$ and $\xi \in
  \on{Lie}(\K(P)) = \Gamma(\Sig, P(\k))$.
\end{definition}
We denote by $\G(P)$ the group of complex gauge transformations associated to the $G$-bundle $P_\C:=P \times_K G$. This
group acts on the space of holomorphic structures on $P_\C$ via
pull-back. This corresponds to an action of $\G(P)$ on the space of
connections that extends the action of $\K(P)$. Let $\Sig$ be a
Riemann surface with cylindrical ends and $P=\Sig \times K$ be a
trivial $K$-bundle. For $k \in \Z_{\geq 0}$, $p>1$ such that
$(k+1)p>2$, we define the space of $W^{k+1,p}$ complex gauge
transformations as $\G^{k+1,p}(P):=\{ke^{i\xi}: k \in \K^{k+1,p}, \xi
\in W^{k+1,p}(\Sig,\k)\}$. By Lemma \ref{lem:gcactona},
$\G^{k+1,p}(P)$ acts smoothly on $\A^{k,p}$.

\begin{definition} [$p$-bounded gauge transformations and gauged
    holomorphic maps] Let $p>2$. Let $C$ be a pre-stable curve with
  $n$-marked points, let $P$ be a principal $K$-bundle on $C$. We call
  a gauged holomorphic map $(A,u)$ on $P \to C$ {\em $p$-bounded} if it
  is smooth on $C \bs \{z_1,\dots,z_n\}$ and for $1\leq j \leq n$, on
  a neighborhood of $z_j$, $N(z_j) \subset C$, $(A,u)|_{N(z_j)} \in
  L^{p} \times W^{1,p}$. A {\em complex gauge transformation} $g$ on
  $P$ is {\em $p$-bounded} if it is smooth on $C \bs \{z_1,\dots,z_n\}$
  and $g|_{N(z_j)} \in W^{1,p}(N(z_j),G)$ for $j=1,\dots,n$.  Denote
  by $\G(P)^p_{\on{bd}}$ the group of $p$-bounded gauge transformations.
\end{definition}

Let $C$ be a pre-stable smooth genus $g$ curve with $n$ marked points
($n \geq 1$), we denote by $Qmap(C,X \qu G, \beta)$ the subset of
stable quasimaps $Qmap_{g,n}(X \qu G,\beta)$ whose domain curve is
$C$. Analogously, if $\Sig$ be the Riemann surface with cylindrical
ends corresponding to $C$, then $MV^K(\Sig,X,\beta) \subset
\ol{MV}_{g,n}^K(X,\beta)$ be the subset of stable vortices with domain
$\Sig$.

\begin{theorem}[Variant of Theorem 3.1 in \cite{VW:affine}]\label{thm:HK} 
  Suppose $X$, $G$, $K$ be as above and $0<\gamma<1$ be as in
  Proposition \ref{prop:decaycyl}. Let $C$ be a pre-stable smooth
  genus $g$ curve with $n$ marked points $z_1,\dots,z_n$ ($n \geq 1$)
  and $\Sig$ be the corresponding Riemann surface with cylindrical
  ends.

  Let $(A,u)$ be a gauged holomorphic map from a principal $K$-bundle
  $P \to C$ to $X$. Suppose $u(z_i) \in X^{\ss}$ for all the marked
  points $z_1,\dots,z_n$. There exists a complex gauge transformation $g$
  on $P$ such that $g(A,u)|_\Sig$ is a smooth finite energy symplectic
  vortex with bounded image.  For any $2<p<\frac 2 {1-\gamma}$, the complex gauge transformation $g$ is $p$-bounded and it is unique up to left multiplication by $p$-bounded unitary gauge transformations $k \in \K(P)^p_{\on{bd}}$. 
  The resulting map
  \begin{equation}\label{eq:smoothbiject}
    Qmap(C,X \qu G, \beta) \to MV^K(\Sig,X,\beta), \quad (A,u) \mapsto g(A,u)|_\Sig
  \end{equation}
  is a bijection.
\end{theorem}
Theorem 3.1 in \cite{VW:affine} states the above result with $\Sig$
replaced by $\C$ with the Euclidean metric $\d x \wedge \d y$ and $C$ is
$\P^1$ with a marked point at infinity. The proof of that result
entirely carries over for the above variant. In fact some parts of the
proof simplify, the construction of the inverse of
\eqref{eq:smoothbiject} uses an asymptotic decay result for
vortices. The corresponding result -- Proposition \ref{prop:decaycyl} above -- is much
simpler for a base space with cylindrical ends than the affine
line. The condition $p<\frac 2 {1-\gamma}$ is required to extend a
vortex over $\Sig$ to a $L^p_{loc} \times W^{1,p}_{loc}$ gauged map
over $\ol \Sig$ (see Corollary \ref{cor:extendoverinfty}). This
condition is absent in the affine case in \cite{VW:affine}, as there
$\gamma$ can be chosen arbitrarily close to $1$. Also note that for
$p_1<p$, $\G(P)^p_{\on{bd}} \subset \G(P)^{p_1}_{\on{bd}}$. We define
a set $\G(P)_{\on{bd}}:=\cap_{2<p<(2/(1-\gamma))} \G(P)_{\on{bd}}^p$.
\begin{corollary}[Bijection between $Qmap_{g,n}(X \qu G,\beta)$ and
  $\ol {MV}^K_{g,n}(X,\beta)$]\label{cor:biject}
  For any equivariant homology class $\beta \in H_2^K(X)$, Theorem
  \ref{thm:HK} produces a bijection
$$\Psi:Qmap_{g,n}(X \qu G,\beta) \to \ol {MV}^K_{g,n}(X,\beta).$$
\end{corollary}
\begin{proof} Consider a stable quasimap $(C,\ol z, P_\C, u)$. Recall
  that $C$ is pre-stable, let $\Sig$ be a Riemann surface with
  cylindrical ends corresponding to $C$. By assumption, we have $n
  \geq 1$, so each component of $C$ has at least one special point. We
  can then apply Theorem \ref{thm:HK} to the gauged map
  $(P_\C,u)|_{C_i}$ defined on each component $C_i$ of $C$. The result
  is a finite energy vortex $(\tilde A_i, \tilde u_i)$ on each
  component $\Sig_i\simeq C_i \bs \{\text{special points}\}$. The
  connectedness condition \eqref{eq:connect} holds because : for the
  quasimap $(P_\C,u)$, we have for any $w \in \Edge(C)$,
  $Gu_{\inc(w^+)}(w^+)=Gu_{\inc(w^-)}(w^-)$. This $G$-orbit is a
  semi-stable orbit, so its intersection with $\Phi^{-1}(0)$ is a
  unique $K$-orbit, lets call it $\chi$. So, on the vortex side, we
  have $K\tilde u_{\inc(w^\pm)}(w^\pm)=\chi$ (see Proposition
  \ref{prop:infvortsing}). Stability of the quasimap, together with $n
  \geq 1$, implies that an unstable domain component $C'$ is a sphere
  with 2 special points and $u|_{C'}$ is non-constant. This implies
  that the vortex $(\tilde A,\tilde u)|_{C'}$ is also
  non-constant. So, $(\tilde A, \tilde u)$ is a stable vortex that is
  uniquely determined up to unitary gauge equivalence. Further $p$-bounded
  complex gauge transformations preserve the equivariant homology
  class. So we define $\Psi:(C,\ul x, P_\C, u) \mapsto [(\tilde A,
  \tilde u)] \in \ol {MV}^K_{g,n}(X,\beta)$. This is a bijection since
  \eqref{eq:smoothbiject} is a bijection.
\end{proof}
 
\begin{remark}[Base points on the vortex side] Base points do not play
  a significant role in convergence of vortices. This is because in
  general the theory of vortices does not require the target manifold
  $X$ to be K\"ahler, so there is no action of the complexified group
  $G$, and so `base points' cannot be defined. For the homeomorphism
  result we rely on two features of quasimaps: first the combinatorial
  nature of domains of quasimaps happens to be same as that of
  vortices with cylindrical ends. Secondly, a continuous family of
  quasimaps is mapped by $\Psi$ to a family of vortices continuous in
  the sense of Gromov topology.
\end{remark}

\subsection{Continuity for smooth curves}
In this and the following section, we prove the continuity of the map
$\Psi$ in Theorem \ref{thm:main}. Let $Q_S$ be a family of quasimaps
parametrized by $S$, $Q_S=(\mC \to S,\{z_j:S \to
\mC\}_{j=1,\dots,k},\mP_\C,u)$, We assume $S$ is a neighborhood of the
origin in $\C^N$. We prove for any sequence $s_\nu \to 0$ in $S$, the
sequence of vortices $\Psi(Q_{s_\nu})$ Gromov converges to
$\Psi(Q_0)$. In this Section we restrict our attention to families of
smooth curves, i.e. the curves $\mC_s$ do not have nodal
singularities, and prove the following Proposition.
\begin{proposition}[Continuity of $\Psi$ for smooth
  curves] \label{prop:contsmooth} Given a family of quasimaps
  $Q_S=(\mC \to S,\{z_j:S \to \mC\}_{j=1,\dots,k}, \mP_\C,u)$
  parametrized by $S$, where $S \subset \C$ is a neighborhood of
  $0$. Suppose $\mC_s$, $s \in S$ are smooth curves, then for any
  sequence $s_\nu \to 0$ as $\nu \to \infty$, $\Psi(Q_{s_\nu})$ Gromov
  converges to $\Psi(Q_0)$.
\end{proposition}
\begin{proof}

  We first outline the proof. The first step is to make a preliminary
  choice of reduction $\sig:\mC \to P_\C/K$ (Lemma
  \ref{lem:stdform}). This produces a principal $K$-bundle $\mP \to
  \mC$ and a family of gauged holomorphic maps $(A_s,u_s)$ defined on
  $\mP_s$. Suppose the complex gauge transformation $g_0 \in
  \G(\mP_0)$ transforms the central element $(A_0,u_0)$ to a
  vortex. Then the gauged holomorphic maps $g_0(A_{s_\nu},u_{s_\nu})$
  converges to $g_0(A_0,u_0)$. In this situation, we can find a
  sequence of complex gauge transformations $e^{i\xi_\nu}$ converging
  to the identity that transform the elements in the sequence
  $g_0(A_{s_\nu},u_{s_\nu})$ to vortices. Then, we will conclude that the vortices
  $e^{i\xi_\nu}g_0v_{s_\nu}$ Gromov converge to $g_0(A_0,u_0)$. The
  details are as follows.

  By Lemma \ref{lem:convprep} and by the fact that the curves $\mC_s$
  do not contain nodes, there is a family of diffeomorphisms
  $h_s:\mC_s \to \mC_0$, that are isometries on the cylindrical ends
  $\mC_s \bs \{z_s\}$ with respect to the neck-stretching metric. Via
  these diffeomorphisms, we identify the Riemann surfaces with
  cylindrical ends $\Sig_s$ to $\Sig :=\Sig_0$.  Using Lemma
  \ref{lem:stdform}, choose a reduction $\sig:\mC \to \mP_\C/K$ so
  that the gauged holomorphic maps $(A_s,u_s)$ are in {\em standard
    form close to marked points} -- i.e. they satisfy conditions (a)
  and (b) in Lemma \ref{lem:stdform}. We fix a trivialization
  $\mP_s|_{\Sig_s}$ so that $A_s=\d+\lambda_j\d\theta$ on $N(z_j)_s\bs
  z_j(s)$ for all $s$. We observe that under this trivialization,
  \begin{equation}\label{eq:quasiconv}
    \cup_{s \in S}\on{Im}(u_s) \subset X \text{ is pre-compact}, \; A_s \xrightarrow{H^1(\Sig)} A_0, \; u_s \xrightarrow{H^2(\Sig)}u_0 \quad \text{as $s \to 0$.}
  \end{equation}
  By Theorem \ref{thm:HK}, there is a $p$-bounded complex gauge
  transformation $e^{i\xi_0}$ on $\mP_0$ that transforms $(A_0,u_0)$
  to a finite energy vortex on $\Sig_0$. Under the above
  trivialization of $\mP_0|_{\Sig_0}$, by Lemma \ref{lem:xiexpbd}, for
  any $0<\gamma_1<\gamma$,
$$\Mod{\xi_0}_{H^2(\rho_{z_j}\{n \leq r \leq n+1\})} \leq ce^{-\gamma_1 n}.$$
Here $\gamma$ is as in Proposition \ref{prop:decaycyl}. Then, the
family of gauged holomorphic maps $e^{i\xi_0}(A_s,u_s)$ continues to
satisfy \eqref{eq:quasiconv}. Pick a sequence of points $s_\nu$ in $S$
converging to $0$. To the converging sequence
$e^{i\xi_0}(A_{s_\nu},u_{s_\nu}) \to e^{i\xi_0}(A_0,u_0)$, apply Lemma
\ref{lem:vortcont}. This will give a sequence $\xi_{s_\nu}:\Sig \to
\k$ converging to $0$ in $H^2(\Sig)$ such that
$e^{i\xi_{s_\nu}}e^{i\xi_0}(A_{s_\nu},u_{s_\nu})$ is a vortex on
$\Sig$. So, the vortices
$e^{i\xi_{s_\nu}}e^{i\xi_0}(A_{s_\nu},u_{s_\nu})$ converge to
$e^{i\xi_0}(A_0,u_0)$ in $H^1(\Sig) \times H^2(\Sig)$. Finally, by
elliptic regularity for vortices, modulo gauge transformations, we can
say the sequence $e^{i\xi_{s_\nu}}e^{i\xi_0}(A_{s_\nu},u_{s_\nu})$
converges to $e^{i\xi_0}(A_0,u_0)$ smoothly on compact subsets of
$\Sig$ (see Lemma 3.8, 3.9 in \cite{VW:affine} and Theorem 3.2 in
\cite{CGMS}). Convergence of energy values $E(A_{s_\nu},u_{s_\nu})$ to
$E(A_0,u_0)$ follows from convergence of $(A_{s_\nu},u_{s_\nu})$ in
$H^1(\Sig) \times H^2(\Sig)$ and from the fact that the images of
$u_s$ are contained in a compact subset of $X$. Finally it remains to
show that the vortex $e^{i\xi_{s_\nu}}e^{i\xi_0}(A_{s_\nu},u_{s_\nu})$ is in the
gauge orbit $\Psi(Q_{s_\nu})$. The complex gauge transformation
$e^{i\xi_{s_\nu}}e^{i\xi_0}$ extends continuously over the bundle
$\mP_{s_\nu} \to \mC_{s_\nu}$.  By arguments in the proof of Theorem
3.1 (a) in \cite{VW:affine}, $e^{i\xi_{s_\nu}}e^{i\xi_0}$ is in the
class $\G(\mP_{s_\nu})_{\on{bd}}$. Finally, Proposition 4.5 in
\cite{VW:affine} says that if two vortices are related by a complex
gauge transformation that is $p$-bounded, then the vortices are
related by a $p$-bounded unitary gauge transformation.

All the lemmas used in this proof assume there is only one marked
point, i.e. $n=1$. This case captures all the technical details and
lets us drop the subscript $j$ for marked points. We also ignore the
variation in metric on curves in the family $\mC \to S$. Since the
metrics agree on cylindrical ends, and they vary smoothly with $s$,
they do not affect the proofs much.
\end{proof}

\begin{lemma}[Standard form around marked points] \label{lem:stdform}
  Let $p>1$. Suppose $Q_S$ is a family of quasimaps such that all the
  fibers in $\mC \to S$ are diffeomorphic via $h_s:\mC_s \to \mC_0$
  (as in Lemma \ref{lem:convprep}). Assume there is a single marked
  point $z:S \to \mC$. There is a reduction $\sig:\mC \to \mP_\C/K$
  and $\lambda \in \frac 1 {2\pi}\exp^{-1}(\Id) \subset \k$ such that
  the following are satisfied. Let $\mP \to \mC$ be the principal
  $K$-bundle, and $\{(A_s,u_s)\}_{s \in S}$ be the family of
  $K$-gauged holomorphic maps induced by the reduction $\sig$.
  \begin{enumerate}
  \item \label{part:conn} There is a trivialization $\tau$ of the
    restriction $\mP|_{\mC \bs z(S)}$ so that on the cylindrical end
    $N(z)\bs \{z\}$, for all $s\in S$, $A_s=\d+\lambda \d\theta$ and
  \item \label{part:map} $u(z(s)) \in \Phi^{-1}(0)$ for all $s \in S$.
  \end{enumerate}
  $\sig$, $\tau$ are smooth in every fiber and vary smoothly with $s$.
\end{lemma}
\begin{proof}[Proof of Lemma \ref{lem:stdform}]
  We make a preliminary choice of reduction $\sig_0:\mC \to \mP_\C/K$
  and obtain a principal $K$-bundle $\mP \to \mC$ and a smooth family
  of $K$-gauged holomorphic maps $(A_s,u_s)$. To prove the lemma, we
  need to find a family of complex gauge transformations $g_s$ on
  $\mP_s \to \mC_s$ so that $g_s(A_s,u_s)$ are in standard form near
  the marked point $z$.

  First, we find a family of complex gauge transformations that make
  the connections flat on $N(z)$. Note that the sets $N(z)_s \subset
  \mC_s$ are compact with boundary and are isometric for all $s$. We
  work on a trivialization of the bundle $\mP|_{\cup_s
    N(z)_s}$. Recall the fact that on a compact surface $\Om$ with
  boundary, a connection can be complex gauge transformed to a flat
  connection by $e^{i\xi}$, where $\xi \in \Gamma(\Om,P(\k))$ and
  $\xi|_{\partial \Om}=0$, and the choice of $\xi$ is unique (see for
  example Theorem 1 in \cite{Do:bdry}). Thus there are complex gauge
  transformations $e^{i\xi_s}:N(z)_s \to G$ such that $e^{i\xi_s}A_s$
  is flat. We now prove that $\xi_s$ varies continuously with $s$ in the
  $W^{2,p}(N(z))$-topology. The argument used for this is the
  prototype of the argument for continuity used in this
  section. Denoting $A_s':=e^{i\xi_0}A_s$, we have
  $\Mod{F_{A'_s}}_{L^p(N(z(s)))} \to 0$ as $s \to 0$.  By Lemma
  \ref{lem:toflat} for small enough $s$, we can find $\xi'_s \in
  W^{2,p}$ that satisfies $\Mod{\xi'_s}_{W^{2,p}} \leq
  c\Mod{F_{A_s'}}_{L^p}$, $\xi_s|_{\partial N(z)}=0$ and
  $e^{i\xi'_s}A_s$ is flat on $N(z)_s$. By the uniqueness of $\xi_s$,
  the complex gauge transformations $e^{i\xi_s}$ and
  $e^{i\xi_s'}e^{i\xi_0}$ differ by a unitary gauge
  transformation. That is, there exist $h_s:N(z) \to K$ such that
  $e^{i\xi_s'}e^{i\xi_0}=h_se^{i\xi_s}$. Since
  \eqref{eq:complexgroupiso} is a diffeomorphism, $\xi_s$ varies
  continuously with $s$ in $W^{2,p}(\mC_0)$.

  Next, we show that the complex gauge transformations vary smoothly
  with $s$.  Since the connections $e^{i\xi_s}A_s$ are flat on
  $N(z)_s$, there exists a unique family of unitary gauge
  transformations $k_s \in W^{2,p}(N(z),K)$ such that $k_s(z)=\Id$ and
  the connection $k_se^{i\xi_s}A_s$ is the trivial connection. The
  continuity of $s \mapsto e^{i\xi_s}A_s$ in $W^{1,p}(N(z))$ implies
  the continuity of $k_s$ in $W^{2,p}(N(z))$. Let
  $g_s':=k_se^{i\xi_s}$. The complex gauge transformations $g_s':N(z)
  \to G$ are smooth and vary smoothly with $s$. This can be seen as
  follows: Let $A_s=\d+a_s$. We know $g_s'A_s=\d$. So,
  $$\delbar_{A_s}=(g_s)^{-1} \circ \delbar \circ g_s'=\delbar + (g_s)^{-1} \delbar g_s' \implies a_s^{0,1}=(g_s')^{-1}(\delbar g_s').$$
  The quantity $a_s^{0,1}$ varies smoothly with $s$, therefore by an
  elliptic regularity argument, $g_s'$ is smooth and varies smoothly
  with $s$.

  Next, we consider the second condition \eqref{part:map}. Since
  $g_s'u_s(z) \in X^\ss$, there is a unique $\psi_s \in \k$ such that
  $e^{i\psi_s}(g_s'u_s(z)) \in \Phi^{-1}(0)$. Further, the map $s
  \mapsto \psi_s$ is smooth. Since trivial connections are preserved
  by constant complex gauge transformations, so $e^{i\psi_s}g_s'A_s$
  is the trivial connection. For all $s$, the map
  $e^{i\psi_s}g_s':N(z) \to G$ is homotopic to identity. So, there is
  a smooth family of complex gauge transformations $g_s \in \G(\mP_s)$
  that agrees with $e^{i\psi_s}g_s'$ in a smaller neighborhood of $z$
  and is equal to $\Id$ on $\mC_s \bs N(z)_s$. The neighborhoods
  $N(z)_s$ can be replaced by smaller neighborhoods, and this proves
  the Lemma, so we have the connections $g_sA_s$ are trivial on
  $N(z)$.

  It remains to construct the trivializations $\tau_s$. The bundles
  $\mP_s \to \mC_s\bs \{z(s)\}$ can be trivialized. The connections
  $g_sA_s$ are flat over $N(z)_s\bs \{z(s)\}$, so the holonomy is
  constant on the cylinder. Hence the trivialization of $\mP_s$ can be
  chosen such that $\hat g_s A_s=\d+\lambda(s) \d\theta$ for some
  $\lambda(s) \in \frac 1 {2\pi}\exp^{-1}(\Id) \subset \k$. The
  holonomies $\lambda(s)$ can be chosen to be independent of $s$ since
  it is a homotopy invariant.
\end{proof}
If $(A,u)$ is a gauged holomorphic map on $\Sig$ that is in standard
form, we have an asymptotic bound on the complex gauge transformation
that makes $(A,u)$ a vortex on $\Sig$.
\begin{lemma}[Asymptotic bound on a complex gauge
  transformation] \label{lem:xiexpbd} Let $0<\gamma_1<\gamma$, where
  $\gamma$ is the constant from Proposition
  \ref{prop:decaycyl}. Suppose $C$ is a smooth curve with a single
  marked point $z$ and $\Sig$ is the corresponding Riemann surface
  with cylindrical ends. Suppose $(A_0,u_0)$ is a smooth gauged
  holomorphic map on a $K$-bundle $P \to C$ that is in standard form
  close to marked points, and the $p$-bounded complex gauge transformation $
  e^{i\xi}\in \G(P)_{\on{bd}}$ transforms $(A_0,u_0)$ to a finite
  energy vortex on $\Sig$. For any $q>1$, there is a constant $c$ such
  that in the trivialization of $P|_\Sig$ found in Lemma
  \ref{lem:stdform},
  \begin{equation}\label{eq:xiexpbd}
    \Mod{\xi}_{W^{2,q}(\{n \leq r \leq n+1\})} \leq ce^{-\gamma_1 n}.
  \end{equation}
  Hence, $\xi \in W^{2,q}(\Sig)$.
\end{lemma}
\begin{proof}[Proof of Lemma \ref{lem:xiexpbd}] On the cylindrical end
  $N(z) \bs \{z\}$, we use the co-ordinates $r$, $\theta$ given by
  $\rho_z$. The gauged holomorphic map $(A_0,u_0)$ being in standard
  form means $\Phi(u_0(z))=0$ and there is a trivialization of the
  bundle $P|_\Sig$ so that $A_0|_{N(z)\bs \{z\}}=\d+\lambda
  \d\theta$. The transition function between $P|_\Sig$ and $P|_{N(z)}
  \simeq N(z) \times K$ is $\theta \mapsto e^{-\lambda \theta}$. We
  fix these trivializations for the rest of the proof. Under these
  trivializations, a gauge transformations $k \in \K(P)$ can be
  expressed by functions $k:\Sig \to K$ and $\hat k:N(z) \to K$
  satisfying $\hat k=e^{-\lambda \theta}ke^{\lambda \theta}$ on $N(z)
  \bs \{z\}$.

  {\sc Step 1}: {\em For any $2<p<\frac 2 {1-\gamma}$, there is a $p$-bounded gauge transformation $k \in \K(P)_{\on{bd}}^p$ such that if $g:=ke^{i\xi}$, then $(gA_0)|_{\Sig}=\d+\lambda \d\theta+a$ where $a$ satisfies the exponential bound \eqref{eq:twistconnbd}.}\\
  We first define a gauge transformation on $\Sig \simeq C \bs \{z\}$
  that transforms the connection to the above form.  Denote
  $(A_1,u_1):=e^{i\xi}(A_0,u_0)$. By Corollary \ref{cor:connW1psig},
  there is a gauge transformation defined on $\Sig$, $k:\Sig \to K$
  such that $kA_1=\d+\lambda_1 \d\theta + a_1$, where $\lambda_1 \in
  \frac 1 {2\pi}\exp^{-1}(\Id) \subset \k$ and $a_1$ satisfies
  \eqref{eq:twistconnbd}. We modify the map $k:\Sig \to K$ so as to
  achieve $\lambda_1=\lambda$.  By left multiplying $k$ by a constant
  element in $K$, we may assume $\lim_{r \to \infty}k(r,0)=\Id$. In
  the first place, this limit exists because the corresponding limits
  exist both for $u_1$ and $ku_1$ and because $K$ acts freely in a
  neighborhood of $\Phi^{-1}(0)$. The connection $kA_1$ extends to a
  connection on a principal bundle $P'$ over $C$, the transition
  function of $P'$ is $(r,\theta) \mapsto e^{-\lambda_1\theta}$
  defined on $N(z) \bs \{z\}$. By the invertibility of the
  Hitchin-Kobayashi map in Theorem \ref{thm:HK}, the bundles $P$ and
  $P'$ are isomorphic. Therefore the geodesic loops $\theta \mapsto
  e^{\lambda \theta}$ and $\theta \mapsto e^{\lambda_1 \theta}$ are
  homotopic, and hence the loop $\R/2\pi\Z \ni \theta \mapsto
  e^{\lambda \theta}e^{-\lambda_1\theta}$ is homotopic to the
  identity. It follows that one can define a map $k_1:\Sig \to K$ that
  restricts to $(r,\theta) \mapsto e^{\lambda
    \theta}e^{-\lambda_1\theta}$ on the cylindrical end.  By replacing
  the gauge transformation $k$ by $k_1k$, we obtain
  $\lambda_1=\lambda$.

  We now show that $k$ extends to a $p$-bounded gauge transformation
  on $P$.  For any map $u$ with domain $\Sig$, we denote
  $u(\infty,\theta):=\lim_{r \to \infty}u(r,\theta)$ (if it
  exists). Since $\xi(\infty)\equiv \Id$, we have
  $u_0(\infty,\cdot)=u_1(\infty,\cdot)$. Next, observe that
  $(ku_1)(\infty,\theta)=e^{-\lambda \theta}((ku_1)(\infty,0))$, and
  that $u_0$ also satisfies the same relation. Further, since
  $k(\infty,0)=\Id$, we conclude $k(\infty,\theta)\equiv \Id$.
  Therefore, the gauge transformation $k:\Sig \to K$ extends
  continuously over $z$. In particular $k$ extends to $k \in \K(P)$
  obtained by defining $\hat k=e^{-\lambda \theta}ke^{\lambda \theta}$
  on $N(z)\bs \{z\}$, and $\Id$ at $z$. Therefore $\hat k \in
  C^0(N(z)) \hra L^p(N(z))$. The complex gauge transformation
  $e^{i\xi}$ is in $\G(P)_{\bd}^p$, and it follows that $e^{-\lambda
    \theta}e^{i\xi}e^{\lambda \theta} \in W^{1,p}(N(z))$. Then by the
  continuous action of complex gauge transformations on connections
  (Lemma \ref{lem:gcactona}), $\hat a_1 \in L^p(N(z))$, where
  $A_1=\d+\hat a_1$ under the trivialization of $P|_{N(z)}$. By
  Corollary \ref{cor:extendoverinfty}, $kA_1$ extends to $p$-bounded
  connection over $P$. We express $kA_1$ as $kA_1=\d+\hat a$ under
  the trivialization of $P|_{N(z)}$. Then, $\hat a \in L^p(N(z))$
  and $\hat a=\d\hat k \hat k^{-1} + \Ad_{\hat k}\hat
  a_1$. Re-arranging, $\d\hat k = \hat a\hat k + \hat k \hat
  a_1$. The left hand side is in $L^p(N(z))$, so $\hat k \in
  W^{1,p}(N(z))$, and hence $k \in \K(P)_{\on{bd}}^p$.

  {\sc Step 2}: {\em For any positive constant $\gamma_1<\gamma$ there
    exists $c>0$, so that on $N(z) \bs \{z\}$,
    \begin{align}\label{eq:xiC0bd}
      d(\Id,g(r,\theta)) \leq ce^{-\gamma_1 r}.
    \end{align}
  } In Step 1, $p$ can be chosen such that $\gamma_1 \leq 1- \frac 2
  p$. We know $g(z)=\Id$. We obtain the result in Step 2 by showing
  that $\hat g$ is in the H\"older space $C^{0,\gamma_1}$ via the
  following embedding:
$$W^{1,p}(N(z)) \hra C^{0,\gamma_1}(N(z))$$
(see Theorem B.1.11 in \cite{MS}). Here $C^{0,\gamma_1}(N(z))$ is the
space of $C^0(N(z))$ functions with finite H\"older norm
$$\Mod{f}_{C^0,\gamma_1}:=\sup_{x \in U}|f(x)| + \sup_{x,y \in U} \frac {|f(x)-f(y)|}{d(x,y)^{\gamma_1}}.$$
Then,
$$d(\Id,\hat g(r,\theta)) \leq ce^{-\gamma_1 r}.$$
The same holds for $g(r,\theta)=e^{\lambda \theta} \hat g(r,\theta)
e^{-\lambda \theta}$ on $N(z) \bs \{z\}$.

{\sc Step 3}: {\em Finishing the proof.}\\
A connection $A$ on $P$ defines an operator $\delbar_A$ on the
$G$-bundle $P_\C$. A complex gauge transformation $g$ acts on
$\delbar_A$ as: $\delbar_{gA}=g \circ \delbar_A \circ g^{-1}$. In our
context $a:=gA-A$ and
\begin{align}\label{eq:strap}
  \delbar_{gA}-\delbar_A=a^{0,1}=g \delbar_A (g^{-1})=-(\delbar_A
  g)g^{-1} \implies a^{0,1}g=\delbar_A g.
\end{align}
For any $q>1$, \eqref{eq:xiC0bd} implies
$$\Mod{g}_{L^q(\{n-\eps \leq r \leq n+1+\eps\})} \leq ce^{-\gamma_1 n}$$
where $0<\eps<0.1$ is a small constant. Recall $a$ satisfies
$$\Mod{a}_{W^{1,q}(\{n-\eps \leq r \leq n+1+\eps)\}} \leq ce^{-\gamma n} < ce^{-\gamma_1 n}$$
By elliptic bootstrapping using the equation \eqref{eq:strap}, we get
\begin{align*}
  \Mod{g}_{W^{2,q}(\{n \leq r \leq n+1\})} \leq
  c(\Mod{a}_{W^{1,q}(\{n-\eps \leq r \leq n+1+\eps\})} +
  \Mod{g}_{L^q(\{n-\eps \leq r \leq n+1+\eps\})}) \leq ce^{-\gamma_1
    n}.
\end{align*}
Elliptic regularity is applied to domains of the form $\{n \leq r \leq
n+1\})$. Since all the domains are isometric to each other, the
constants $c$ are independent of $n$. By \eqref{eq:complexgroupiso},
we get similar asymptotic bounds on $k$ and $\xi$ also, which proves
the Lemma.
\end{proof}
\begin{lemma}[Gauged maps near a vortex can be complex gauge
  transformed to vortices]\label{lem:vortcont} Suppose $\Sig$ is a
  Riemann surface with a cylindrical end $N(z) \bs \{z\}$
  corresponding to a point $z \in \oSig \bs \Sig$. Suppose there is a
  sequence of gauged holomorphic maps $v_\nu=(A_\nu,u_\nu)$ on $\Sig$
  that converges to a finite energy vortex
  $v_\infty=(A_\infty,u_\infty)$ in the following sense:
$$\cup_{\nu}\on{Im}(u_\nu) \subset X \text{ is pre-compact}, \; A_\nu \xrightarrow{H^1(\Sig)} A_\infty, \; u_\nu \xrightarrow{H^2(\Sig)}u_\infty.$$
% $$\Mod{A_\nu -A_\infty}_{H^1(\Sig)}, \;
% \Mod{d(u_\nu,u_\infty)}_{L^\infty(\Sig)}, \;
% \Mod{\Phi(u_\nu)-\Phi(u_\infty)}_{L^2(\Sig)}\to 0 \quad \text{as
% $\nu \to 0$.}$$
The norms are taken in a trivialization $\Sig \times K$ in which
$A_\infty=\d+\lambda \d\theta + a_\infty$ on the cylindrical end $N(z)
\bs \{z\}$ and $a_\infty$ satisfies the exponential bound
\eqref{eq:twistconnbd}.  Then we can find complex gauge
transformations $e^{i\xi_\nu}$ so that $\xi_\nu \to 0$ in $H^2(\Sig)$
and $e^{i\xi_\nu}v_\nu$ is a vortex.
\end{lemma}
\begin{proof} The proof is by an Implicit Function Theorem argument on
  the functional $\F_v$ defined below. For a gauged holomorphic map
  $v:=(A,u)$, define
  \begin{align*}
    \F_v:\Gamma(\Sig,\k) \to \Gamma(\Sig,\k) \quad \xi \mapsto
    *F_{e^{i\xi}A} + \Phi(e^{i\xi}u).
  \end{align*}
  Its linearization at $\xi=0$ is
  \begin{equation}\label{eq:diffsmooth}
    \D\F_v(0):\xi \mapsto \d_A^*\d_A \xi + u^*\d\Phi(J\xi_X).
  \end{equation}

  We first show that $\D\F_v(0)$ extends to an operator between
  Sobolev completions.  The operator $\d_{A_\infty}$ extends to a
  continuous operator $\d_{A_\infty}: H^{k+1}(\Sig,\k) \to
  H^k(\Sig,\k)$ for $k=0,1$. This is because on the cylinder $N(z) \bs
  \{z\}$, $A_\infty =\d+\lambda \d\theta + a_{\infty}$, $a_\infty$ is
  bounded in $H^1(N(z)\bs \{z\})$, so by Sobolev multiplication
  (Proposition \ref{prop:sobmult}) the operator $\xi \mapsto \d\xi +
  [\lambda ,\xi]\d\theta + [a_{\infty},\xi]$ is continuous from
  $H^{k+1} \to H^k$, $k=0,1$. Since $\Mod{A_\infty -
    A_\nu}_{H^1(\Sig)}$ is bounded, the same is true for the operators
  $\d_{A_\nu}$ also. Writing $\d_A^*=*\d_A*$, we see that
  $\d_{A_\nu}^*:H^{k+1}(\Sig,\k) \to H^k(\Sig,\k)$ is also a bounded
  operator for $k=0,1$ and for all $\nu$ including $\nu=\infty$. For
  $x \in X$, let
$$L_x \in \End(\k) := \xi \mapsto \d\Phi(J\xi_x).$$
The image of each $u_\nu$ is compact, so $\sup_{x \in
  \on{Im}(u_\nu)}|L_x|$ is bounded. So, the operator
$u_\nu^*\d\Phi(J\xi_X):H^2 \to L^2$ is bounded because it is
multiplication by a $C^0$ bounded function.  Therefore, for all $\nu$,
$$\D\F_{v_\nu}:H^2(\Sig,\k) \to L^2(\Sig,\k)$$
is a bounded operator.

{\sc Step 1}: {\em For all $\nu$, including $\nu=\infty$, $\D\F_{v_\nu}:H^2 \to L^2$ is invertible and the norm of the inverses are uniformly bounded. That is, there exists a constant $C$ such that $\Mod{\D\F_{v_\nu}^{-1}} \leq c$ for all $\nu$.} \\
The operator $\D\F_{v_\nu}:H^2 \to L^2$ is injective because for $\xi
\in H^2(\Sig,\k)$,
\begin{equation*}
  \int_\Sig\langle \d_{A_\nu}^*\d_{A_\nu}\xi+u_\nu^*\d\Phi(J(\xi)_{u_\nu}),\xi\rangle_\k=\Mod{\d_{A_\nu}\xi}_{L^2}^2 + \int_\Sig\omega_{u_\nu}((\xi)_{u_\nu},J(\xi)_{u_\nu}).
\end{equation*}
The right hand side is positive if $\xi \neq 0$.

We first prove invertibility of the operator in the case that $v$ is
equal to $v_\infty$, which is a vortex on $\Sig$. Recall that
$A_\infty=\d+\lambda \d\theta +a_\infty$ on $N(z)\bs \{z\}$. Next, we
observe that for $x \in X$, if the infinitesimal action of $\k$ is
free, the map $L_x:\k \to \k$ is positive and symmetric for $x \in
X^\ss$ because $\lan L_x \xi,\xi\ran=\om_x(\xi_x,J\xi_x)$. On the
cylindrical end $N(z) \bs \{z\}$, we use the co-ordinates $r$,
$\theta$ given by $\rho_z$. By Proposition \ref{prop:infvortsing},
$u(\infty,\theta):=\lim_{r \to \infty}u(r,\theta)$ is well-defined and
$u(\infty,\theta)=e^{\lambda \theta}u(\infty,0)$. By deforming $L_u$,
one can produce a smooth section $\L:\Sig \to \Aut(\k)$ such that on
the cylindrical end, $\L(r,\theta)=L(u(\infty,\theta))$, and $\L(x)$
is positive and symmetric for all $x \in \Sig$. The connection
$A_\infty$ can be deformed to a smooth connection $A_\infty'$ such
that $A_\infty'=\d+\lambda \d\theta$ on the cylindrical end $N(z) \bs
\{z\}$. By Proposition \ref{prop:lapiso},
$\d_{A'_\infty}^*\d_{A'_\infty}+\L:H^2(\Sig,\k) \to L^2(\Sig,\k)$ is
an isomorphism. The operators
$\d_{A'_\infty}^*\d_{A'_\infty}-\d_{A_\infty}^*\d_{A_\infty}$ and $\L
- L_u$ are lower order operators that decay to zero on the cylindrical
end, therefore they are compact by Proposition \ref{prop:deccomp}. So,
$\D\F_{v_\infty}(0)=\d_{A_\infty}^*\d_{A_\infty}+L_u$ has the same
Fredholm index as $\d_{A'_\infty}^*\d_{A'_\infty}+\L$, which is
$0$. The operator is invertible because it is injective.

Next, we consider $v_\nu$ in the sequence. For large $\nu$,
\begin{equation}\label{eq:diffdiff}
  \begin{split}
    \Mod{\d_{A_\infty}^*\d_{A_\infty} - \d_{A_\nu}^*\d_{A_\nu}} &\leq c\Mod{A_\infty-A_\nu}_{H^1(\Sig)},\\
    \Mod{L_{u_\infty}-L_{u_\nu}} &\leq c\sup_{p \in
      \Sig}d_X(u_\infty(p),u_\nu(p)).
  \end{split}
\end{equation}
This means $\Mod{\D\F_{v_\infty}(0)-\D\F_{v_\nu}} \leq
\hh\Mod{(\D\F_{v_\infty})^{-1}}^{-1}$ for large enough $\nu$. For such
$\nu$, $\D\F_{v_\nu}(0)$ is invertible and $\Mod{\D\F_{v_\nu}(0)^{-1}}
\leq 2\Mod{\D\F_{v_\infty}(0)^{-1}}$. The constant $C$ is
$2\Mod{\D\F_{v_\infty}(0)^{-1}}$.

{\sc Step 2:} {\em There is a constant $\delta>0$ such that if $\Mod{\xi}_{H^2(\Sig,\k)}<\delta$, then $\Mod{\D\F_\nu(\xi)-\D\F_\nu(0)} \leq \frac 1 {2C}$ for large enough $\nu$.}\\
This follows from inequalities as in \eqref{eq:diffdiff} and the fact
that there exist constants $c$, $\eps>0$ such that if a connection $A$
satisfies $\Mod{A-A_\infty}_{H^1(\Sig)} \leq \eps$, then
$\Mod{e^{i\xi}A-A}_{H^1} \leq c\Mod{\xi}_{H^2(\Sig)}$ (see Lemma
\ref{lem:gcactona}).

{\sc Step 3:} {\em Completing the proof.} By the implicit function
theorem (Proposition \ref{prop:impfn}), if $\Mod{*F_{A_\nu} +
  \Phi(u_\nu)}_{L^2}<\frac \delta {4C}$, then there exists $\xi_\nu$
such that $e^{i\xi_\nu}v_\nu$ is a vortex and $\Mod{\xi_\nu}_{H^2}
\leq C\Mod{*F_{A_\nu}+\Phi(u_\nu)}_{L^2}$. Since
$*F_{A_\nu}+\Phi(u_\nu) \to 0$ in $L^2$, the Lemma is proved.
\end{proof}

The following result is used in the proof of Lemma
\ref{lem:vortcont}. It proves a certain differential operator is
invertible on a surface with cylindrical ends.
\begin{proposition}[Invertibility of an elliptic
  operator]\label{prop:lapiso} Let $\Sig$ be a surface with a
  cylindrical end and $\gamma>0$ be a constant. Let $A$ be a smooth
  connection on the trivial $K$-bundle $\Sig \times K$, such that on
  the cylindrical end $N \simeq S^1 \times (0,\infty)$, the connection
  is $A=\d + \lambda \d\theta$, where $\lambda \in \k$. Further, let
  $L:\Sig \to \Aut(\k)$ be a smooth map that is invariant under translation by $\R$ on the
  cylindrical ends. Suppose $L(x)$ is positive and symmetric for all $x
  \in \Sig$.  Then, the operator
  \begin{equation}\label{eq:idlap}
    L + \d_A^*\d_A: H^{k+1}(\Sig,\k) \to H^{k-1}(\Sig,\k)
  \end{equation}
  is an isomorphism for all integers $k \geq 0$.
\end{proposition}
\begin{proof} We first show that \eqref{eq:idlap} is a bounded
  operator. For any $l \in \Z_{\geq 0}$, the map $L$ is bounded in
  $C^l$ because $\Sig \bs N$ is compact and on the cylindrical end
  $N$, $L$ is $\R$-invariant.  The norm of the operator $\d_A^*\d_A$ in
  \eqref{eq:idlap} is bounded because of the form of the connection
  $A$ and the multiplication Theorem (Proposition
  \ref{prop:sobmult}). Together, these imply that
  \eqref{eq:idlap} is a bounded operator.

  We prove the proposition for $k=0$ by the Lax-Milgram theorem (See
  Section 6.2, Theorem 1 in \cite{Evans}). Recall that $H^{-1}(\Sig)$
  is the dual of $H^1(\Sig)$ under the $L^2$-pairing. The bilinear
  operator
  $$B:H^1(\Sig,\k) \times H^1(\Sig,\k) \to \R, \quad (\xi,\xi') \mapsto \int_\Sig\lan (L+\d_A^*\d_A)\xi,\xi'\ran$$
  satisfies $B(\xi,\xi') \leq \Mod{\xi}_{H^1}\Mod{\xi'}_{H^1}$. Next,
  observe
  $$B(\xi,\xi) \geq c_1^{-1}(\Mod{\d_A\xi}^2_{L^2}+ \Mod{\xi}^2_{L^2}) \geq c_2^{-1}(\Mod{\xi}^2_{L^2} + \Mod{d\xi}^2_{L^2}) = c_2^{-1}\Mod{\xi}^2_{H^1},$$
  where $c_1$ and $c_2$ are constants independent of $\xi$. The first
  inequality is based on the fact that $L(x)^{-1} \in \Aut(\k)$ is
  uniformly bounded for all $x \in \Sig$ because the image of
  $L^{-1}:\Sig \to \Aut(\k)$ is bounded in $\Aut(\k)$.  By the Lax
  Milgram Theorem, for any $f \in H^{-1}=(H^1)^*$, there is a unique
  $\sig \in H^1$ such that $B(\sig,\sig')=\int_\Sig \lan
  f,\sig'\ran$. So, $(L+\d_A^*\d_A)\sig=f$. This proves surjectivity and
  injectivity for \eqref{eq:idlap}. By the open mapping theorem, the
  inverse map is bounded.

  Now, we consider $k>0$. By the $k=0$ result, the operator is
  injective for all $k$. Given $f \in H^{k-1}(\Sig,\k) \hra
  H^{-1}(\Sig,\k)$, there is a $\sig \in H^1(\Sig,\k)$ satisfying $(L
  + \d_A^*\d_A)\sig=f$. By elliptic regularity for compact sets $\sig
  \in H^{k+1}_{loc}$. It remains to bound $\sig$ in $H^{k+1}$.  For
  this, write $\Sig= \Sig_0 \cup N$, where $\Sig_0 \subset \Sig$ is
  pre-compact, the cylindrical end $N$ is equipped with an isometry
  $\rho:N \to [0,\infty) \times S^1$. Let $\Sig_0'$ be a pre-compact
  set containing $\Sig_0$. We also define the sets
  $$S_n:=\rho^{-1}((n,n+1)\times S^1), \quad S_n':=\rho^{-1}((n-\eps,n+1+\eps)\times S^1) \subset N,$$
  where $0<\eps<0.1$ is a fixed constant. For any set $S=\Sig_0$ or
  $S_n$ and correspondingly $S'=\Sig'_0$ or $S'_n$, there is a
  constant $c_S$ such that
  $$\Mod{\sig}_{H^{k+1}(S)} \leq c_S(\Mod{\sig}_{H^{k-1}(S')} + \Mod{f}_{H^{k-1}(S')})$$
  The constants $c_{S_n}$ are identical because the geometry of the
  sets and the connection on them is independent of $n$. So, the set
  $\{c_{\Sig_0}, c_{S_n}\}$ is finite and has a finite maximum. So, we
  have $\Mod{\sig}_{H^{k+1}(\Sig)} \leq c\Mod{f}_{H^{k-1}(\Sig)}$ for
  some constant $c$. This finishes the proof.
\end{proof}

\subsection{Continuity: Nodal degeneration} \label{sec:nodecont}

In this section we prove the following Proposition.
\begin{proposition} [Continuity of $\Psi$ at a nodal
  curve]\label{prop:nodal} Suppose $Q_S=(\mC \to
  S,\{z_j:S \to \mC\}_{j=1,\dots,k}, \mP_\C,u)$ is a family of quasimaps parametrized by $S$,
  where $S \subset \C$ is a neighborhood of $0$. We assume that
  $\mC_0$ is a curve with a single node $w$ and $\mC_s$ are smooth
  curves obtained from $\mC_0$ by the gluing procedure. Then for any
  sequence $s_\nu \to 0$ as $\nu \to \infty$, the sequence of vortices
  $\Psi(Q_{s_\nu})$ Gromov converges to $\Psi(Q_0)$. ($\Psi$ is as
  described in Corollary \ref{cor:biject}.)
\end{proposition}
\begin{remark}\label{rem:extendnodal} The proof of the above
  Proposition carries over word for word in the case when the curve
  $\mC_0$ has more than one nodal singularities. The hypotheses of the
  Proposition also require that away from cylindrical parts, the
  curves are isometric under the neck-stretching metric. The proof
  carries over if instead there were a smooth family of metrics
  instead. In that case $S$ can be a neighborhood of the origin in
  $\C^N$ rather than in $\C$.
\end{remark}
With this result in hand we can prove Theorem \ref{thm:main}.
\begin{proof}[Proof of Theorem \ref{thm:main}] The map $\Psi$ is
  constructed by Corollary \ref{cor:biject}. It is a bijection and the
  spaces $\ol{MV}^K_{g,n}(X,\beta)$ and $Qmap(X \qu G, \beta)$ are
  compact and Hausdorff. So, to show that $\Psi$ is a homeomorphism,
  it is enough to prove that it is continuous.

  For any family of quasimaps $Q_S$ parametrized by $S$ a neighborhood
  of the origin in $\C^N$, we prove that for a sequence $s_\nu \to 0$
  in $S$, the gauge equivalence classes of stable vortices
  $\Psi(Q_{s_\nu})$ Gromov converge to $\Psi(Q_0)$. If the domain
  curve $\mC_0$ (and hence $\mC_{s_\nu}$) does not have nodal
  singularities, then this is proved by Proposition
  \ref{prop:contsmooth}. This will also prove the result if the
  $\mC_0$ has the same combinatorial type as $\mC_{s_\nu}$, in that
  case Proposition \ref{prop:contsmooth} can be applied
  component-wise. Next, suppose that $\mC_0$ is in a higher stratum
  than $\mC_{s_\nu}$, i.e. $[\mC_{s_\nu}] \preceq [\mC_0]$. Again, we
  can assume the curves $\mC_{s_\nu}$ to be smooth, otherwise the same
  arguments can be carried out component-wise. In that case, the
  result is proved by Proposition \ref{prop:nodal} and Remark
  \ref{rem:extendnodal}.
\end{proof}

Suppose $\{\Sig_s\}_{s \in S}$ are the Riemann surfaces with
cylindrical ends corresponding to the family $\mC \to S$ in
Proposition \ref{prop:nodal}. By Lemma \ref{lem:convprep},
there is a cover $\pi_s: \tSig_s \to \Sig_s$ such that $\tSig_s$ can be diffeomorphically\marginpar{*****}
identified to an open subspace of the central Riemann surface
$\Sig_0$. We recall the precise definition of the covers
$\tSig_s$. Let $\tilde \mC_0$ be the normalization of $\mC_0$ with
marked points $w^+$ and $w^-$ corresponding to the node. For $s \in
S\bs \{0\}$, let 
\begin{equation}
  \label{eq:lsts}
  l_s=L_s +it_s:=-\ln s.
\end{equation}
 The Riemann surface with
metric $\Sig_s$ is obtained by gluing:
\begin{equation}\label{eq:gluedfamily}
  \Sig_s=(\Sig_0 \bs (\rho^{-1}_{w^+}\{r>L_s\} \cup \rho^{-1}_{w^-}\{r<-L_s\}))/\sim, \quad \rho^{-1}_{w^-}(z) \sim \rho^{-1}_{w^+}(z+l_s).
\end{equation}
We denote
\begin{equation}\label{eq:deftsig}
  \tSig_s:=\Sig_0 \bs (\rho^{-1}_{w^+}\{r>L_s\} \cup \rho^{-1}_{w^-}\{r<-L_s\}).
\end{equation}
For $s=0$, $\tSig_0:=\Sig_0$. The construction of $\tSig_s$ can be performed in a family to produce a four-dimensional manifold $\tSig$ with a map $\tSig \to S$, such that the fiber over $s \in S$ is $\tSig_s$. 

We need some notation for the proof of Proposition \ref{prop:nodal}.
\begin{notation}\label{note:bundleglue} Let $\mC \to S$ be a family of
  curves as in Proposition \ref{prop:nodal}. For any $s \in S$, let $\Sig_s$
  be the Riemann surface with cylindrical ends corresponding to the curve $\mC_s$. Suppose
  $\mP \to \mC$ is a principal $K$ bundle.
  \begin{Notes}
  \item {\rm(Cylindrical part $N_s$)} Let
    $N_s$ be the cylindrical part in $\Sig_s$ corresponding to the
    node $w$. By the definition of $\Sig_s$ in \eqref{eq:deftsig}, the restrictions of $\rho_{w^\pm}$ give two co-ordinates on $N_s$.
    The co-ordinates are related as: $\rho_{w^+}-l_s=\rho_{w^-}$.

  \item {\rm (Core $\Om_s$, sleeve length $\Delta$, sleeves
      $\Delta_s^\pm$, map relating sleeves $r:\Delta_s^+ \to
      \Delta_s^-$)} The {\em sleeve length} $\Delta$ is a large
    positive constant that will be determined in the proof of Lemma \ref{lem:vortcontnodal}. The
    {\em core} $\Om_s$ and {\em sleeves} $\Delta_s^\pm$ are subspaces
    of $\tSig_s$:
    \begin{align*}
      \Om_s&:=\tSig_s \bs (\rho_{w^+}^{-1}\{r> (L_s-\Delta)/2 \} \cup \rho_{w^-}^{-1}\{r< (-L_s+\Delta)/2 \}),\\
      \Delta^+_s&:=\rho^{-1}_{w^+}\{ (L_s -\Delta)/2 \leq r \leq (L_s+\Delta)/2\},\\
      \Delta^-_s&:=\rho^{-1}_{w^-}\{ (-L_s -\Delta)/2 \leq r \leq
      (-L_s+\Delta)/2\}.
    \end{align*}
    There is a diffeomorphism between the sleeves $r:\Delta_s^+ \to
    \Delta_s^-:= \rho^{-1}_{w^-} \circ (\cdot-l_s) \circ\rho_{w^+}$,
    such that $\pi_s \circ r=\pi_s$ on $\Delta_s^+$. We remark that
    the projection $\pi_s$ restricted to the core $\Om_s$ is
    injective. Let $\Delta_s$ be the image $\pi_s(\Delta_s^+)$ (or
    $\pi_s(\Delta_s^-)$, which is the same). Then, $\Sig_s=\Om_s \cup
    \Delta_s$. See Figure \ref{fig:sleef}.

  \item{\rm (A cover $\Om_s'$)}\label{notecover} We define $\Om_s':=\Om_s \cup
    \Delta^+_s \cup \Delta^+_s \subset \tSig_s$, which is a cover of
    $\Sig_s$ and a subset of the central nodal curve $\Sig_0$.
  \item {\rm(Lifts $\pi_s^{-1}$, $\pi_{s,\pm}^{-1}$)} Recall that the
    covering map $\pi_s :\Om_s' \to \Sig_s$ is 1-1 everywhere except
    on the sleeves $\Delta_s^\pm$ where it is 2-1. So, a lift
    $\pi_s^{-1}$ is well-defined on $\Sig_s\bs \Delta_s$. On
    $\Delta_s$, there are two right inverses denoted by
    $\pi_{s,\pm}^{-1}$ mapping to $\Delta_s^\pm$.

  \item{\rm (Families of cores $\Om$ and covers $\Om'$)} The
    construction of the core $\Om_s$ and the cover $\Om_s'$ can be
    performed in a family to produce four-dimensional manifolds $\Om$
    and $\Om'$ respectively with maps $\Om,\Om' \to S$, such that the
    fibers over $s \in S$ are $\Om_s$ and $\Om_s'$ respectively.
  \item {\rm (Trivialization of bundles $\mP_s$, transition function
      $\kappa_s:\Delta_s \to K$)} \label{note2} The bundle over the
    family of covers $\pi^*\mP \to \Om'$ is trivializable because
    $\Om'$ is homotopic to a one-skeleton. A choice of trivialization
    gives trivializations of the bundles $\pi_s^*\mP_s \to
    \Om_s'$. Then, for any $s$, using the embedding $\Om'_s \to
    \Sig_0$ gauged holomorphic maps defined over the bundle $\mP_s \to
    \Sig_s$ can be pulled back to the cover $\Om_s'$ and can be
    compared to gauged holomorphic maps defined over the trivial
    bundle $\mP_0 \to \Sig_0$.  Given a trivialization of the bundle
    $\pi_s^*\mP_s \to \Om_s'$, the bundle $\mP_s \to \Sig_s$ can be recovered
    by gluing the bundle on sleeves $\pi_s^*\mP_s|_{\Delta_s^\pm}$ via a
    transition function $\kappa_s:\Delta_s \to K$.
  \item {\rm(Cut-off function $\phi_s$ on the cover $\Om_s'$)} We
    define a cut-off function $\phi_s:\Om_s' \to [0,1]$ that is a
    partition of unity of the cover $\pi_s:\Om_s' \to \Sig_s$. That
    is, for any $x \in \Sig_s$, $\sum_{\tilde x \in
      \Om_s':\pi_s(\tilde x)=x}\phi_s(\tilde x)=1$.  Suppose $\phi:\R
    \times S^1 \to [0,1]$ is a cut-off function invariant in the
    second coordinate, $0$ on $(-\infty,-\Delta/2]$ and $1$ on
    $[\Delta/2,\infty)$, and that satisfies $\phi(r,\theta)+\phi(-r,\theta)=1$ for all $(r,\theta)$. Let $\phi_z(x):=\phi(x-z)$ for any $z\in
    \R$. Define
    \begin{equation*}
      \phi_s:=\begin{cases}
        1 &\text{ on $\Om_s$}\\
        (1-\phi_{L_s/2}) \circ \rho_{w^+}& \text{ on $\Delta_s^+$}\\
        \phi_{L_s/2} \circ \rho_{w^-} &\text{ on $\Delta_s^-$}
      \end{cases}
    \end{equation*}
  \item {\rm(Norm on the space $W^{k,p}(\Sig_s, \mP_s(\k))$)}\label{notenorm} Let $k
    \geq 0$ be an integer. There are many equivalent ways of defining
    a norm on $W^{k,p}(\Sig_s,\mP_s(\k))$, of which we choose the
    following.  Fix a trivialization of the family $\pi^*\mP \to
    \Om'$.  The resulting trivialization of the bundle $\pi_s^*\mP_s
    \to \Om_s'$ produces a trivialization of the associated bundle
    $\pi_s^*\mP_s(\k) \simeq \Om_s' \times \k$. We can now define
$$\Mod{\sig}_{W^{k,p}(\Sig_s,\mP_s(\k))}:=\Mod{\pi_s^*\sig}_{W^{k,p}(\Om'_s,\k)}.$$
\item {\rm(Push-forward)} For a section $\sig \in \Gamma(\Om'_s, \k)$ that is supported away from the boundary $\partial \Om_s'$, the push-forward $(\pi_s)_*\sig \in \Gamma(\Sig_s,\mP(\k))$ is

 \begin{equation*}
    (\pi_s)_* \sig :=\begin{cases} 
      (\pi_s^{-1})^*\sig &\text{ on $\Om_s$}\\
      (\pi_{s,+}^{-1})^*\sig + (\pi_{s,-}^{-1})^*\sig &\text{ on $\Delta_s$.}
    \end{cases}
  \end{equation*}

\end{Notes}
\qed \end{notation}
\begin{figure}
  \centering \scalebox{.75}{ \input{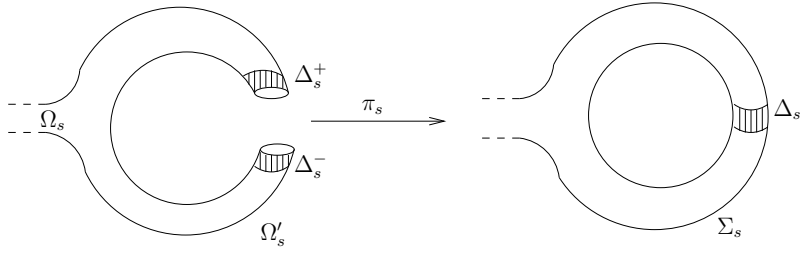}}
  \caption{The covering map $\pi_s$ glues the sleeves $\Delta_s^\pm$.}
  \label{fig:sleef}
\end{figure}
\begin{remark} We perform a calculation related to Notation \ref{note:bundleglue} that will be useful in the proof of Lemma \ref{lem:vortcontnodal}. Let $s \in S \bs \{0\}$ and $\tilde \xi \in \Gamma(\Om_s',\k)$ be a section that is supported away from the boundary. Let $\xi:=(\pi_s)_*\tilde \xi$. Then, we claim
  \begin{equation}
    \label{eq:pilu}
    \Mod{\xi}_{L^2(\Sig_s,\mP_s(\k))} \leq 2\Mod{\tilde \xi}_{L^2(\Om_s',\k)}.
  \end{equation}
The proof is as follows. By the definition in Notation 
\ref{note:bundleglue} \eqref{notenorm}, we have 
\begin{equation}
  \label{eq:split3}
\Mod{\xi}_{L^2(\Sig_s,\mP_s(\k))}:=\Mod{\pi_s^*\xi}_{L^2(\Om_s',\k)} \leq \Mod{\tilde \xi}_{L^2(\Om_s,\k)} + \Mod{\pi_s^*\xi}_{L^2(\Delta_s^+,\k)} + \Mod{\pi_s^*\xi}_{L^2(\Delta_s^-,\k)}.
\end{equation}
On $\Delta_s^+$, we have $\pi_s^*\xi=\tilde \xi + r^*(\Ad_{\kappa_s}\tilde
\xi|_{\Delta_s^-})$. Unitary gauge transformations preserve the
$L^2$-norm, so we get $\Mod{r^*(\Ad_{\kappa_s}\tilde
  \xi|_{\Delta_s^-})}_{L^2(\Delta_s,\k)}=\Mod{\tilde
  \xi}_{L^2(\Delta_s^-)}$. So, we conclude
$\Mod{\pi_s^*\xi}_{L^2(\Delta_s^+,\k)} \leq \Mod{\tilde
  \xi}_{L^2(\Delta_s^+)} + \Mod{\tilde
  \xi}_{L^2(\Delta_s^-)}$. Writing down an analogous bound for
$\Mod{\pi_s^*\xi}_{L^2(\Delta_s^-,\k)}$ and using \eqref{eq:split3},
we get the result \eqref{eq:pilu}.
\end{remark}

\begin{remark}[Model for $\mC \to S$ near the node
  $w$]\label{rem:localmodel} Let $\mC \to S$ be the family of curves
  in Proposition \ref{prop:nodal}. A neighborhood of the node $w$ in $\mC$ can be mapped bi-holomorphically to a neighborhood of the origin in $\C^2$ in a way that $w$ is mapped to the origin and the curve $\mC_s$ is mapped to $\{xy=s\}$ for all $s \in
  S$ (see Section 3, Chapter 11 in \cite{Arb}). We call this map $\Upsilon$ and describe it explicitly. For smooth curves, i.e. for $s \neq 0$,
  \begin{align*}
    \Upsilon: N_s \to \{xy=s\} \subset \C^2 \quad z \mapsto
    (e^{\rho_{w^+}(z)-l_s},e^{-\rho_{w^-}(z)}).
  \end{align*}
  For $s=0$,
$$ \Upsilon:\begin{cases}N(w^+) \to \{x=0\} \subset \C^2 & z \mapsto (0,e^{-\rho_{w^+}(z)})\\
  N(w^-) \to \{y=0\} \subset \C^2 & z \mapsto (e^{\rho_{w^-}(z)},0)
\end{cases} $$ Denote by $\mN_w \subset \mC$ the open neighborhood
$\{|x|,|y|<1\}$. Then, we see that $\cup_{s \in \C, |s|<1}N_s \cup
N(w^\pm) \cup \{w\}$ is identified to $\mN_w$ via $\Upsilon$. Figure \ref{fig:Om's} is related to this construction.
\end{remark}
The next Lemma is the analog of Lemma \ref{lem:stdform} in the presence
of a node. Given a family of quasimaps where the central curve has a node and the other curves are smooth, there is a reduction of the structure group such that the resulting $K$-gauged maps are in standard form at all infinite cylindrical ends. In addition, there is a trivialization of the family of bundles $\mP$ pulled back to the family of covers $\Om'$ such that the $K$-gauged maps have desirable convergence properties and the transition functions $\kappa_s$ have a uniform $C^2$ bound.
\begin{lemma}[Standard form near a node $w$]\label{lem:stdform_node}
  Let $p\geq 2$. Given a family of quasimaps $Q_S=(\mC \to S,\{z_j:S
  \to \mC\}_{j=1,\dots,k},\mP_\C,u)$ where $\mC \to S$ as in
  Proposition \ref{prop:nodal}, there is a reduction $\sig:S \to
  \mP_\C/K$ such that the following are satisfied. Let $\mP \to \mC$
  be the principal $K$-bundle, and $\{(A_s,u_s)\}_{s \in S}$ be the
  family of $K$-gauged holomorphic maps induced by the reduction
  $\sig$. There is a trivialization of $\pi^*\mP \to \Om'$ such that
  \begin{enumerate}
  \item \label{part:nodea} $(A_s,u_s)$ is in standard form about
    marked points (as in Lemma \ref{lem:stdform}).
  \item \label{part:nodeb} On the cylindrical ends $N(w^\pm)\bs
    \{w^\pm\}$, $A_0=\d+\lambda^\pm \d\theta$ and $\Phi(u(w))=0$.
  \item \label{part:nodec} Under this trivialization, as $s \to 0$,
    \begin{equation}\label{eq:convtau0}
      \begin{split}
        &\cup_{s \in S} u_s(\Om'_s) \subset X \text{ is pre-compact},\\
        &\Mod{A_s - A_0}_{W^{1,p}(\Om'_s )}, d_{W^{2,p}(\Om'_s
          )}(u_s,u_0) \to 0.
      \end{split}
    \end{equation}
  \item \label{part:noded} the functions $\kappa_s:\Delta_s \to K$ are
    uniformly bounded in $C^2$. ($\kappa_s$ is defined in Notation
    \ref{note:bundleglue}, \eqref{note2})
  \end{enumerate}
\end{lemma}
\begin{proof} 
{\sc Step 1}: {\em On a neighborhood of $w$ in $\mN(w)
    \subset \mC$ there is a reduction $\sig:\mN(w) \to \mP_\C/K$ and a
    trivialization $\tau_1$ of $\mP \to \mN(w)$ such that
    \begin{equation}\label{eq:convtau1}
      % \begin{split}
      \Mod{A_s - A_0}_{W^{1,p}(\Om'_s \cap N(w^\pm))}, d_{W^{2,p}(\Om'_s \cap N(w^\pm))}(u_s,u_0) \quad \text{as $s \to 0$.}
      % \end{split}
    \end{equation}}
  We will prove the statement for the positive ends $N(w^+)$, that for the negative end will follow analogously. We focus on a neighborhood of $w$ in $\mC$ and work with the local model described in Remark \ref{rem:localmodel}. Pick a reduction $\sig:\mC \to \mP_\C/K$, so we have a principal $K$-bundle $\mP \to \mC$. For each $s$, we have a $K$-gauged holomorphic map $(A_s,u_s)$ defined on $\mP_s \to \mC_s$. Pulling back, $(\hat A_s, \hat u_s):=(\Upsilon^{-1})^*(A_s,u_s)$ is defined on a $K$-bundle over $\{xy=s\} \cap \mN(w)$. The reduction $\sig:\mC \to \mP_\C/K$ can be chosen such that $u(w) \in \Phi^{-1}(0)$ and in a neighborhood of $(0,0)$, $F_{\hat A_0}|_{x=0}$, $F_{\hat A_0}|_{y=0}=0$. We deduce the convergence properties of $A_s$, $u_s$ working in the local model -- these will continue to hold on $\Sig$ because for a 1-form $a$, the norms $\Mod{a}_{L^p}$ and $\Mod{\nabla a}_{L^p}$ decrease when there is a conformal change of coordinates that increases the volume. The same is true for $\Mod{\nabla u}_{L^p}$ and $\Mod{\nabla^2 u}_{L^p}$. The $L^p$ convergence of $u_s$ is worked out separately. 

  Pick a smooth trivialization $\tau_1$ of the bundle $\mP|_{\mN(w)}$
  so that $\hat A_0|_{x=0}$, $\hat A_0|_{y=0}$ are the trivial
  connections.  In order to compare $\hat A_s$ and $\hat A_0$, recall
  points in $\{xy=s\}$ are identified to those in $\{x=0\}$ as $(x,y)
  \mapsto (0,y)$. The connection $\hat A_s$ is pulled back to a connection on
  $\{x=0\}$ and this is also called $\hat A_s$. By smoothness of the
  reduction $\sig$ and trivialization $\tau_1$, we have
$$\Mod{A_s - A_0}_{W^{1,p}(\Om'_s \cap N(w^+))} \leq \Mod{\hat A_s- \hat A_0}_{W^{1,p}(\{x=0, e^{(-L_s-\Delta)/2} \leq |y| \leq 1\})} \leq c_1e^{2\Delta}e^{(-L_s+\Delta)/2}.$$ 
(See Figure \ref{fig:Om's}.) This shows $\Mod{A_s -
  A_0}_{W^{1,p}(\Om'_s \cap N(w^+))} \to 0$. To study $u_s$, by
possibly shrinking the neighborhood $\mN(w)$, we can assume that the
image $u(\mN(w))$ is contained in a single chart of $X$, which is
identified to a neighborhood in $\C^N$.
$$\Mod{\nabla u_s - \nabla u_0}_{W^{1,p}(\Om'_s \cap N(w^+))} \leq \Mod{\nabla \hat u_s- \nabla \hat u_0}_{W^{1,p}(\{x=0, e^{(-L_s-\Delta)/2} \leq |y| \leq 1\})} \leq c_1e^{2\Delta}e^{(-L_s+\Delta)/2}.$$ 
Further,
$$\Mod{d(u_s,u_0)}_{L^\infty(\Om'_s \cap N(w^+))} = \Mod{d(\hat u_s,\hat u_0)}_{L^\infty(\{x=0, e^{(-L_s-\Delta)/2} \leq |y| \leq 1\})} \leq c_1 e^{(-L_s+\Delta)/2}.$$
This implies
\begin{align*}
  \Mod{u_s-u_0}_{L^p(\Om'_s \cap N(w^+))} \leq
  c_1L_s^{1/p}e^{(-L_s+\Delta)/2}.
\end{align*}
As $s \to 0$, $L_s \to \infty$ and so, $\Mod{u_s - u_0}_{L^p(\Om'_s
  \cap N(w^+))} \to 0$.

\begin{center}
  \begin{figure}
    \includegraphics{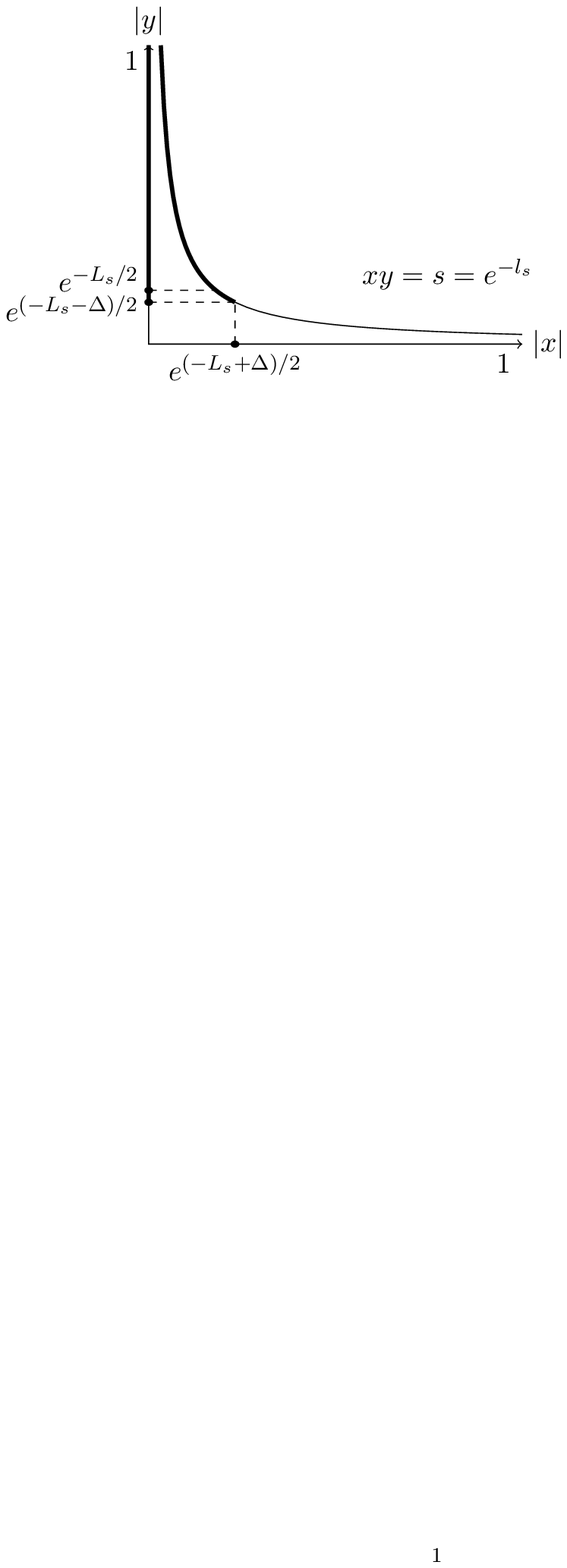}
    \caption{The region $\Om_s'$ on the curves $\mC_s$ and
      $\mC_0$. $L_s=\on{Re}(l_s)$.}
    \label{fig:Om's}
  \end{figure}
\end{center}

{\sc Step 2}: {\em The reduction in Step 1 can be extended to $\sig:\mC \to \mP_\C/K$. There is a trivialization $\tau_2$ of $\pi^*\mP \to \Om'$ such that \eqref{part:nodea}-\eqref{part:noded} are satisfied.}\\
We remark that after making a choice of reduction, the bundle
$\pi^*\mP \to \Om'$ is trivializable because $\Om'$ is homotopic to a
one-skeleton. On the other hand, $\mC$ is made up of a 4-cell $\mN(w)$
glued to $\Om$ ($\Om$ is a deformation retract of $\Om'$), so $\mP \to
\mC$ is not trivializable. In Step 1, we found a trivialization
$\tau_1$ over $\mN(w)$. By comparing $\tau_1$ and $\tau_2$ in the
intersections, we obtain the functions $\kappa_s:\Delta_s \to K$.

First, we describe $\sig$ and $\tau_2$, and list the properties they
automatically satisfy. By Lemma \ref{lem:stdform}, there is a
reduction in the neighborhood of marked points, $\sig : \cup_{s,j}
N(z_j)_s \to \mP_\C/K$ such that the connections $F_{A_s}$ are flat in
these neighborhoods and $\Phi(u(z_j(s)))=0$. Since the fibers
$\mP_\C/K$ are contractible, the reduction chosen here extends to
$\sig:\mC \to \mP_\C/K$, while agreeing with $\sig|_{\mN(w)}$ chosen
in Step 1. As in Lemma \ref{lem:stdform}, there is a trivialization
$\tau_2$ of the bundle $\pi^*\mP \to \Om'$ such that
\eqref{part:nodea} and \eqref{part:nodeb} are satisfied. The reduction 
$\sig$ and trivialization $\tau_2$ are fiber-wise smooth and vary smoothly with $s$. Therefore the convergence of the gauged maps $(A_s,u_s)$ as in 
\eqref{eq:convtau0} holds on the regions $\Om'_s \bs \on{domain}(\rho_{w^\pm})$, i.e. away
from the neighborhood of the node $w$.

We now compare $\tau_1$ and $\tau_2$ to finish the proof. Suppose
$\tau_1$ and $\tau_2$ are related by the gauge transformation
$k_\pm:\Om' \cap \on{domain}(\rho_{w^\pm}) \to K$. We focus on
the $+$ side. For all $s \in
S$ and any $R \geq 0$ let $\gamma_{s,R}$ denote the loop
$\rho_{w^+}^{-1}\{r=R, 0 \leq \theta \leq
2\pi\}$. The homotopy class of $k_+|_{\gamma_{s,R}}$ is constant for
all $s \in S$ and $R \geq 0$. Suppose this class contains the geodesic
loop $[0,2\pi] \ni \theta \mapsto e^{\lambda^+ \theta}$ for some
$\lambda^+ \in \frac 1 {2\pi}\exp^{-1}(\Id) \subset \k$. Then, as in the proof of
Corollary \ref{cor:connW1psig}, $\tau_2$ can be homotopically altered so that
$k_+=e^{\lambda^+ \theta}\circ \rho_{w^+}$. Similarly, we ensure $k_-=e^{\lambda^- \theta}\circ \rho_{w^-}$ for some $\lambda^- \in \frac 1 {2\pi}\exp^{-1}(\Id) \subset \k$.
Then, we see that the
convergence in \eqref{eq:convtau1} holds in the trivialization
$\tau_2$ also, therefore \eqref{part:nodec} is satisfied. Finally, under the trivialization $\tau_2$,
$\kappa_s=e^{\lambda^+\theta}e^{-\lambda^-(\theta+t_s)} \circ \rho_{w^+} \circ \pi_{s,+}^{-1}:\Delta_s \to K$ (see
\eqref{eq:lsts} for definition of $t_s$), which proves
\eqref{part:noded}.
\end{proof}
The following Lemma produces a family of gauged holomorphic maps over
$(S,0)$, such that the central element, defined on a nodal curve, is a
vortex. The family is continuous at $0 \in S$ in the sense of
\eqref{eq:convnearnode_prepare}.
\begin{lemma}[Making the central element a
  vortex]\label{lem:centralvortex} Let $p\geq 2$ and $(A_s,u_s)_{s \in
    S}$ be the family of gauged holomorphic maps defined over the
  principal $K$-bundle $\mP \to \mC$ produced by Lemma
  \ref{lem:stdform_node}. %Denote $P_s:=\mP_s|_{\Sig_s}$. 
  There is a family of
  complex gauge transformations $e^{i\xi_s} \in \G(\mP_s)$
  such that $e^{i\xi_0}(A_0,u_0)$ is a finite energy vortex on $P_0 \to
  \Sig_0$ and for any sequence $s_\nu \to 0$ in
  $S$,
  \begin{equation}\label{eq:convnearnode_prepare}
    \begin{split}
      &\cup_{s_\nu} u_{s_\nu}(\Om'_{s_\nu}) \subset X \text{ is pre-compact},\\
      &\Mod{e^{i\xi_{s_\nu}}A_{s_\nu} - e^{i\xi_{s_0}}A_0}_{W^{1,p}(\Om'_{s_\nu})}
      \to 0 \quad
      d_{W^{2,p}(\Om'_{s_\nu})}(e^{i\xi_{s_\nu}}u_{s_\nu},e^{i\xi_{s_0}}u_0)\to 0,
    \end{split}
  \end{equation}   
  where the norms are defined under the trivialization of the bundle
  $\mP \to \Om'$ from Lemma \ref{lem:stdform_node}.
\end{lemma}
\begin{proof}
  Using Theorem \ref{thm:HK}, there is a complex gauge transformation
  $e^{i\xi_0}$ on $\mP_0 \to \Sig_0$ that makes $(A_0,u_0)$ a
  vortex. We remark that in the analogous stage of proof of the smooth
  case (see Proposition \ref{prop:contsmooth}), we just applied to
  $e^{i\xi_0}$ to all elements of the family -- that is not possible in
  the nodal case. By Lemma \ref{lem:xiexpbd}, $\xi_0:\Sig_0 \to \k$ decays
  exponentially on the cylindrical ends
  corresponding to $z_1,\dots,z_n$ and $w^\pm$: for any $\gamma_1 <\gamma$ there is a constant $c$ such that
\begin{equation}\label{eq:xi0bd}
    \Mod{\xi_0}_{W^{2,p}(\rho_z^{-1}(\{n \leq r \leq n+1\}))} \leq ce^{-\gamma_1 n}.
  \end{equation}
  for $z=z_1,\dots,z_n$, $w^\pm$.  The element $\xi_s \in
  \Gamma(\Sig_s,P_s(\k))$ is defined as follows: Let $\tilde \xi_s=\phi_s (\xi_0|_{\Om_s'}):\Om'_s \to \k$ and $\xi_s:=(\pi_s)_*\tilde \xi_s$.
  For the estimates in \eqref{eq:convnearnode_prepare}, we work in the
  trivialization of $\pi_s^*P_s \to \tSig_s$ from Lemma \ref{lem:stdform_node}. So, we write down the
  pull-back $\pi_s^*\xi_s$ on $\Om_s' \subset \tSig_s$ as
  \begin{equation*}
    \pi_s^* \xi_s=\begin{cases} \tilde \xi_s& \text{ on $\Om_s$,}\\
      \tilde \xi_s + \Ad_{\kappa_s}(r^*\tilde \xi_s)& \text{ on $\Delta_s^+$,}\\
      \tilde \xi_s + \Ad_{\kappa_s}^{-1}((r^{-1})^*\tilde \xi_s)& \text{ on $\Delta_s^-$.}
    \end{cases}
  \end{equation*}
  On $\Om_s$, $\xi_s=\xi_0$, therefore the gauged maps
  $e^{i\xi_s}(A_s,u_s)$ converge in the sense of
  \eqref{eq:convnearnode_prepare} by the convergence of $(A_s,u_s)$ as
  in \eqref{eq:convtau0} from Lemma \ref{lem:stdform_node}. Next, we
  look at the regions $\Delta_s^\pm$.  We know $\kappa_s$ is uniformly bounded in $C^2(\Delta_s,K)$. By the exponential decay \eqref{eq:xi0bd} of
  $\xi_0$, both the terms $\xi_s$ and $\Ad_{\kappa_s}(r^*\tilde \xi_s)$ (resp. $\Ad_{\kappa_s}^{-1}((r^{-1})^*\tilde \xi_s)$) decay to zero in $W^{2,p}(\Delta_s^+)$ (resp. $W^{2,p}(\Delta_s^-)$) as $s \to 0$. So, the convergence of $(A_s,u_s)$ in \eqref{eq:convtau0} implies  \eqref{eq:convnearnode_prepare}.
\end{proof}

The family of $K$-gauged maps that is produced by Lemma
\ref{lem:centralvortex} can be transformed into a family of vortices
via small complex gauge transformations.  This is carried out in the
following Lemma. The control over the size of the complex gauge
transformations  is used in the proof of Proposition \ref{prop:nodal} to show that the family of vortices is continuous
in the sense of Gromov topology.
\begin{lemma} [Gauged maps near a nodal vortex can be complex gauge
  transformed to vortices] \label{lem:vortcontnodal} Let $\mC \to S$
  be a family of curves as in Proposition \ref{prop:nodal} and let
  $(A_s,u_s)_{s \in S}$ be a smooth family of gauged holomorphic maps
  defined on the $K$-bundle $\mP \to \mC$ in which the central element
  $(A_0,u_0)$ is a vortex. Further, there is a trivialization of the
  pullback bundle $\pi^*\mP \to \Om'$ (see Notation
  \ref{note:bundleglue}) such that the corresponding transition
  functions have a $C^2$ bound: $\Mod{\kappa_s}_{C^2(\Delta_s)}<c$ and
  such that the following is satisfied. For any sequence
  $\{s_\nu\}_{\nu \in \Z_{\geq 1}}$ converging to $0$ in $S$, let
  $\Sig_\nu:=\Sig_{s_\nu}$, $\Sig:=\Sig_0$ be Riemann surfaces with
  cylindrical ends corresponding to the curves $\mC_{s_\nu}$ and the
  central curve $\mC_0$ respectively. Further, 
  \begin{equation}\label{eq:convnearnode}
    \begin{split}
      &\cup_\nu u_\nu(\Om'_\nu) \subset X \text{ is pre-compact,}\\
      &\Mod{A_\nu - A_\infty}_{H^1(\Om'_\nu)} \to 0, \quad
      d_{H^2(\Om'_\nu)}(u_\nu,u_0) \to 0
    \end{split}
  \end{equation}
\marginpar{*****}
  as $\nu \to \infty$.  If the sleeve length $\Delta$ (see Notation \ref{note:bundleglue}) is large enough, we can find complex gauge transformations
  $e^{i\xi_\nu}$ on the bundles $\mP_\nu \to \Sig_\nu$ so that
  $\Mod{\xi_\nu}_{H^2(\Sig_\nu,P_\nu(\k))}\to 0$ and
  $e^{i\xi_\nu}v_\nu$ is a vortex on $\Sig_\nu$.
\end{lemma}
\begin{proof} The proof is similar to the proof of Lemma
  \ref{lem:vortcont}. The proofs of invertibility of 
$$\D\F_{v_\nu}(0):H^2(\Sig_\nu,P_\nu(\k)) \to L^2(\Sig_\nu,P_\nu(\k)).$$
  and the existence of a uniform bound on the inverse are more complicated and are carried out in steps 1A and 1B.

  {\sc Step 1A}: {\em There is an approximate inverse $\hat Q_\nu$ of $\D\F_\nu(0)$ satisfying $\Mod{\hat Q_\nu} \leq C$ for all $\nu$.}\\
By arguments in Lemma \ref{lem:vortcont}, the operator $D\F_\nu(0)$ is invertible for all $\nu$, including $\nu=\infty$. \marginpar{*****} The challenge lies in proving a uniform bound on the inverse.
The operators $D\F_\nu(0)$ and $D\F_\infty(0)$ are close to each other, so the inverse of $D\F_\infty$ is restricted to a cover of $\Sig_\nu$, which is then patched to produce an approximate inverse of $D\F_\nu$, which can be bounded uniformly.
Suppose $Q_\infty$ is the inverse of the operator
  $$\D\F_{v_\infty}(0):H^2(\Sig,\k) \to L^2(\Sig,\k).$$
Given $\eta \in L^2(\Sig_\nu,P_\nu(\k))$, it lifts to $\pi_\nu^*\eta \in L^2(\Om'_\nu,\k)$. Recall there is an embedding $\Om_\nu' \subset \Sig$. The section $\pi_\nu^*\eta$ can be extended by zero to get $\eta_\infty \in L^2(\Sig,\k)$. Now, we apply the inverse on the central fiber to get $\zeta_\infty:=Q_\infty \eta_\infty$. Using a cut-off function, we next produce a section on $\Om_\nu'$, namely $\tilde \zeta :=\phi_\nu \zeta_\infty|_{\Om_\nu'} \in H^2(\Om_\nu',\k)$, which is supported away from the boundary $\partial \Om_\nu'$. Finally, we define the approximate inverse of $\eta$ as the push-forward $\zeta=\hat Q_\nu \eta:=(\pi_\nu)_*\tilde \zeta$.

Next we show that $\Mod{\hat Q_\nu}$ is uniformly bounded. It is
easy to see that the first 3 steps -- $\eta \mapsto \pi_\nu^*\eta$,
$\pi_\nu^*\eta \mapsto \zeta_\infty$ and $\zeta_\infty \mapsto \tilde
\zeta$ have norm bounds independent of $\nu$. For the last step
$\tilde \zeta \mapsto \zeta:H^2(\tSig,\k) \to
H^2(\Sig_\nu,P_\nu(\k))$, we need to bound $\Mod{\pi_\nu^*
  \zeta}_{H^2(\Om'_\nu,\k)}$ in terms of $\tilde \zeta$. We can write
$$\pi_\nu^* \zeta=\begin{cases} \tilde \zeta& \text{ on $\Om_\nu$,}\\
  \tilde \zeta + k_\nu(r^*\tilde \zeta)& \text{ on $\Delta_\nu^+$,}\\
  \tilde \zeta + k_\nu^{-1}((r^{-1})^*\tilde \zeta)& \text{ on
    $\Delta_\nu^-$.}
\end{cases}$$ Since we have uniform $C^2$ bounds on $k_\nu$, the last
step $\tilde \zeta \mapsto \zeta$ also has a uniform bound.

{\sc Step 1B}: {\em $\Mod{\D\F_\nu(0)\hat Q_\nu - \Id} \leq \hh$ for large enough $\nu$.}\\
We work on the cover $\Om_\nu'$ of $\Sig_\nu$ and first split up the
required bound to bounds on the core $\Om_\nu$ and the sleeves
$\Delta_\nu^\pm$.  Consider any $\eta \in L^2(\Sig_\nu,P_\nu(\k))$. We observe that $\D\F_\nu(0)\hat Q_\nu\eta=\D\F_\nu(0)(\pi_\nu)_*\tilde \zeta=(\pi_\nu)_*(\D\F_\nu(0)\tilde \zeta)$, and $\eta=(\pi_\nu)_*(\phi_\nu\pi_\nu^* \eta)$. Using \eqref{eq:pilu}, we get
\begin{align*}
  \Mod{\D\F_\nu(0)\hat Q_\nu\eta - \eta}_{L^2(\Sig_\nu,P_\nu(\k))} &\leq 2(\Mod{\D\F_\nu(0)\hat Q_\nu\eta - \eta}_{L^2(\Om_\nu,\k)} + \Mod{\D\F_\nu(0)\tilde \zeta - \phi_\nu\pi_\nu^* \eta}_{L^2(\Delta^+_\nu,\k)} \\
&+ \Mod{\D\F_\nu(0)\tilde \zeta -\phi_\nu\pi_\nu^*\eta}_{L^2(\Delta_\nu^-,\k)})=2(T1+T2+T3).
\end{align*}

To bound $T1$, on $\Om_\nu$, we can write $\eta=\D\F_\infty (\hat Q_\nu\eta)|_{\Om_\nu}$. Then,
$$T1 \leq \Mod{(\D\F_\nu(0) - \D\F_\infty(0))|_{\Om_\nu}}\cdot\Mod{\hat Q_\nu}\cdot \Mod{\eta}_{L^2(\Sig_\nu)} \leq c_\nu\Mod{\hat Q_\nu}\Mod{\eta}_{L^2(\Sig_\nu)},$$
where $c_\nu \to 0$ as $\nu \to \infty$. This uses \eqref{eq:diffdiff}
and the convergence \eqref{eq:convnearnode}.  $T2$ and $T3$ are
bounded in a similar way to each other. We show the case of $T2$:
\begin{equation}\label{eq:Qbd2}
  \begin{split}
    \Mod{\D\F_\nu(0)\tilde \zeta - \phi_\nu\pi_\nu^* \eta}_{L^2(\Delta^+_\nu,\k)} &\leq \Mod{\D\F_\nu(0)\tilde \zeta - \D\F_\infty(0)\tilde \zeta}_{L^2(\Delta^+_\nu,\k)}\\
    & + \Mod{\D\F_\infty(0)\tilde \zeta - \phi_\nu\pi_\nu^*
      \eta}_{L^2(\Delta^+_\nu,\k)} = T2A + T2B.
  \end{split}
\end{equation}
$T2A$ is bounded in a similar way to $T1$ and gives $T2A \leq
c_\nu\Mod{\hat Q_\nu}\Mod{\eta}_{L^2(\Sig_\nu)}$, where $\lim_{\nu \to
  \infty}c_\nu=0$. To work on $T2B$, we observe that on
$\Delta_\nu^+$, $\pi_\nu^*\eta=\D\F_\infty(0)\zeta_\infty$ and so,
\begin{equation*}
  T2B=\Mod{\D\F_\infty(0) \phi_\nu \zeta_\infty - \phi_\nu \D\F_\infty(0)\zeta_\infty}_{L^2}=\Mod{\d^*_{A_\infty}\d_{A_\infty}(\phi_\nu \zeta_\infty) - \phi_\nu \d^*_{A_\infty}\d_{A_\infty}\zeta_\infty}_{L^2}
\end{equation*}
Recall $A_\infty=\d+\lambda^+ \d\theta + a^+$, where $\lambda^+ \in \k$
and $\Mod{a^+}_{W^{1,p}(N(w^+))}<\infty$. Then,
\begin{equation*}
  \begin{split}
    \d^*_{A_\infty}\d_{A_\infty}(\phi_\nu \zeta_\infty) &- \phi_\nu \d^*_{A_\infty}\d_{A_\infty}\zeta_\infty\\
    &=(\d^*\d \phi_\nu)\zeta_\infty + 2(\nabla \phi_\nu)\cdot (\nabla
    \zeta_\infty) + *[(\lambda^+\d\theta + a^+) \wedge (\d\phi_\nu
    \tensor \zeta_\infty)]
  \end{split}
\end{equation*}
The bounds on $a^+$, $\lambda^+$ are independent of $\eta$, $\nu$. So,
\begin{equation*}
  T2B \leq c\Mod{\zeta_\infty}_{H^2(\Delta_\nu^+,\k)}(\Mod{\nabla \phi}_{C^0} + \Mod{\d^*\d \phi}_{C^0}) \leq c_2\Mod{\eta}_{L^2(\Sig_\nu)}(\Mod{\nabla \phi}_{C^0} + \Mod{\d^*\d \phi}_{C^0}).
\end{equation*}
where the norms of $\nabla \phi$ and $\d^*\d \phi$ are taken on
$[-\Delta/2,\Delta/2] \times S^1$. These norms can be made arbitrarily small by
enlarging $\Delta$, which has the effect of stretching out $\phi$. We fix $\Delta$ such
that
$$\Mod{\nabla \phi}_{C^0([-\Delta/2,\Delta/2] \times S^1)} + \Mod{\d^*\d \phi}_{C^0([-\Delta/2,\Delta/2] \times S^1)}) \leq \frac 1 {8c_2}.$$
Putting things together in the above discussion, we get
$$\Mod{\D\F_\nu \hat Q_\nu \eta -\eta}_{L^2(\Sig_\nu)} \leq (\qq + c_\nu)\Mod{\eta}_{L^2(\Sig_\nu)},$$
where the constants $c_\nu$ are such that $\lim_{\nu \to
  \infty}c_\nu=0$. So, by taking $\nu$ large enough, Step 1B is
proved.

Step 1B shows that $\D\F_\nu(0) \hat Q_\nu$ is invertible, so
$\D\F_\nu(0)^{-1}= \hat Q_\nu (\D\F_\nu(0) \hat Q_\nu)^{-1}$ and hence
$\Mod{\D\F_\nu(0)^{-1}} \leq 2\Mod{\hat Q_\nu}$. The rest of the proof
-- steps 2 and 3 -- are identical to the proof of Lemma
\ref{lem:vortcont}.
\end{proof}

\begin{proof}[Proof of Proposition \ref{prop:nodal}] The proof follows
  by applying Lemma \ref{lem:stdform_node}, followed by Lemma
  \ref{lem:centralvortex}, and then followed by Lemma
  \ref{lem:vortcontnodal} to the family of quasimaps $Q_S$.
\end{proof}

\subsection{Non-compact GIT quotient}\label{sec:noncpt}
We recall from Section \ref{sec:quasi}, there is a projective morphism
from the GIT quotient $X\qu G$ to the affine variety $X/_{\on{aff}}
G=\Spec(A^G)$. Also, there is a $G$-invariant map $X \to
X/_{\on{aff}}G$. So, a gauged map $(C,P,u:P \to X)$ descends to a map
$C \to X/_{\on{aff}}G$ which must be constant since $C$ is
affine. So, there is a map $Qmap_{g,n}(X \qu G,
\beta) \to X/_{\on{aff}}G$, which is proper by Theorem 4.3.1. in \cite{CKM:quasimap}.

A similar fact is true for vortices with target $X$. A finite energy
vortex $(A,u)$ descends to a holomorphic map $C \to X/_{\on{aff}}G$
which is necessarily a constant, so the image of $u$ is contained in a
single fiber of the map $X \qu G \to X/_{\on{aff}} G$. In this case,
Theorem \ref{thm:maincpt} has to be modified,
$\ol{MV}^K_{g,n}(X,\beta)$ is no longer compact. Instead, we get a
continuous proper map $\pi_{\on{aff}}:\ol{MV}^K_{g,n}(X,\beta) \to
X/_{\on{aff}}G$. The homeomorphism in Theorem \ref{thm:main} still
holds. The proofs carry over with minor modifications.

\appendix
\section{Some analytic results}

\begin{proposition}[Implicit Function Theorem, Proposition A.3.4 in \cite{MS}] \label{prop:impfn}
   Let $F: X \to Y$ be a
  differentiable map between Banach spaces. The operator $DF(0)$ has an inverse
  $Q$, with $\Mod{Q} \leq c$. Suppose for all $x \in B_\delta$,
  $\Mod{DF(x)-DF(0)}<\frac 1 {2c}$ and $\Mod{F(0)}<\frac \delta {4c}$.
  Then, there is a unique point $x \in B_\delta$ for which $F(x)=0$. 
\end{proposition} 

\begin{proposition}[Sobolev multiplication]\label{prop:sobmult} Let $\Om \subseteq \R^n$ be a domain that is
  not necessarily compact and has smooth boundary. Suppose $\Om$ satisfies the cone condition with a cone $C$ (refer to \cite{Adams:sobspace} for definition).
\begin{enumerate}
\item {\rm{(Theorem 4.39 in \cite{Adams:sobspace})}} Suppose $p>1$ and $k \geq 0$ is an integer such that $kp
  \geq n$. Then there is a constant $c(k,p,n,C)$ such that 
  $$\Mod{uv} \leq c \Mod{u} \cdot \Mod{v},$$
  where $\Mod{\cdot}=\Mod{\cdot}_{W^{k,p}(\Om)}$. So, $W^{m,p}(\Om)$ is
  a Banach algebra. 
\item \label{part:Hk_mult} {\rm{(Multiplication in $H^k$)}} Suppose $k_1$, $k_2$ and $k_3 \in \Z$ are such that $k_3 \leq \min\{k_1,k_2\}$, $k_3<k_1+k_2 - \frac n 2$ and $k_1 + k_2>0$, then there is a constant $c(k_1,k_2,k_3,n,C)$, such that $\Mod{uv}_{H^{k_3}(\Om)}\leq c \Mod{u}_{H^{k_1}(\Om)}\Mod{v}_{H^{k_2}(\Om)}$.
\end{enumerate}
\end{proposition}
\begin{proof} [Proof of \eqref{part:Hk_mult}] First, we focus on $k_3=0$ and assume $k_1 \geq k_2$. If $k_1 -\frac n 2>0$, the result follows from the embedding $H^{k_1} \hra L^\infty$. Otherwise, let $\eps:=k_1+k_2 - \frac n 2>0$ and pick $p_i$ for $i=1,2$ such that $\frac n {p_i}=-k_i + \frac n 2+ \frac \eps 2$. We can embed $H^{k_i} \hra H^{k_i-\frac \eps 2} \hra L^{p_i}$. We have $\frac 1 {p_1} + \frac 1 {p_2}=\hh$, so the result for $k_3=0$ follows by H\"older's inequality. For $k_3>0$, the result can be obtained by induction. For $k_3<0$, recall that $H^{k_3}$ is the dual space $(H_0^{-k_3})^*$, so to prove the result, we need to prove the triple multiplication $H^{k_1} \times H^{k_2} \times H^{-k_3} \to L^1$ is continuous. The constants in Sobolev embedding depend only on the Sobolev indices and the cone $C$.
\end{proof}

\begin{remark}\label{rem:multcyl} The multiplication theorem holds for a surface $\Sig$ with cylindrical ends. In general the constants $c$ are not independent of $\Sig$. However, if $\Sig_1 \subset \Sig_2 \subset \dots \subset \Sig$ be a sequence of subsets exhausting $\Sig$ such that the boundaries $\partial \Sig_i$ are smooth, lie in the cylindrical part of $\Sig$ and are $\R$-translates of each other, then the constants $c(\Sig_i)$ can be chosen to be independent of $i$. This is because $\Sig$ can be covered by a finite number of Euclidean charts, we need two charts for every cylindrical end. The charts used for the Euclidean end are isometric. On these charts, we can find a cone $C$ such that the cone condition is satisfied with the same $C$ for all $\Sig_i$. 
\end{remark}
\begin{lemma}[Norm bound for $\G_\C$ action on $\A$]\label{lem:gcactona} Let $\Sig$ be a Riemann surface with cylindrical ends. Let $k$,
  $p\geq 0$ be such that $(k+1)p>2$. Let $P:=\Sig \times K$ be a trivial $K$-bundle on $\Sig$. Complex gauge transformations in
  $\G^{k+1,p}(P)$ act smoothly on the space of connections $\A^{k,p}(P)$.

  Let $A_0 \in \A^{k,p}$ be a connection on $P$. For
  any $\eps>0$, there is a constant $C$ so that the following is
  satisfied. For any $W^{k,p}$ connection $A=A_0+a$ which satisfies
  $\Mod{a}_{W^{k,p}(\Sig)}<\eps$ and any $\xi \in W^{k+1,p}(\Sig, \k)$
  that satisfies $\Mod{\xi}_{W^{k+1,p}}<1$, 
\begin{equation}\label{eq:actbd}
  \Mod{(\exp i\xi)A - A}_{W^{k,p}(\Sig)} \leq C\Mod{\xi}_{W^{k+1,p}(\Sig)}.
\end{equation}
Suppose $\Sig_1 \subset \Sig_2 \subset \dots \subset \Sig$ be a
sequence of subsets exhausting $\Sig$, whose boundaries $\partial
\Sig_i$ are smooth, lie in the cylindrical region of $\Sig$ and are
$\R$-translates of each other. Then the constant $C$ can be picked so
that it satisfies \eqref{eq:actbd} for all $\Sig_i$.
\end{lemma}
Lemma 6.4 in \cite{VW:affine} is a version of the above Lemma when $\Sig$ is compact. The same proof carries over in this case. The uniform constants for the sequence $\Sig_i$ can be obtained using Remark \ref{rem:multcyl}.
The following Lemma says that on a trivial principal bundle, a
$W^{k,p}$-small connection can be transformed to a flat connection via a
$W^{k+1,p}$-small complex gauge transformation.
\begin{lemma}
  \label{lem:toflat} {\rm (Lemma 4.3, Remark 4.4 in \cite{VW:affine})} Let  $k \in \Z_{\geq 0}$ and $p>1$ be such that $(k+1)p>2$ and let $\Sig$ be a compact connected Riemann surface with metric with non-empty boundary. Let $P:=\Sig \times K$ be the trivial principal $K$-bundle
  on $\Sig$. There are
  constants $c_1$, $c_2$ and $c_2'$ so that the following holds. Let
  $A=\d +a$ be a connection on $P$ so that $a \in
  \Om^1(\Sig,\k)_{W^{k,p}}$. If $\Mod{a}_{W^{k,p}(\Sig)}<c_1$, there is a
  unique $\xi \in W^{k+1,p}(\Sig,\k)$ satisfying $\xi|_{\partial \Sig}=0$,
  $F_{e^{i\xi}A}=0$ and $\Mod{\xi}_{W^{k+1,p}} \leq
  c_2\Mod{F_A}_{W^{k-1,p}} \leq c_2'\Mod{a}_{W^{k,p}}$.
\end{lemma}
%
%
%
% Suppose $F$ is a differential operator whose Sobolev completion is compact for compact domains. As an example, $F$ could be the inclusion $W^{s_1,p}(\Sig) \to W^{s_2,p}(\Sig)$, where $s_2<s_1$ and $\Sig$ is compact. It is not true that the operator $F$ would be compact on non-compact domains. But, if in addition, we know that $F$ decays to zero at the non-compact ends, then the following result says that $F$ is compact. 
\begin{proposition}\label{prop:deccomp} Suppose $\Sig$ is a non-compact manifold that is exhausted by a sequence of compact manifolds
$$\Sig_1 \subset \Sig_2 \subset \dots, \quad \Sig=\cup_i\Sig_i.$$
Let $s_1$, $s_2 \in \Z_{\geq 0}$ and $p_1$, $p_2>0$ and let $F:W^{s_1,p_1}(\Sig) \to W^{s_2,p_2}(\Sig)$ be a differential operator that satisfies the following. The restriction $F|_{\Sig_i}:W^{s_1,p_1}(\Sig_i) \to W^{s_2,p_2}(\Sig_i)$ is a compact operator for all $i$. The restriction $F|_{\Sig \bs \Sig_i}$ has bounded norm for all $i$ and the operator norm $\Mod{F|_{\Sig \bs \Sig_i}} \to 0$ as $i\to \infty$. Then, the operator $F$ is compact.
\end{proposition}
The proof is similar to Lemma 2.1 in \cite{choquet}, see also Proposition E.6 (v) in \cite{Zilt:thesis}. 
\begin{theorem}[Uhlenbeck Compactness, Theorem 2.1 in \cite{Uh:compactness}]\label{thm:Uhcpt} Let $K$ be a compact Lie group, let $B \subset \R^d$ be the unit ball and $\frac d 2 < p < d$. There exist constants $\kappa(d,p)$, $c(d,p)$ such that the following is satisfied. Any $W^{1,p}$-connection $A$ on the trivial bundle $B \times K$ that satisfies the curvature bound $\Mod {F_A}_{L^p(B_1)}<\kappa$ is gauge equivalent to a connection $d+a$, such that 
\begin{enumerate}
\item\rm{(Coulomb gauge condition)} $\d^*a=0$, $*a|_{\partial B}=0$. 
\item $\Mod{a}_{W^{1,p}(B_1)} \leq c \Mod{F_A}_{L^p(B_1)}$.
\end{enumerate}
\end{theorem}
%\bibliographystyle{amsplain}
%\bibliography{mybiblio}

\def\dbar{\leavevmode\hbox to 0pt{\hskip.2ex \accent"16\hss}d}
\providecommand{\bysame}{\leavevmode\hbox to3em{\hrulefill}\thinspace}
\providecommand{\MR}{\relax\ifhmode\unskip\space\fi MR }
\providecommand{\MRhref}[2]{%
  \href{http://www.ams.org/mathscinet-getitem?mr=#1}{#2}
}
\providecommand{\href}[2]{#2}

\end{document}